\documentclass[11pt,reqno]{amsart}

\usepackage[T1]{fontenc}							
\usepackage[english]{babel}		
\usepackage{csquotes}

\usepackage{cite}					

\usepackage{amssymb, amsmath, amsthm}
\usepackage{mathtools}				
\usepackage{enumitem}			
\usepackage{bm}							

\usepackage{xcolor}				

\newcommand{\N}{\ensuremath{\mathbb{N}}}

\newcommand{\R}{\ensuremath{\mathbb{R}}}

\newcommand{\rmL}{\mathrm L}
\newcommand{\rmH}{\mathrm H}
\newcommand{\rmW}{\mathrm W}
\newcommand{\rmC}{\mathrm C}
\newcommand{\BC}{\mathrm{BC}}
\newcommand{\BUC}{\mathrm{BUC}}
\newcommand{\BCw}{\mathrm{BC_w}}

\newcommand{\Grad}{\nabla}
\newcommand{\Lap}{\Delta}
\newcommand{\Div}{\mathrm{div}}

\newcommand{\ddt}{\frac{\mathrm d}{\mathrm dt}}

\newcommand{\delt}{\partial_t}
\newcommand{\delxj}{\partial_{x_j}}
\newcommand{\D}{\mathrm D}

\newcommand{\dx}{\;\mathrm dx}
\newcommand{\dsix}{\;\mathrm ds_x}
\newcommand{\dy}{\;\mathrm dy}
\newcommand{\dt}{\;\mathrm dt}

\newcommand{\dtx}{\;\mathrm d(t,x)}

\newcommand{\dtaux}{\;\mathrm d(\tau,x)}
\newcommand{\dtsix}{\;\mathrm d(t,s_x)}
\newcommand{\dtausix}{\;\mathrm d(\tau,s_x)}

\newcommand{\n}{\mathbf{n}}										
\newcommand{\btau}{\bm{\tau}}									
\newcommand{\bnu}{\bm{\nu}}										

\newcommand{\loc}{\textnormal{loc}}
\newcommand{\uloc}{\textnormal{uloc}}
\newcommand{\Neum}{\textnormal{N}}						
\newcommand{\tot}{\textnormal{tot}}
\newcommand{\free}{\textnormal{free}}
\newcommand{\kin}{\textnormal{kin}}

\newcommand{\vel}{\mathbf{v}}									
\newcommand{\uel}{\mathbf{u}}									
\newcommand{\w}{\mathbf{w}}										
\newcommand{\J}{\mathbf{J}}
\newcommand{\Str}{\mathbf{S}}									
\newcommand{\T}{\mathbf{T}}
\newcommand{\bpsi}{\bm{\psi}}
\newcommand{\z}{\mathbf{z}}	
\newcommand{\f}{\mathbf{f}}		

\newcommand{\mj}{m_{\textnormal{j}}}		
\newcommand{\mr}{m_{\textnormal{r}}}

\newcommand{\prt}[1]{( #1 )}										
\newcommand{\bigprt}[1]{\big( #1 \big)}
\newcommand{\Bigprt}[1]{\Big( #1 \Big)}
\newcommand{\biggprt}[1]{\bigg( #1 \bigg)}
\newcommand{\abs}[1]{| #1 |}										
\newcommand{\bigabs}[1]{\big| #1 \big|}
\newcommand{\Bigabs}[1]{\Big| #1 \Big|}
\newcommand{\biggabs}[1]{\bigg| #1 \bigg|}
\newcommand{\norm}[1]{\| #1 \|}										
\newcommand{\bignorm}[1]{\big\| #1 \big\|}
\newcommand{\Bignorm}[1]{\Big\| #1 \Big\|}
\newcommand{\ang}[2]{ \langle #1 , #2  \rangle}			

\newcommand{\mean}[1]{ \langle #1 \rangle}				

\newcommand{\ssubset}{\subset\joinrel\subset}		
\newcommand{\trans}[1]{ #1^{\mathrm T} }

\DeclareMathOperator{\dom}{dom}
\DeclareMathOperator{\supp}{supp}

\newtheorem{theorem}{Theorem}[section]
\newtheorem{lemma}[theorem]{Lemma}
\newtheorem{proposition}[theorem]{Proposition}

\theoremstyle{definition}
\newtheorem{definition}[theorem]{Definition}

\newtheorem{remark}[theorem]{Remark}

\numberwithin{equation}{section}

\newcommand{\owntag}[1]{\stepcounter{equation}\tag{\theequation$_{#1}$}}

\usepackage[colorlinks = true, linkcolor = blue, citecolor=green]{hyperref}

\begin{document}
	\title[Diffuse Interface Models for Two-Phase Flows with Phase Transition]{Diffuse Interface Models for Two-Phase Flows with Phase Transition: Modeling and Existence of Weak Solutions}
	\author{Helmut Abels}
	\author{Harald Garcke}	
	\author{Julia Wittmann}
	\address{Fakultät für Mathematik, Universität Regensburg \\ D--93040 Regensburg, Deutschland}
	\email{helmut.abels@ur.de}
	\email{harald.garcke@ur.de}
	\email{julia4.wittmann@ur.de}

	\begin{abstract}
		The flow of two macroscopically immiscible, viscous, incompressible fluids with unmatched densities is studied, where a transfer of mass between the constituents by phase transition is taken into account. To this end, two quasi-incompressible diffuse interface models with singular free energies are analyzed, differing primarily in their velocity averaging. Firstly, to generalize a model by Abels, Garcke, and Grün, a thermodynamically consistent system of Navier--Stokes/Cahn--Hilliard type with source terms is derived in a framework of continuum fluid dynamics, followed by a proof of existence of weak solutions to the latter. Secondly, the quasi-stationary version of a model by Aki, Dreyer, Giesselmann, and Kraus is investigated analytically, with existence of weak solutions being established for the resulting quasi-stationary Stokes system coupled to a Cahn--Hilliard equation with a source term.
		
		
	\end{abstract}

	\maketitle

	\section{Introduction} \label{sec:intro}
	
	Modeling the behavior of two-phase flows plays a crucial role in continuum fluid dynamics, with widespread relevance across both engineering and natural sciences. The complex phenomena occurring in these processes are particularly significant due to their ability to incorporate local mass exchanges between components, often driven by phase transitions, like those arising from chemical reactions.
	
	More precisely, this work focuses on the dynamics of a binary isothermal mixture composed of viscous, incompressible Newtonian fluids with different densities. It takes into account that in these systems, phase transitions may cause mass transfers between the constituents. Therefore, in the derivation of models describing such flows, the individual mass balance equations for each component include non-trivial source terms, representing mass production rates.
	
	A key feature is the consideration of macroscopically immiscible fluids, for which we assume partial miscibility within a thin transition layer of positive width. This framework leads to diffuse interface models for describing the motion of the binary mixture. Typically, such models consist of the incompressible Navier--Stokes equations for an averaged velocity field, which governs the fluid motion, coupled to the Cahn--Hilliard equation to bring phase separation into play.

	\medskip
	\noindent\textbf{Background.}
	The basic diffuse interface model for two-phase flows of incompressible fluids by Hohenberg \& Halperin \cite{Hohenberg1977}, known as ``Model H'', relies heavily on the assumption that both components have identical constant densities (matched densities). Gurtin et al.~\cite{Gurtin1996} derived this model rigorously and verified its thermodynamic consistency by showing a local energy dissipation inequality.
	In the literature, several approaches to generalize this model to the case of unmatched densities are of significance.
	Lowengrub \& Truskinovsky \cite{LowengrubTruskinovsky1998} proposed a thermodynamically consistent model, characterized as quasi-incompressible in the sense that both fluids in the mixture are incompressible but their mass averaged velocity field is not divergence-free.
	In contrast, Ding et al.~\cite{Ding2007} derived a model based on a solenoidal volume averaged velocity, whose thermodynamic consistency, however, is not known. 
	A more complex model by Boyer \cite{Boyer2002} also involves a volume averaged velocity, resulting in a divergence-free velocity field. But, to the best of the authors' knowledge, neither global nor local energy inequalities have been established for this system either.
	Finally, Abels, Garcke~\&~Grün \cite{AbelsGarckeGruen2012} developed a thermodynamically consistent, incompressible model using the same velocity averaging.
	For further diffuse interface models for two-phase flows with unmatched densities, the reader may be referred to \cite{Heida2011, ShokrpourRoudbari2018}. In addition, ten Eikelder et al.~\cite{tenEikelder2023} presented a unified framework for various known Navier--Stokes/Cahn--Hilliard models, along with a comparison of these. 
	
	Concerning models that take phase transitions into account, we mention two works.
	Aki et al.~\cite{Aki2014} proposed a thermodynamically consistent, quasi-incompressible system with mass averaged velocity allowing for mass transfers. 
	Based on mixture theory, ten Eikelder et al.~\cite{tenEikelder2024} derived a multi-phase model with phase transitions which satisfies the second law of thermodynamics, and which contains a law of conservation of linear momentum for each constituent instead of one balance equation for an averaged velocity field. However, their system does not incorporate a Cahn--Hilliard or Allen--Cahn equation.
	
	From an analytical perspective, Model H has been thoroughly investigated in the literature, see \cite{Boyer1999, GiorginiMiranvilleTemam2019, Abels2009m} to name just a few works.
	Likewise, numerous results are available concerning the analysis of the model by Abels, Garcke \& Grün, e.g., existence of global weak solutions \cite{AbelsDepnerGarcke2012, AbelsDepnerGarcke2013}, well-posedness \cite{Weber2016, AbelsWeber2020, Giorgini2021, Giorgini2022} as well as global regularity and asymptotic stabilization of weak solutions \cite{AbelsGarckeGiorgini2023}. We further mention extensions of this model to the case with dynamic boundary conditions \cite{GalGrasselliWu2019, GalLvWu2024, GiorginiKnopf2023}, where a transfer of the total mass between bulk and surface is possible.
	
	In contrast, analyzing the model by Lowengrub \& Truskinovsky \cite{LowengrubTruskinovsky1998} seems more challenging due to the presence of the pressure in the equation for the chemical potential and a resulting stronger coupling between the Navier--Stokes and the Cahn--Hilliard systems. In \cite{Abels2009g}, existence of weak solutions was proven under the condition that the free energy density is regular and sufficiently ``steep'', with the exponent of its gradient term exceeding the spatial dimension. Strong well-posedness of this model locally in time was established in \cite{Abels2012}.
	
	An alternative (but physically less justified) way to describe two-phase flows of fluids undergoing phase transitions is to govern the phase field variable by an Allen--Cahn type equation instead of the Cahn--Hilliard equation. A thermodynamically consistent, compressible Navier--Stokes/Allen--Cahn model was proposed by Blesgen \cite{Blesgen1999}, with local well-posedness shown in \cite{Kotschote2012}. For an early analytical study of a two-dimensional incompressible system with matched densities, see \cite{GalGrasselli2010}. Moreover, Jiang at al.~\cite{JiangLiLiu2017} derived a thermodynamically consistent model for incompressible flows with different densities, and proved global existence of weak solutions in three space dimensions.
	We finally refer to \cite{GiorginiGrasselliWu2022} for results on global well-posedness for a mass-conserving Navier--Stokes/Allen--Cahn system.

	\medskip
	\noindent\textbf{The two models of interest.}
	Given a sufficiently smooth bounded domain $\Omega\subset\R^d$, where $d\in\{2,3\}$, with exterior unit normal vector field $\n$ on $\partial\Omega$, we write $Q=(0,\infty)\times\Omega$ for the space-time cylinder and $S=(0,\infty)\times\partial\Omega$ for its lateral surface. 
	
	\textit{Model I:} To generalize the model by Abels, Garcke, and Grün \cite{AbelsGarckeGruen2012} to the case including phase transitions, we derive the following quasi-incompressible, instationary system of inhomogeneous Navier--Stokes/Cahn--Hilliard type:
	\begin{equation}
		\label{eq:AGG}
		\arraycolsep=2pt
		\begin{array}{rclll}
			\partial_t (\rho\vel) + \nabla\cdot(\rho\vel\otimes\vel) + \nabla\cdot(\vel\otimes\tilde\J) \\[0.2ex]
			- \nabla\cdot \Str(\varphi,\D\vel) + \Grad\lambda &=& -\varphi\Grad\mu 
			&\text{ in $Q$}, \\[1ex]
			\nabla\cdot\vel &=& - \alpha \mr(\varphi) \bigprt{\mu+\alpha\lambda} 
			&\text{ in $Q$}, \\[1ex]
			\partial_t \varphi + \nabla\cdot (\varphi\vel) &=& \nabla\cdot \bigprt{\mj(\varphi) \Grad\mu} - \mr(\varphi) \bigprt{\mu+\alpha\lambda}
			&\text{ in $Q$}, \\[1ex]
			\mu &=& F'(\varphi) - \Lap\varphi 
			&\text{ in $Q$}, \\[1ex]
			\vel \vert_{\partial\Omega} &=& 0
			&\text{ on $S$}, \\[1ex]
			\partial_\n\varphi \vert_{\partial\Omega} = \partial_\n\mu \vert_{\partial\Omega} &=& 0
			&\text{ on $S$}, \\[1ex]
			(\vel,\varphi) \vert_{t=0} &=& (\vel_0,\varphi_0) 
			&\text{ in $\Omega$}.
		\end{array}
	\end{equation}
	Here, $\vel$ is the \textit{volume averaged} velocity of the mixture, while the phase field variable $\varphi$ acts as an order parameter given as difference of the volume fractions of both fluids, such that the pure constituents correspond to $\varphi=\pm1$. Additionally, $\lambda$ denotes the pressure and $\mu$ the chemical potential associated to $\varphi$. The homogeneous free energy density $F$ is a double-well potential specified below, with minima close to $\pm1$. Moreover, the viscous stress tensor $\Str$, the mass density $\rho$, and the total mass flux $\tilde{\J}$ relative to the velocity field are determined by the aforementioned variables via the relations
	\begin{equation}
		\label{eq:AGG_relations}
		\arraycolsep=2pt
		\begin{array}{rcl}
			\Str(\varphi, \D\vel) &=& 2\nu(\varphi) \D\vel + \eta(\varphi) \nabla\cdot\vel \mathbf I,\\[1ex]
			\rho(\varphi) &=& b_+ + b_- \varphi,\\[1ex]
			\tilde\J(\varphi,\Grad\mu) &=& -b_- \mj (\varphi) \Grad\mu,
		\end{array}
	\end{equation}
	where $\D\vel$ denotes the symmetrized gradient of $\vel$ and $\mathbf I$ the identity matrix in $\R^d$. Eventually,~$\mj(\varphi)>0$ represents a (non-degenerate) diffusion mobility and $\mr(\varphi)>0$ a transition mobility, whereas $\nu(\varphi)>0$ as well as~$\eta(\varphi)>0$ are viscosity coefficients.
	For constants related to the specific mass density~$\tilde\rho_\pm>0$ of the respective fluid, we use the abbreviations
	\begin{align*}
			b_\pm = \tfrac{\tilde\rho_+ \pm \tilde\rho_-}{2}, \qquad
			c_\pm = \tfrac{1}{\tilde\rho_+} \pm \tfrac{1}{\tilde\rho_-}, \qquad
			\alpha = \tfrac{c_-}{c_+}, \qquad
			\text{where } \tilde\rho_+ \neq \tilde\rho_-.
	\end{align*}
	Concerning the boundary conditions in \eqref{eq:AGG}, we prescribe a no-slip boundary condition for~$\vel$. Furthermore, the homogeneous Neumann boundary condition for $\varphi$ enforces a contact angle of $\pi/2$ between the diffuse interface and the domain boundary, while the corresponding condition for $\mu$ ensures the absence of mass flux over the boundary.
	
	The total energy of the system is given as sum of the kinetic energy and a free energy of Ginzburg--Landau type, i.e.,
	\begin{align*}
		E_\tot\prt{\vel,\varphi}
		= E_\kin(\vel) + E_\free(\varphi) 
		= \int_\Omega \rho \frac{\abs \vel^2}{2} \dx + \int_\Omega F(\varphi) + \frac{\abs{\Grad\varphi}^2}{2} \dx,
	\end{align*}
	where $F$ belongs to a class of singular free energy densities including the well-known logarithmic Flory--Huggins potential. This choice ensures that in the analysis, $\varphi$ stays in the physically meaningful interval $[-1,1]$.
	For any sufficiently smooth solution $(\vel,\lambda,\mu,\varphi)$ of~\eqref{eq:AGG}, the total energy dissipates over time according to the evolution law
	\begin{align*}
		&\ddt E_\tot\bigprt{\vel(t),\varphi(t)}
		= - \int_\Omega \Str(\varphi,\D\vel) : \D\vel + \mj(\varphi) \abs{\Grad\mu}^2 + \mr(\varphi) \bigprt{\mu+\alpha\lambda}^2 \dx.
	\end{align*}

	\textit{Model II:} On the other hand, we study a quasi-stationary version of the diffuse interface model suggested by Aki, Dreyer, Giesselmann, and Kraus \cite{Aki2014}. This quasi-stationary variant is relevant for scenarios involving small Reynolds numbers. From a modeling point of view, it differs from \eqref{eq:AGG} in the averaging of the velocity and in the interpretation of the chemical potential, resulting in the following quasi-incompressible, inhomogeneous quasi-stationary Stokes system coupled to the Cahn–Hilliard equation:
	\begin{equation}
		\label{eq:LT}
		\arraycolsep=2pt
		\begin{array}{rclll}
			- \nabla\cdot \Str(\varphi, \D\uel) + \beta\rho\Grad\lambda &=& -\varphi\Grad \omega
			&\text{ in $Q$}, \\[1ex]
			\partial_t \rho + \nabla\cdot(\rho\uel) &=& 0
			&\text{ in $Q$}, \\[1ex]
			\partial_t \varphi + \nabla\cdot (\varphi\uel) &=& c_+^2 \bigprt{\nabla\cdot \bigprt{\mj(\varphi)\Grad\omega} - \mr(\varphi) \omega}
			&\text{ in $Q$}, \\[1ex]
			\omega &=& F'(\varphi) - \Lap\varphi + \alpha\lambda
			&\text{ in $Q$}, \\[1ex]
			\n\cdot\uel \vert_{\partial\Omega}&=&0 
			&\text{ on $S$}, \\[1ex]
			(\Str\n)_{\btau} &=&-\gamma\uel_{\btau} \vert_{\partial\Omega} 
			&\text{ on $S$}, \\[1ex]
			\partial_\n\varphi \vert_{\partial\Omega} = \partial_\n \omega \vert_{\partial\Omega} &=&0 
			&\text{ on $S$}, \\[1ex]
			\varphi \vert_{t=0} &=& \varphi_0 
			&\text{ in $\Omega$}.
		\end{array}
	\end{equation}
	In contrast to system \eqref{eq:AGG}, $\uel$ is defined as the \textit{mass averaged} velocity of the mixture, and we introduce $\omega$ as a modified chemical potential related to $\mu$ via $\omega = \mu+\alpha\lambda$. We emphasize that the different velocity averaging leads to a stronger coupling, since the pressure enters the equation for the chemical potential. Moreover, the continuity equation is listed in \eqref{eq:LT}, instead of the divergence equation, which can be reconstructed from the mass balance and the phase field equation. As previously, the relations
	\begin{equation}
		\label{eq:LT_relations}
		\arraycolsep=2pt
		\begin{array}{rcl}
			\Str(\varphi, \D\uel) &=& 2\nu(\varphi) \D\uel + \eta(\varphi) \nabla\cdot\uel \mathbf I,\\[1ex]
			\rho(\varphi) &=& b_+ + b_- \varphi
		\end{array}
	\end{equation}
	are satisfied, whereas the total mass flux relative to the velocity field does not appear explicitly in this system. The occurring constants associated with the specific mass densities are denoted by
	\begin{align*}
			b_\pm = \tfrac{\tilde\rho_+ \pm \tilde\rho_-}{2}, \qquad
			c_\pm = \tfrac{1}{\tilde\rho_+} \pm \tfrac{1}{\tilde\rho_-}, \qquad
			\alpha = \tfrac{c_-}{c_+}, \qquad
			\beta = \tfrac{2}{\tilde\rho_+ + \tilde\rho_-}, \qquad
			\text{where } \tilde\rho_+ \neq \tilde\rho_-.
	\end{align*}
	
	Instead of a no-slip boundary condition, we impose a Navier-slip boundary condition here, meaning that only the normal part of the velocity vanishes on the boundary, while the tangential part (denoted by an index $\btau$) is related to the tangential part of the normal stresses through a friction parameter $\gamma>0$.
	
	Since we work with a quasi-stationary model neglecting the kinetic energy, the total energy of the system only consists of the free energy
	\begin{align*}
		E_\free(\varphi)
		= \int_\Omega F(\varphi) + \frac{\abs{\Grad\varphi}^2}{2} \dx,
	\end{align*}
	with a singular homogeneous free energy density $F$ as above. Here, the energy dissipation is governed by the law
	\begin{align*}
		&\ddt E_{\free}\bigprt{\varphi(t)} 
		= - \int_\Omega \Str(\varphi, \D\uel) : \D\uel + \mj(\varphi) c_+^2 \abs{\Grad \omega}^2  
		+ \mr(\varphi) c_+^2 \omega^2 \dx
		- \int_{\partial\Omega} \gamma \abs{\uel_{\btau}}^2 \dsix
	\end{align*}
	for sufficiently smooth solutions $(\uel,\lambda,\omega,\varphi)$ to \eqref{eq:LT}, where $\mathrm ds_x$ indicates integration with respect to the surface measure on $\partial\Omega$. Of course, acceleration terms could also be considered in \eqref{eq:LT}, leading to a Navier--Stokes type momentum equation and to an additional kinetic energy term in the energy dissipation, see \cite{Aki2014}.

	\medskip
	\noindent\textbf{Main results.}
	Our first central contribution is the derivation of the thermodynamically consistent diffuse interface model \eqref{eq:AGG} for two-phase flows with phase transition. It describes the dynamics of an isothermal mixture of viscous, incompressible Newtonian fluids with unmatched densities that may exchange mass.
	
	Analytically, our main results are the following two theorems.
	
	\begin{theorem}
		\label{thm:AGG}
		Let $\vel_0\in \rmL^2(\Omega)^d$ and $\varphi_0\in \rmH^1(\Omega)$ with $\abs{\varphi_0}\leq1$ a.e.~in $\Omega$ and $\mean{\varphi_0} \in (-1,1)$. Under the assumptions \eqref{ass:domain}--\eqref{ass:coeffs} below, there exists a weak solution $(\vel,\lambda,\mu,\varphi)$ of \eqref{eq:AGG} in the sense of Definition \ref{AGG:exws:def:weak_sol}.
	\end{theorem}
	
	\begin{theorem}
		\label{thm:LT}
		Let $\varphi_0\in \rmH^1(\Omega)$ with $\abs{\varphi_0}\leq1$ a.e.~in $\Omega$ and $\mean{\varphi_0} \in (-1,1)$. Under the assumptions \eqref{ass:domain}--\eqref{ass:friction} below, there exists a weak solution $(\uel,\lambda,\omega,\varphi)$ of \eqref{eq:LT} in the sense of Definition~\ref{LT:exws:def:weak_sol}.
	\end{theorem}
	
	Here, $\mean{\cdot}$ denotes the spatial mean of a function, and we refer to Subsection~\ref{subsec:notation} for further detailed notation.
	
	The proof of both theorems is based on an implicit time discretization, where solutions to the time-discrete system are established by a Leray--Schauder fixed-point argument, following the strategy of \cite{AbelsDepnerGarcke2012}. Passing to the limit in the corresponding time-continuous system primarily relies on \textit{a priori} estimates provided by an energy dissipation inequality. Regarding the non-linearities, proving an additional strong convergence of both the velocity and the phase field is required in the case of the instationary model \eqref{eq:AGG}, whereas in the quasi-stationary model \eqref{eq:LT}, only the phase field needs to converge strongly. Handling the pressure $\lambda$, or specifically its mean-free part $\lambda_0$, is more intricate in \eqref{eq:LT}. In order to obtain an \textit{a priori} estimate for $\lambda_0$ in $\rmL^2(\Omega)$, we introduce a damping term within the time-discrete system, similar to the approach in \cite{Abels2009g}. On the time-continuous level, we employ the momentum equation to bound $\lambda_0$ in a suitable space~$\rmL^2(0,\infty;\rmL^r(\Omega))$, where the Naver-slip boundary condition for the velocity comes into play. Finally, the limit passage in $F'(\varphi)$~differs as well: while we apply results for subgradients related to the free energy in \eqref{eq:AGG}, a more technical, measure theoretic argument is used in \eqref{eq:LT}.

	\medskip
	\noindent\textbf{Outline.}
	This work is structured as follows. We start in Section~\ref{sec:preliminaries} by introducing some notation, fixing general assumptions and collecting preliminary results of importance for our analysis. Section~\ref{sec:AGG_modeling} is dedicated to the derivation of the model \eqref{eq:AGG} for two-phase flows with phase transition using a volume averaged velocity. Existence of weak solutions to this system is established in Section~\ref{sec:AGG_exws}. This is followed by Section~\ref{sec:LT_reformulation}, where we reformulate the model with mass averaged velocity proposed by Aki et al.~\cite{Aki2014}. Finally, in Section~\ref{sec:LT_exws}, we prove existence of weak solutions to the quasi-stationary variant of the latter, which coincides with system \eqref{eq:LT}.

	\section{Preliminaries} \label{sec:preliminaries}
	
	\subsection{Notation} \label{subsec:notation} 
	
	We fix some notation which is supposed to hold throughout the entirety of the present work.
	
	\medskip
	\noindent\textbf{Notation for function spaces.}
	Let $U\subset\R^d$, $d\in\N$, be an open set. We adopt the standard notation~$\rmL^p(U)$, $p\in [1,\infty]$, to represent Lebesgue spaces, with their respective norms denoted by $\norm{\cdot}_{\rmL^p(U)}$. Accordingly, $\rmL^p(U;X)$ refers to Bochner spaces with values in a Banach space $X$. In the case $U = (a, b)$, we simplify the notation to $\rmL^p(a,b)$ and~$\rmL^p(a,b;X)$. Moreover, we write $f \in \rmL^p_\loc([0,\infty);X)$ if and only if $f \in \rmL^p(0,T;X)$ for every $T > 0$, whereas~$\rmL^p_\uloc([0, \infty); X)$ denotes the uniformly local variant of $\rmL^p(0, \infty; X)$ consisting of all strongly measurable $f \colon [0, \infty) \to X$ such that
	\begin{align*}
		\norm{f}_{\rmL^p_\uloc([0,\infty);X)} \coloneqq \sup_{t\geq0} \norm{f}_{\rmL^p(t,t+1;X)} < \infty.
	\end{align*}

	For a domain $U\subset\R^d$, we denote Sobolev spaces by $\rmW^{k,p}(U)$, $k\in\N_0$, $p\in [1,\infty]$. Furthermore, $\rmW^{k,p}_0(U)$ represents the closure of $\rmC_0^\infty(U)$ in $\rmW^{k,p}(U)$, and $\rmW^{-k,p}(U) = (\rmW^{k,p'}_0(U))'$ refers to the corresponding dual space, where the dual exponent $p'\in [1,\infty]$ of $p$ is such that it holds~$\frac1p + \frac1{p'}=1$. In addition, we introduce the $\rmL^2$-Bessel potential spaces $\rmH^s(U)$, $s\in\R$. With a Banach space $X$ and $0 < T < \infty$,
	we write $f \in \rmW^{1,p}(0,T;X)$, $p\in[1,\infty)$, if and only if $f,\frac{\mathrm{d} f}{\mathrm{d} t} \in \rmL^p(0,T;X)$, where $\frac{\mathrm{d} f}{\mathrm{d} t}$~refers to the vector-valued distributional derivative of~$f$. Moreover, $\rmW^{1,p}_\uloc([0,\infty);X)$ is defined by replacing $\rmL^p(0,T;X)$ by $\rmL^p_\uloc([0,\infty);X)$. We further set $\rmH^1(0,T;X) = \rmW^{1,2}(0,T;X)$ and~$\rmH^1_\uloc([0,\infty);X) = \rmW^{1,2}_\uloc([0,\infty);X)$. 
	
	Now let $X$ be some Banach space and let $I = [0,T]$ with $0 < T < \infty$ or $I = [0,\infty)$ if~$T = \infty$. 
	Then $\BC(I;X)$ denotes the Banach space of all bounded and continuous functions~$f\colon I \to X$ endowed with the supremum norm, while we write $\BUC(I;X)$ for the subspace of all bounded and uniformly continuous functions. Additionally, we define $\BCw(I;X)$ as topological vector space of all bounded and weakly continuous functions $f\colon I \to X$. We finally denote by $\rmC_0^\infty(0,T;X)$ the vector space of all smooth functions~$f\colon (0,T) \to X$ with~$\supp f \ssubset (0,T)$.

	\medskip
	\noindent\textbf{Further notation.}
	The outer product of $\mathbf{a},\mathbf{b}\in\R^d$ is given by $\mathbf{a}\otimes\mathbf{b}=(a_ib_j)_{i,j=1}^d$.
	
	For any Banach space $X$ and its dual $X'$, the duality pairing between elements $x'\in X'$ and $x\in X$ is denoted by $\ang{x'}{x}_{X',X}$. If $X$ is a Hilbert space, we write $(\cdot,\cdot)_X$ for its inner product on $X$. 
	
	In the following, let $\Omega\subset\R^d$ be a sufficiently smooth bounded domain with exterior unit normal vector field $\n$. 
	The notation $\D\uel \coloneqq \frac12 \bigprt{ \Grad\uel + \trans{\Grad\uel} }$ represents the symmetrized gradient of a function $\uel\in \rmH^1(\Omega)^d$.
	We further write
	\begin{align*}
		\mean{f} \coloneqq \frac{1}{|\Omega|}\ang{f}{1}_{\rmH^1(\Omega)', \rmH^1(\Omega)} \quad\text{for } f\in \rmH^1(\Omega)'
	\end{align*}
	to denote the generalized spatial mean of $f$, where $|\Omega|$ is the $d$-dimensional Lebesgue measure of $\Omega$. With the standard identification $\rmL^2(\Omega) \subset \rmH^1(\Omega)'$, it holds that $\mean{f} = \frac{1}{|\Omega|} \int_\Omega f \dx$ if~$f\in \rmL^2(\Omega)$.
	Given $m\in\R$, we set 
	\begin{align*}
		\rmL^p_{(m)}(\Omega) \coloneqq \{u\in \rmL^p(\Omega) : \mean{u} = m \}, \quad p\in[1,\infty],
	\end{align*}
	as well as $\rmH^k_{(m)}(\Omega) \coloneqq \rmH^k(\Omega) \cap \rmL^2_{(m)}(\Omega)$, $k\in\N$.
	
	Finally, we indicate the space of all functions in $\rmW^{k,p}(\Omega)^d$ with vanishing normal part on the boundary by
	\begin{align*}
		\rmW^{k,p}_\n(\Omega) \coloneqq \{\uel\in \rmW^{k,p}(\Omega)^d : \n \cdot \uel \vert_{\partial\Omega} = 0 \}, 
		\quad k\in\N,\; p\in[1,\infty),
	\end{align*}
	and the space of all functions in $\rmW^{k,p}(\Omega)$ satisfying a homogeneous Neumann boundary condition by
	\begin{align}
		\label{notation:eq:WkpN}
		\rmW^{k,p}_\Neum(\Omega) \coloneqq \{u\in \rmW^{k,p}(\Omega) : \partial_\n u \vert_{\partial\Omega} = 0 \},
		\quad 2\leq k\in\N,\; p\in[1,\infty).
	\end{align}

	\subsection{General Assumptions.} 
	
	We make the following assumptions for the remainder of this section and for the analysis sections of this paper, which are Sections~\ref{sec:AGG_exws} and~\ref{sec:LT_exws}.
	
	\begin{enumerate}[itemsep=0.8ex, label=$(\mathbf{A \arabic*})$, ref = $\mathbf{A \arabic*}$]
		\item \label{ass:domain}
			The set $\Omega\subset\R^d$, $d\in\{2,3\}$, is a bounded domain with $\rmC^2$-boundary with exterior unit normal vector field $\n$ on $\partial\Omega$.
		\item \label{ass:pot}
		The potential $F\colon [-1,1]\to\R$ is of regularity $F\in \rmC([-1,1])\cap \rmC^2((-1,1))$ and satisfies
		\begin{align*}
			\lim_{s\to-1} F'(s)=-\infty,
			\quad \lim_{s\to1} F'(s)=\infty,
			\quad F''(s) \geq -\kappa \text{ for some } \kappa\in\R.
		\end{align*}
		We denote the convex part of $F$ by $F_0(s)\coloneqq F(s)+\frac\kappa2s^2$.
		\item \label{ass:coeffs}
		For the mobility and viscosity coefficients $\mj \in \rmC^1(\R)$ and $\mr,\nu,\eta\in \rmC^0(\R)$, there exist constants $0<K_*<K^*$ such that
		\begin{align*}
			K_*\leq \mj(s), \mr(s), \nu(s), \eta(s) \leq K^* \quad\text{ for all } s\in\R.
		\end{align*}
		\item \label{ass:friction}
		The friction parameter $\gamma$ is positive.
	\end{enumerate}

	\begin{remark}
		With regard to these assumptions, we point out the following.
		\begin{enumerate}[label=(\roman*)]
			\item A homogeneous free energy density $F$ satisfying \eqref{ass:pot} is for instance given by the continuous extension to the interval $[-1,1]$ of the logarithmic potential
			\begin{align*}
				F(s) := \frac{\theta}{2} \bigprt{ (1+s)\ln(1+s)+(1-s)\ln(1-s) } + \frac{\theta_\mathrm{c}}{2} \bigprt{1-s^2}
			\end{align*}
			for all $s\in (-1,1)$ with some constants $0<\theta<\theta_\mathrm{c}$.
			\item By an approximation argument it can be verified that all our results remain valid in the case $\mj\in \rmC^0(\R)$.
		\end{enumerate}
	\end{remark}

	\subsection{Preliminary Results} 
	
	The preliminary results that are subsequently gathered are of relevance to our analysis.
	
	\medskip
	\noindent\textbf{Subgradients with prescribed mean value.} 
	First of all, we collect some useful results about subgradients related to the free energy from \cite{AbelsWilke2007}, Section~4. Given $m\in(-1,1)$, the convex part of the free energy, namely
	\begin{align*}
		E_m(u) =
		\begin{cases}
			\int_\Omega F_0(u) + \frac12 \abs{\Grad u}^2 \dx &\text{for } u\in\dom E_m, \\
			\infty &\text{else},
		\end{cases}
	\end{align*}
	with $F_0$ as in \eqref{ass:pot}, is defined on the domain $\dom E_m=\{u\in \rmH^1_{(m)}(\Omega) : \abs{u}\leq1 \text{ a.e.~in $\Omega$}\}$. With $\mathcal{P}$ being the power set, we denote by $\partial E_m \colon \rmL^2_{(m)}(\Omega) \to \mathcal{P} (\rmL^2_{(0)} (\Omega))$ the subgradient of~$E_m$ in the following sense. For $u\in\dom E_m$ and $v\in \rmL^2_{(0)}(\Omega)$, we have $v\in\partial E_m(u)$ if and only if
	\begin{align*}
		(v,u'-u)_{\rmL^2(\Omega)} \leq E_m(u') - E_m(u)
		\quad \text{for all } u'\in \rmL^2_{(m)}(\Omega).
	\end{align*}
	This generalized definition of subgradients is used since $\rmL^2_{(m)}(\Omega)$ is an \textit{affine} subspace of the Hilbert space $\rmL^2(\Omega)$ with tangent space $\rmL^2_{(0)}(\Omega)$.
	
	Identifying $\partial E_m$ with $\partial\tilde{E} \colon \rmL^2_{(0)}(\Omega) \to \mathcal{P} (\rmL^2_{(0)} (\Omega))$ by the shift $\tilde{E}(u)=E(u+m)$, we say $\partial E_m$ is maximal monotone if and only if $\partial \tilde{E}$ is maximal monotone. With $\tilde E$ and $E_0$ coinciding in the case $m=0$, and since $\partial E_0$ is a maximal monotone operator on $\rmL^2_{(0)}(\Omega)$, we obtain that $\partial E_m$ is a maximal monotone operator for arbitrary $m\in(-1,1)$, cf.~\cite{Abels2007}, Section~3.12.
	
	\begin{proposition}
		\label{prelim:prop:subgrad_mean}
		The domain of the subgradient of $E_m$ is given by
		\begin{align}
			\label{eq:subgrad_with:domain}
			\mathcal{D}(\partial E_m) 
			= \big\{u\in \rmH^2_{(m)}(\Omega) : F_0'(u)\in \rmL^2(\Omega),\, F_0''(u)\abs{\Grad u}^2 \in \rmL^1(\Omega),\, \partial_\n u\vert_{\partial\Omega}=0 \big\},
		\end{align}
		and for $u\in \mathcal{D}(\partial E_m)$, it holds 
		\begin{align}
			\label{eq:subgrad_with:subgradient}
			\partial E_m(u) = - \Lap u + P_0 (F_0'(u)),
		\end{align}
		where $P_0 \colon \rmL^2(\Omega)\to \rmL^2_{(0)}(\Omega)$, $P_0(f) \coloneqq f - \mean{f}$, is the orthogonal projection onto $\rmL^2_{(0)}(\Omega)$. Moreover, there exists a constant $C>0$ independent of $u\in \mathcal{D}(\partial E_m)$ such that
		\begin{align}
			\label{eq:subgrad_with:estimate_H2}
			\norm{u}_{\rmH^2(\Omega)}^2 + \norm{F_0'(u)}_{\rmL^2(\Omega)}^2 + \int_\Omega F_0''(u)\abs{\Grad u}^2\dx
			\leq C \Bigprt{\norm{\partial E_m(u)}_{\rmL^2(\Omega)}^2 + \norm{u}_{\rmL^2(\Omega)}^2 +1 }.
		\end{align}
		Finally, for every $r\in(1,2]$, there exists a constant $C_r>0$ such that
		\begin{align}
			\label{eq:subgrad_with:estimate_W2r}
			\norm{u}_{\rmW^{2,r}(\Omega)} + \norm{F_0'(u)}_{\rmL^r(\Omega)}
			\leq C_r \bigprt{\norm{\partial E_m(u)}_{\rmL^r(\Omega)} + \bigabs{\mean{F_0'(u)}} + \norm{u}_{\rmL^2(\Omega)} +1 }.
		\end{align}
	\end{proposition}

	For a proof of \eqref{eq:subgrad_with:domain}--\eqref{eq:subgrad_with:estimate_H2}, we refer to Theorem~4.3 in \cite{AbelsWilke2007}. Estimate~\eqref{eq:subgrad_with:estimate_W2r} is established in \cite{Abels2007}, Theorem 3.12.8, in the case without prescribed mean value, which causes the presence of the additional second term on the right-hand side of \eqref{eq:subgrad_with:estimate_W2r}.

	\medskip
	\noindent\textbf{Subgradients without prescribed mean value.} 
	A similar result as above holds for the subgradient of
	\begin{align*}
		E(u) =
		\begin{cases}
			\int_\Omega F_0(u) + \frac12 \abs{\Grad u}^2 \dx &\text{for } u\in\dom E, \\
			\infty &\text{else},
		\end{cases}
	\end{align*}
	with $F_0$ as in \eqref{ass:pot}, in the space $\rmL^2(\Omega)$ with $\dom E=\{u\in \rmH^1(\Omega) : \abs{u}\leq1 \text{ a.e.~in }\Omega\}$.
	
	\begin{proposition}
		\label{prelim:prop:subgrad}
		The domain of the subgradient of $E$ is given by
		\begin{align}
			\label{eq:subgrad_without:domain}
			\mathcal{D}(\partial E) 
			= \big\{u\in \rmH^2(\Omega) : F_0'(u)\in \rmL^2(\Omega),\, F_0''(u)\abs{\Grad u}^2 \in \rmL^1(\Omega),\, \partial_\n u\vert_{\partial\Omega}=0 \big\},
		\end{align}
		and for $u\in \mathcal{D}(\partial E)$, it holds 
		\begin{align}
			\label{eq:subgrad_without:subgradient}
			\partial E(u) = - \Lap u + F_0'(u).
		\end{align}
		Moreover, there exists a constant $C>0$ independent of $u\in \mathcal{D}(\partial E)$ such that
		\begin{align}
			\label{eq:subgrad_without:estimate_H2}
			\norm{u}_{\rmH^2(\Omega)}^2 + \norm{F_0'(u)}_{\rmL^2(\Omega)}^2 + \int_\Omega F_0''(u)\abs{\Grad u}^2\dx
			\leq C \Bigprt{\norm{\partial E(u)}_{\rmL^2(\Omega)}^2 + \norm{u}_{\rmL^2(\Omega)}^2 +1 }.
		\end{align}
		Finally, for every $p\in[2,\infty)$, there exists a constant $C_p>0$ such that
		\begin{align}
			\label{eq:subgrad_without:estimate_W2p}
			\norm{u}_{\rmW^{2,p}(\Omega)} + \norm{F_0'(u)}_{\rmL^p(\Omega)}
			\leq C_p \bigprt{\norm{\partial E(u)}_{\rmL^p(\Omega)} +1 }.
		\end{align}
	\end{proposition}
	
	The results \eqref{eq:subgrad_without:domain}--\eqref{eq:subgrad_without:estimate_H2} are found in \cite{Abels2007}, Theorem 3.12.8, whereas \eqref{eq:subgrad_without:estimate_W2p} follows from a slight modification of the proof of Lemma 3.12.3 in \cite{Abels2007} to the case without prescribed mean value.

	\medskip
	\noindent\textbf{Neumann--Laplace equation.}
	Following the outline in \cite{Abels2009g}, Section 4, we provide an overview of the Neumann--Laplace operator.
	To begin with, we note that $\rmH_{(0)}^1(\Omega)$, endowed with the inner product
	\begin{align*}
		(u,v)_{\rmH_{(0)}^1(\Omega)} \coloneqq (\Grad u,\Grad v)_{\rmL^2(\Omega)},
	\end{align*}
	is a Hilbert space due to Poincaré's inequality, with its dual denoted by $\rmH_{(0)}^{-1}(\Omega)$. Next, we define the weak divergence operator $\Div_\n \colon \rmL^2(\Omega)^d\to \rmH_{(0)}^{-1}(\Omega)$ by
	\begin{align*}
		\ang{\Div_\n \mathbf{u}}{\phi}_{\rmH_{(0)}^{-1}(\Omega), \rmH_{(0)}^1(\Omega)}
		= - (\mathbf{u},\Grad\phi)_{\rmL^2(\Omega)}
		\quad\text{for all } \phi\in \rmH_{(0)}^1(\Omega).
	\end{align*}
	Accordingly, the weak Neumann--Laplace operator $\Lap_\Neum \coloneqq \Div_\n\Grad \colon \rmH_{(0)}^1(\Omega) \to \rmH_{(0)}^{-1}(\Omega)$ is introduced, given by
	\begin{align*}
		\ang{\Lap_\Neum u}{\phi}_{\rmH_{(0)}^{-1}(\Omega), \rmH_{(0)}^1(\Omega)}
		= - (\Grad u,\Grad\phi)_{\rmL^2(\Omega)}
		\quad\text{for all } \phi\in \rmH_{(0)}^1(\Omega).
	\end{align*}
	In view of the Lax--Milgram theorem, this operator is an isomorphism, and for $f\in \rmH_{(0)}^{-1}(\Omega)$ and $u\in \rmH_{(0)}^1(\Omega)$ solving $\Lap_\Neum u=f$, it follows
	\begin{align*}
		\norm{u}_{\rmH_{(0)}^1(\Omega)} \leq \norm{f}_{\rmH_{(0)}^{-1}(\Omega)}.
	\end{align*}
	Furthermore, if $u\in \rmH_{(0)}^1(\Omega)$ satisfies the equation $\Lap_\Neum u=f$ for some $f\in \rmL^p_{(0)}(\Omega)$, $p\in(1,\infty)$, standard elliptic regularity theory implies that $u$ even belongs to $\rmW^{2,p}(\Omega)$, and that it holds~$\Lap_\Neum u=f$ a.e.~in $\Omega$ as well as $\partial_\n u\vert_{\partial\Omega}=0$ in the sense of traces. In this case, there additionally exists a constant $C_p>0$ such that
	\begin{align*}
		\norm{u}_{\rmW^{2,p}(\Omega)} \leq C_p \norm{f}_{\rmL^p(\Omega)}.
	\end{align*}
	In particular, $\Lap_\Neum\colon \rmW^{2,p}_\Neum(\Omega) \to \rmL^p(\Omega)$ is a bijective operator, where the notation from \eqref{notation:eq:WkpN} is used.
	
	Eventually, the following proposition based on a duality argument allows us to derive a concept of \textit{very weak solutions} to the Neumann--Laplace equation.
	
	\begin{proposition}
		\label{prelim:prop:very_weak_sol}
		Let $p\in (1,\infty)$. For every $f\in \bigprt{\rmW^{2,p'}_\Neum(\Omega)}'$, there is a unique $u\in \rmL^p(\Omega)$ satisfying
		\begin{align*}
			(u,\Lap\phi)_{\rmL^2(\Omega)}
			= \ang{f}{\phi}_{\bigprt{\rmW^{2,p'}_\Neum(\Omega)}', \rmW^{2,p'}_\Neum(\Omega)}
			\quad\text{for all } \phi\in \rmW^{2,p'}_\Neum(\Omega),
		\end{align*}
		which we refer to as \textnormal{very weak solution} to the Neumann--Laplace equation. Moreover, there exists a constant $C_p>0$ such that
		\begin{align}
			\label{eq:Neum-Lap:estimate}
			\norm{u}_{\rmL^p(\Omega)} \leq C_p \norm{f}_{\bigprt{\rmW^{2,p'}_\Neum(\Omega)}'}.
		\end{align}
	\end{proposition}
	
	\begin{proof}
		We recall that the operator $\Lap_\Neum\colon \rmW^{2,p'}_\Neum(\Omega) \to \rmL^{p'}(\Omega)$ is bijective, and so is its adjoint~$\Lap_\Neum'\colon \rmL^p(\Omega) \to \bigprt{\rmW^{2,p'}_\Neum(\Omega)}'$. Consequently, for every $f\in \bigprt{\rmW^{2,p'}_\Neum(\Omega)}'$, there exists a unique $u\in \rmL^p(\Omega)$ solving the equation $\Lap_\Neum'u=f$. By the definition of the adjoint, it follows
		\begin{align*}
			(u,\Lap\phi)_{\rmL^2(\Omega)}
			= (u,\Lap_\Neum\phi)_{\rmL^2(\Omega)}
			= \ang{\Lap_\Neum' u}{\phi}_{\bigprt{\rmW^{2,p'}_\Neum(\Omega)}', \rmW^{2,p'}_\Neum(\Omega)}
			= \ang{f}{\phi}_{\bigprt{\rmW^{2,p'}_\Neum(\Omega)}', \rmW^{2,p'}_\Neum(\Omega)}
		\end{align*}
		for all $\phi\in \rmW^{2,p'}_\Neum(\Omega)$. Finally, the continuity of $(\Lap_\Neum')^{-1}$ implies estimate~\eqref{eq:Neum-Lap:estimate}.
	\end{proof}

	\section{Derivation of a Model with Volume Averaged Velocity} \label{sec:AGG_modeling}
	
	In this section, we derive a diffuse interface model for two-phase flows with phase transitions in the framework of continuum fluid dynamics, which is formulated for incompressible, isothermal constituents with unmatched densities. Using a volume averaged velocity, we establish a generalization of the model by Abels, Garcke, and Grün \cite{AbelsGarckeGruen2012} to a quasi-incompressible system that explicitly accounts for mass exchanges between the components. The thermodynamic consistency of the resulting model is ensured by selecting constitutive assumptions in such a way that an energy dissipation inequality is satisfied.
	
	\medskip
	\noindent\textbf{Balance equations.}
	Let $\Omega\subset\R^d$, $d\in\N$, denote an open, bounded set representing the spatial region occupied by a binary fluid mixture whose flow is studied over the time interval~$[0,\infty)$.
	The indices $\pm$ are used to distinguish quantities associated with each component of the mixture, where the constituents may describe either two distinct substances or two phases of a single material. For these macroscopically immiscible fluids, we assume a partial miscibility in a thin interfacial region.
	It is sufficient to describe the thermodynamic state of the mixture by the following three quantities. The constituents are characterized by their respective mass density $\rho_\pm \colon [0,\infty)\times\Omega\to\R_+$ and velocity $\vel_\pm \colon [0,\infty)\times\Omega\to\R^d$~and, considering an isothermal mixture, by a common temperature $\theta\colon [0,\infty)\times\Omega\to\R_+$ which is assumed to be a given constant here.
	
	Each unmixed fluid is assigned a specific constant mass density $\tilde\rho_\pm >0$, where $\tilde\rho_+ \neq \tilde\rho_-$ is required. For the sum of the two volume fractions $\frac{\rho_\pm}{\tilde\rho_\pm}$, we assume vanishing excess volume, expressed as
	\begin{align}
		\label{AGG:modeling:eq:excess_vol}
		\frac{\rho_+}{\tilde\rho_+} + \frac{\rho_-}{\tilde\rho_-} = 1,
	\end{align}
	which indicates that the fluids behave like a simple mixture, cf.~\cite{LowengrubTruskinovsky1998}. 

	Since mass transfer between the components by phase transition is taken into account, the \textit{partial mass balance equation} for each constituent reads
	\begin{align}
		\label{AGG:modeling:eq:partial_mb}
		\partial_t \rho_\pm + \nabla\cdot (\rho_\pm\vel_\pm) = r_\pm,
	\end{align}
	with $r_\pm$ describing the mass production rate of the respective constituent. While individual masses of the components are not conserved, we require total mass conservation for the mixture in the sense that
	\begin{align}
		\label{AGG:modeling:eq:total_mass_cons}
		r_+ + r_- = 0,
	\end{align}
	where the total mass density $\rho\coloneqq \rho_+ + \rho_-$ of the mixture is given as sum of the partial mass densities~$\rho_\pm$. Notably, we work here with a \textit{volume averaged} velocity 
	\begin{align}
		\vel\coloneqq \frac{\rho_+}{\tilde\rho_+}\vel_+ + \frac{\rho_-}{\tilde\rho_-}\vel_-.
	\end{align}
	Moreover, for each constituent, we consider a diffusion velocity $\w_\pm \coloneqq \vel_\pm - \vel$ as well as a diffusion flux $\J_\pm \coloneqq \rho_\pm\w_\pm$. Summing the two partial mass balance equations~\eqref{AGG:modeling:eq:partial_mb}~corresponding to the indices $\pm$ and using \eqref{AGG:modeling:eq:total_mass_cons} reveals the \textit{total mass balance equation} 
	\begin{align}
		\label{AGG:modeling:eq:mb}
		\partial_t \rho + \nabla\cdot(\rho\vel + \tilde\J) = 0,
	\end{align}
	where $\tilde\J \coloneqq \J_+ + \J_-$ represents the total mass flux relative to the velocity field.
	Next, we define the phase field variable $\varphi\coloneqq \frac{\rho_+}{\tilde\rho_+} - \frac{\rho_-}{\tilde\rho_-}$ as difference of the volume fractions, along with the corresponding phase field flux $\J_\varphi \coloneqq \frac{1}{\tilde\rho_+} \J_+ - \frac{1}{\tilde\rho_-} \J_-$. In view of \eqref{AGG:modeling:eq:excess_vol}, these definitions yield the converse relations $\rho = b_+ + b_- \varphi$ and $\tilde\J = b_- \J_\varphi$, where $b_\pm \coloneqq \frac{\tilde\rho_+ \pm \tilde\rho_-}{2}$. Additionally, we set $c_\pm \coloneqq \frac{1}{\tilde\rho_+} \pm \frac{1}{\tilde\rho_-}$ for brevity. 
	
	Dividing the individual mass balance equation \eqref{AGG:modeling:eq:partial_mb} for $\pm$ by $\tilde\rho_\pm$, respectively, their difference, along with \eqref{AGG:modeling:eq:total_mass_cons}, gives the \textit{phase field equation}
	\begin{align}
		\label{AGG:modeling:eq:phase_field_eq}
		\partial_t \varphi + \nabla\cdot (\varphi\vel + \J_\varphi)
		= c_+r_+
		\eqqcolon r_\varphi
	\end{align}
	with the phase field production rate $r_\varphi$. Similarly, after dividing the partial mass balance equation \eqref{AGG:modeling:eq:partial_mb} for $\pm$ by $\tilde\rho_\pm$ and summing now, we incorporate \eqref{AGG:modeling:eq:excess_vol} and \eqref{AGG:modeling:eq:total_mass_cons} to obtain the \textit{divergence equation}
	\begin{align}
		\label{AGG:modeling:eq:div_eq}
		\nabla\cdot\vel = c_- r_+
		= \alpha r_\varphi,
	\end{align}
	where we abbreviate $\alpha \coloneqq \frac{c_-}{c_+}$.
	
	Finally, introducing a stress tensor $\tilde\T$, we assume conservation of linear momentum for the total mixture, expressed as
	\begin{align*}
		\partial_t (\rho\vel) + \nabla\cdot \bigprt{\rho\vel\otimes\vel - \tilde\T} = 0.
	\end{align*}
	Here, external forces are neglected.
	With an objective stress tensor $\tilde\Str \coloneqq \tilde\T + \vel\otimes\tilde\J$, this can be rewritten as the \textit{(linear) momentum balance equation}
	\begin{align}
		\label{AGG:modeling:eq:momentum_eq}
		\partial_t (\rho\vel) + \nabla\cdot \bigprt{\rho\vel\otimes\vel} + \nabla\cdot \bigprt{\vel\otimes\tilde\J}
		= \nabla\cdot \tilde\Str.
	\end{align}

	\medskip
	\noindent\textbf{Energy dissipation.}
	In our system, we require the total energy density to be of the form~$e(\vel,\varphi,\nabla\varphi)= f(\varphi,\nabla\varphi) + \rho \frac{\abs \vel^2}{2}$ consisting of a free energy density $f$ and the kinetic energy $\rho \frac{\abs \vel^2}{2}$. Along with this quantity, we consider some general energy flux $\J_e$ which will be specified later. For every volume $V(t)$ in $\Omega$ transported with the flow with velocity $\vel$~over time $t$, the total energy is supposed to dissipate such that
	\begin{align*}
		\ddt \int_{V(t)} e(\vel,\varphi,\nabla\varphi) \dx
		+ \int_{\partial V(t)} \J_e \cdot \bnu \dsix
		\leq 0,
	\end{align*}
	where $\bnu$ is the outer unit normal field on $\partial V(t)$.
	In view of the divergence theorem and the Reynolds transport theorem we obtain the corresponding \textit{energy dissipation inequality} in a local form 
	\begin{align}
		\label{AGG:modeling:eq:enery_diss}
		\partial_t e + \nabla\cdot (e\vel) + \nabla\cdot \J_e \leq 0
	\end{align}
	since $V(t)$ is arbitrary. In the remainder of this section, let $\dot\phi \coloneqq \delt \phi + \vel\cdot\nabla\phi$ denote the material time derivative of a function $\phi$ with respect to $\vel$, and $g_{,\phi} \coloneqq \frac{\partial}{\partial\phi} g$ the partial derivative of a function $g=g(\phi)$ with respect to $\phi$. Using the formula
	\begin{align*}
		\delxj\dot{\varphi}
		= \delt\delxj\varphi + \vel\cdot\Grad\delxj\varphi + \Grad\varphi\cdot\delxj\vel
		= (\delxj\varphi)^{\boldsymbol\cdot} + \Grad\varphi\cdot\delxj\vel
	\end{align*}
	we compute for the free energy density
	\begin{align*}
		\partial_t f + \nabla\cdot (f\vel) 
		&= \delt f + \vel\cdot\Grad f + f\nabla\cdot\vel \\
		&= f_{,\varphi} \big( \delt \varphi + \vel\cdot\nabla\varphi \big) 
		+ f_{,\Grad\varphi} \cdot \big( \delt \Grad\varphi 
		+ (\vel\cdot\nabla)\Grad\varphi \big ) 
		+ f\nabla\cdot\vel \\
		&= f_{,\varphi} \dot\varphi 
		+ f_{,\Grad\varphi} \cdot (\Grad\varphi)^{\boldsymbol\cdot}
		+ f\nabla\cdot\vel \\
		&= f_{,\varphi} \dot\varphi 
		+ f_{,\Grad\varphi} \cdot (\Grad\dot\varphi) - \big(\Grad\varphi \otimes f_{,\Grad\varphi} \big) : \Grad\vel
		+ f\nabla\cdot\vel \\
		&= \big( f_{,\varphi} - \nabla\cdot f_{,\Grad\varphi} \big) \dot\varphi 
		+ \nabla\cdot \big( f_{,\Grad\varphi} \dot\varphi \big) 
		- \big(\Grad\varphi \otimes f_{,\Grad\varphi} - f \mathbf I \big) : \Grad\vel.
	\end{align*}
	Invoking the total mass balance equation~\eqref{AGG:modeling:eq:mb} and the momentum equation~\eqref{AGG:modeling:eq:momentum_eq} the kinetic energy density term results in 
	\begin{align*}
		\delt \biggprt{ \rho\frac{\abs \vel^2}{2} } + \nabla\cdot \biggprt{ \rho\frac{\abs \vel^2}{2} \vel}
		&= \bigprt{\delt\rho + \nabla\cdot(\rho\vel)} \frac{\abs \vel^2}{2} 
		+ \rho \bigprt{\delt\vel + (\vel\cdot\nabla)\vel} \cdot \vel \\
		&= - \nabla\cdot \tilde\J \frac{\abs \vel^2}{2}
		+ \bigprt{\nabla\cdot \tilde\Str - (\tilde\J \cdot\nabla )\vel} \cdot \vel \\
		&= \nabla\cdot \biggprt{ -\frac{\abs \vel^2}{2} \tilde\J + \trans{\tilde\Str}\vel} - \tilde\Str : \Grad\vel.
	\end{align*}

	The phase field equation~\eqref{AGG:modeling:eq:phase_field_eq} and the divergence equation~\eqref{AGG:modeling:eq:div_eq} provide a framework for introducing Lagrange multipliers into the dissipation inequality~\eqref{AGG:modeling:eq:enery_diss}. Specifically, for any given pair of scalar functions $\mu, \lambda$, where later $\mu$ will describe the chemical potential and~$\lambda$~a certain pressure, we have the inequality
	\begin{align*}
		-\mathcal D
		&\coloneqq \partial_t e + \nabla\cdot (e\vel) + \nabla\cdot \J_e  - \mu \bigprt{ \partial_t \varphi + \nabla\cdot (\varphi\vel + \J_\varphi) - r_\varphi} 
		- \lambda \bigprt{\nabla\cdot \vel - \alpha r_\varphi}
		\leq 0.
	\end{align*}
	In view of the above computations for the free energy density and the kinetic energy density, this is equivalent to
	\begin{align}
		\label{AGG:modeling:eq:Lagrange_multipliers}
		-\mathcal D
		&= \big( f_{,\varphi} - \nabla\cdot f_{,\Grad\varphi} \big) \dot\varphi 
		+ \nabla\cdot \big( f_{,\Grad\varphi} \dot\varphi \big) 
		- \big(\Grad\varphi \otimes f_{,\Grad\varphi} - f \mathbf I \big) : \Grad\vel 
		+ \nabla\cdot \biggprt{ -\frac{\abs \vel^2}{2} \tilde\J + \trans{\tilde\Str}\vel} \nonumber\\
		&\quad- \tilde\Str : \Grad\vel 
		+ \nabla\cdot \J_e 
		- \mu \bigprt{ \dot\varphi + \varphi \nabla\cdot\vel + \nabla\cdot \J_\varphi  - r_\varphi} 
		- \lambda \bigprt{\nabla\cdot \vel - \alpha r_\varphi} \nonumber\\
		&= \big( f_{,\varphi} - \nabla\cdot f_{,\Grad\varphi} - \mu\big) \dot\varphi 
		+ \nabla\cdot \biggprt{ f_{,\Grad\varphi} \dot\varphi - \frac{\abs \vel^2}{2} \tilde\J + \trans{\tilde\Str}\vel + \J_e - \mu\J_\varphi} \nonumber\\
		&\quad+ \Bigprt{-\Grad\varphi \otimes f_{,\Grad\varphi} - \tilde\Str + \bigprt{f-\mu\varphi-\lambda }\mathbf I } : \Grad\vel 
		+ \J_\varphi \cdot \Grad\mu
		+ r_\varphi \bigprt{\mu+\alpha\lambda} 
		\leq 0.
	\end{align}
	Assuming that the terms $f_{,\Grad\varphi} \dot\varphi - \frac{\abs \vel^2}{2} \tilde\J + \trans{\tilde\Str}\vel + \J_e - \mu\J_\varphi$ and $-\Grad\varphi \otimes f_{,\Grad\varphi} - \tilde\Str + \bigprt{f-\mu\varphi-\lambda }\mathbf I$ as well as $\Grad\mu$ and $\mu+\alpha\lambda$ do not depend on $\dot\varphi$, and regarding that $\dot\varphi$ may attain arbitrary values, we conclude that the term $f_{,\varphi} - \nabla\cdot f_{,\Grad\varphi} - \mu$ needs to vanish. This yields
	\begin{align}
		\label{AGG:modeling:eq:eq_chem_pot}
		\mu 
		= f_{,\varphi} - \nabla\cdot f_{,\Grad\varphi}
	\end{align}
	as equation for the chemical potential $\mu$. Furthermore, we specify the energy flux $\J_e$ as
	\begin{align*}
		\J_e 
		= - f_{,\Grad\varphi} \dot\varphi + \frac{\abs \vel^2}{2} \tilde\J - \trans{\tilde\Str}\vel + \mu\J_\varphi 
	\end{align*}
	such that the divergence term in inequality \eqref{AGG:modeling:eq:Lagrange_multipliers} vanishes. Now, denoting the symmetric part of $\Grad\vel$ by $\D\vel \coloneqq \frac12 \bigprt{ \Grad\vel + \trans{\Grad\vel} }$, the additional assumption that $\tilde\Str$ is symmetric implies
	\begin{align*}
		&\bigprt{ \tilde\Str + \Grad\varphi \otimes f_{,\Grad\varphi} } : \Grad\vel \\
		&= \bigprt{ \tilde\Str + \Grad\varphi \otimes f_{,\Grad\varphi} } : \D\vel
		+ \bigprt{ \Grad\varphi \otimes f_{,\Grad\varphi} } : \tfrac12 \bigprt{ \Grad\vel - \trans{\Grad\vel} } \\
		&= \bigprt{ \tilde\Str + \Grad\varphi \otimes f_{,\Grad\varphi} } : \D\vel 
		+ \tfrac12 \bigprt{ \Grad\varphi \otimes f_{,\Grad\varphi} - f_{,\Grad\varphi} \otimes \Grad\varphi} : \tfrac12 \bigprt{ \Grad\vel - \trans{\Grad\vel} }.
	\end{align*}
	With the skew part of $\Grad\vel$ attaining arbitrary values independent of $\D\vel$, it follows that the term $\Grad\varphi \otimes f_{,\Grad\varphi} - f_{,\Grad\varphi} \otimes \Grad\varphi$ vanishes. Altogether, inequality~\eqref{AGG:modeling:eq:Lagrange_multipliers} simplifies to
	\begin{align}
		\label{AGG:modeling:eq:Lagrange_multipliers_remaining}
		-\mathcal D
		&= \Bigprt{-\Grad\varphi \otimes f_{,\Grad\varphi} - \tilde\Str + \bigprt{ f-\mu\varphi-\lambda } \mathbf I }: \D\vel 
		+ \J_\varphi \cdot \Grad\mu 
		+ r_\varphi \bigprt{\mu+\alpha\lambda}
		\leq 0.
	\end{align}
	This reveals that the mechanisms leading to energy dissipation are viscosity, diffusion and phase transition.

	\medskip
	\noindent\textbf{Constitutive assumptions.}
	For the dissipation inequality~\eqref{AGG:modeling:eq:Lagrange_multipliers_remaining} to hold true, we propose the following \textit{constitutive assumptions}: 
	\begin{enumerate}
		\item
		The stress $\tilde\Str$ satisfies the identity
		\begin{align*}
			-\Grad\varphi \otimes f_{,\Grad\varphi} - \tilde\Str + \bigprt{ f-\mu\varphi-\lambda } \mathbf I 
			= - 2\nu(\varphi) \D\vel - \eta(\varphi) \nabla\cdot\vel \mathbf I
			\eqqcolon -\Str(\varphi, \D\vel)
		\end{align*}
		for some viscosity coefficients $\nu(\varphi), \eta(\varphi)\geq0$. In particular, this means that the mixture is Newtonian.
		\item For the phase field flux $\J_\varphi$, motivated by a generalized Fick's law, a diffusion law
		\begin{align*}
			\J_\varphi = -\mj(\varphi) \Grad\mu
		\end{align*}
		holds, with a diffusion mobility $\mj(\varphi)\geq0$.
		\item The phase field production rate $r_\varphi$ fulfills the phase transition law
		\begin{align*}
			r_\varphi = - \mr(\varphi) \bigprt{\mu+\alpha\lambda}
		\end{align*}
		for some transition mobility $\mr(\varphi)\geq0$.
	\end{enumerate}

	\medskip
	\noindent\textbf{The model.}
	Eventually, we insert the constitutive assumptions into our balance equations. For the momentum equation~\eqref{AGG:modeling:eq:momentum_eq}, we compute, invoking the equation for the chemical potential~\eqref{AGG:modeling:eq:eq_chem_pot},
	\begin{align*}
		&\partial_t (\rho\vel) + \nabla\cdot \bigprt{\rho\vel\otimes\vel} + \nabla\cdot \bigprt{\vel\otimes\tilde\J}
		= \nabla\cdot \tilde\Str \\
		& = \nabla\cdot \Bigprt{ \Str(\varphi, \D\vel) -\Grad\varphi \otimes f_{,\Grad\varphi} + \bigprt{ f-\mu\varphi-\lambda } \mathbf I } \\
		&= \nabla\cdot\Str(\varphi, \D\vel) - \bigprt{\nabla\cdot f_{,\Grad\varphi}} \Grad\varphi - \D^2\varphi f_{,\Grad\varphi} 
		 + f_{,\varphi} \Grad\varphi + \D^2\varphi f_{,\Grad\varphi} - \varphi\Grad\mu - \mu\Grad\varphi - \Grad\lambda\\
		&= \nabla\cdot\Str(\varphi, \D\vel) - \varphi\Grad\mu - \Grad\lambda.
	\end{align*}
	The divergence equation~\eqref{AGG:modeling:eq:div_eq} takes the form
	\begin{align*}
		\nabla\cdot\vel
		= \alpha r_\varphi
		= - \alpha \mr(\varphi) \bigprt{\mu+\alpha\lambda},
	\end{align*}
	whereas the phase field equation~\eqref{AGG:modeling:eq:phase_field_eq} results in 
	\begin{align*}
		\partial_t \varphi + \nabla\cdot (\varphi\vel)
		= - \nabla\cdot \J_\varphi + r_\varphi
		= \nabla\cdot \bigprt{\mj(\varphi) \Grad\mu} - \mr(\varphi) \bigprt{\mu+\alpha\lambda}.
	\end{align*}

	Altogether, writing $Q=(0,\infty)\times\Omega$ for the space-time cylinder, we derived the quasi-incompressible Navier--Stokes/Cahn--Hilliard model
	\begin{equation}
		\label{AGG:modeling:eqs:model}
		\arraycolsep=2pt
		\begin{array}{rclll}
			\partial_t (\rho\vel) + \nabla\cdot(\rho\vel\otimes\vel) 
			+ \nabla\cdot(\vel\otimes\tilde\J) \\[0.2ex]
			- \nabla\cdot \Str(\varphi,\D\vel) + \Grad\lambda
			&=& -\varphi\Grad\mu 
			&\text{ in $Q$}, \\[1ex]
			\nabla\cdot\vel 
			&=& - \alpha \mr(\varphi) \bigprt{\mu+\alpha\lambda} 
			&\text{ in $Q$}, \\[1ex]
			\partial_t \varphi + \nabla\cdot (\varphi\vel) 
			&=& \nabla\cdot \bigprt{\mj(\varphi) \Grad\mu} - \mr(\varphi) \bigprt{\mu+\alpha\lambda}
			&\text{ in $Q$}, \\[1ex]
			\mu 
			&=& f_{,\varphi} - \nabla\cdot f_{,\Grad\varphi}
			&\text{ in $Q$},
		\end{array}
	\end{equation}
	where the relations
	\begin{equation*}
		\arraycolsep=2pt
		\begin{array}{rcl}
			\Str(\varphi, \D\vel) &=& 2\nu(\varphi) \D\vel + \eta(\varphi) \nabla\cdot\vel \mathbf I,\\[1ex]
			\rho(\varphi) &=& b_+ + b_- \varphi,\\[1ex]
			\tilde\J(\varphi,\Grad\mu) &=& -b_- \mj (\varphi) \Grad\mu
		\end{array}
	\end{equation*}
	are satisfied. We recall that the occurring constants are given by 
	\begin{align*}
			b_\pm = \tfrac{\tilde\rho_+ \pm \tilde\rho_-}{2}, \qquad
			c_\pm = \tfrac{1}{\tilde\rho_+} \pm \tfrac{1}{\tilde\rho_-}, \qquad
			\alpha = \tfrac{c_-}{c_+}, \qquad
			\text{where } \tilde\rho_\pm>0, \quad \tilde\rho_+ \neq \tilde\rho_-,
	\end{align*}
	and that the viscosity and mobility coefficients $\nu(\varphi),\eta(\varphi),\mj(\varphi),\mr(\varphi) \geq0$ are nonnegative.
	
	With $S=(0,\infty)\times\partial\Omega$ denoting the lateral surface of the space-time cylinder, we add boundary conditions to system \eqref{AGG:modeling:eqs:model}, namely
	\begin{equation}
		\label{AGG:modeling:eqs:bdry_cond}
		\arraycolsep=2pt
		\begin{array}{rclll}
			\vel \vert_{\partial\Omega}&=&0&\text{ on $S$}, \\[1ex]
			\partial_\n\varphi \vert_{\partial\Omega} = \partial_\n\mu \vert_{\partial\Omega} &=&0&\text{ on $S$},
		\end{array}
	\end{equation}
	where $\n$ is the outer unit normal vector field on $\partial\Omega$.
	Concerning $\vel$, this is the no-slip boundary condition for viscous fluids. Moreover, the identity $\partial_\n\varphi \vert_{\partial\Omega} =0$ prescribes a contact angle of $\pi/2$ between the diffuse interface and the boundary of the domain, while we ensure by $\partial_\n\mu \vert_{\partial\Omega} =0$ that there is no mass flux across the boundary.
	
	Finally, before stating the formal energy balance of the system, let us recall that the total energy density is given by $e(\vel,\varphi,\nabla\varphi)= f(\varphi,\nabla\varphi) + \rho \frac{\abs \vel^2}{2}$.
	Multiplying the first equation of~\eqref{AGG:modeling:eqs:model} by $\vel$, the second by $\lambda$, the third by $\mu$, and the fourth by $-\delt\varphi$, then integrating over~$\Omega$ and summing, reveals in view of the boundary conditions \eqref{AGG:modeling:eqs:bdry_cond} that every sufficiently smooth solution to the model satisfies
	\begin{align*}
		&\ddt \int_\Omega e\bigprt{\vel(t),\varphi(t),\Grad\varphi(t)} \dx
		= - \int_\Omega \Str(\varphi,\D\vel) : \D\vel + \mj(\varphi) \abs{\Grad\mu}^2 + \mr(\varphi) \bigprt{\mu+\alpha\lambda}^2 \dx.
	\end{align*}
	For details, we refer to the proof of Lemma~\ref{AGG:exws:lemma:discr_energy_est}, where the energy inequality for the corresponding time-discrete system is shown.

	\section{Existence of Weak Solutions to the Model with Volume Averaged Velocity} \label{sec:AGG_exws}
	
	In this section, we prove existence of weak solutions to the system \eqref{eq:AGG} under the assumptions \eqref{ass:domain}--\eqref{ass:coeffs} which we require to hold throughout the entire section. We point out that \eqref{eq:AGG} is the model derived in Section \ref{sec:AGG_modeling}, where we impose the free energy density to be of the form $f(\varphi,\Grad\varphi) = F(\varphi) + \frac{\abs{\Grad\varphi}^2}{2}$, with a singular homogeneous free energy density~$F$ specified in \eqref{ass:pot}.
	Then, the total energy of the system is given by
	\begin{align*}
		E_\tot\prt{\vel,\varphi}
		= E_\kin(\vel) + E_\free(\varphi) 
		= \int_\Omega \rho \frac{\abs \vel^2}{2} \dx + \int_\Omega F(\varphi) + \frac{\abs{\Grad\varphi}^2}{2} \dx.
	\end{align*}
	Moreover, with $F_0(s)\coloneqq F(s)+\frac\kappa2s^2$ being the convex part of $F$, we denote the convex part of the free energy by
	\begin{align*}
		E(\varphi) =
		\begin{cases}
			\int_\Omega F_0(\varphi) + \frac12 \abs{\Grad \varphi}^2 \dx &\text{for } \varphi\in\dom E, \\
			\infty &\text{else},
		\end{cases}
	\end{align*}
	where $\dom E=\{\varphi\in \rmH^1(\Omega) : \abs{\varphi}\leq1 \text{ a.e.~in }\Omega\}$. This will be used to apply results related to the subgradient of $E$, cf.~Proposition~\ref{prelim:prop:subgrad}.
	
	In the following, we write $Q_{(s,t) }= (s,t)\times\Omega$ and $Q=Q_{(0,\infty)}$ for space-time cylinders as well as $S_{(s,t)} = (s,t)\times\partial\Omega$ and $S=S_{(0,\infty)}$ for their corresponding lateral surfaces.

	\subsection{Definition of Weak Solutions}  \label{AGG:exws:subsec:def_ws}
	
	Subsequently, we state our precise definition of weak solutions to system~\eqref{eq:AGG}.
	
	\begin{definition}
		\label{AGG:exws:def:weak_sol}
		Let $\vel_0\in \rmL^2(\Omega)^d$ and $\varphi_0\in \rmH^1(\Omega)$ with $\abs{\varphi_0}\leq1$ a.e.~in $\Omega$. Under the assumptions \eqref{ass:domain}--\eqref{ass:coeffs}, a quadruple $(\vel,\lambda,\mu,\varphi)$ with the properties
		\begin{alignat*}{2}
			\vel &\in \BCw([0,\infty);\rmL^2(\Omega)^d) \cap \rmL^2(0,\infty;\rmH_0^1(\Omega)^d),\\
			\lambda &\in \rmL^2_\uloc([0,\infty);\rmL^2(\Omega)),\\
			\mu &\in \rmL^2_\uloc([0,\infty);\rmH^1(\Omega)) \text{ with } \Grad\mu\in \rmL^2(0,\infty;\rmL^2(\Omega)^d),\\
			\varphi &\in \BCw([0,\infty);\rmH^1(\Omega)) \cap \rmL^2_\uloc([0,\infty);\rmW^{2,p}(\Omega)),\\
			F'(\varphi) &\in \rmL^2_\uloc([0,\infty);\rmL^p(\Omega)),
		\end{alignat*}
		where $p=6$ if $d=3$ and $p\in[2,\infty)$ is arbitrary if $d=2$, respectively, is called a \emph{weak solution of \eqref{eq:AGG}} if the following equations are satisfied:
		\begin{subequations}
		\begin{align} 
			\label{AGG:exws:eq:weak_sol_a}
			&\int_Q -\rho\vel \cdot \delt \bpsi - (\rho\vel\otimes\vel) : \Grad\bpsi - (\vel\otimes\tilde\J) : \Grad\bpsi
			+ \Str(\varphi, \D\vel) : \Grad\bpsi - \lambda \nabla\cdot\bpsi \dtx \nonumber\\
			&= -\int_Q \varphi\Grad\mu \cdot \bpsi \dtx
		\end{align}
		for all $\bpsi\in \rmC_0^\infty(Q)^d$, along with
		\begin{align}
			\label{AGG:exws:eq:weak_sol_b} 
			\nabla\cdot\vel = -\alpha \mr(\varphi) \bigprt{\mu+\alpha\lambda} \quad\text{a.e.~in $Q$}, 
		\end{align}
		as well as
		\begin{align}
			\label{AGG:exws:eq:weak_sol_c} 
			\int_Q - \varphi\delt\zeta - \varphi\vel\cdot\Grad\zeta \dtx
			= \int_Q -\mj(\varphi) \Grad\mu \cdot\Grad\zeta - \mr(\varphi) \bigprt{\mu+\alpha\lambda} \zeta \dtx
		\end{align}
		for all $\zeta\in \rmC_0^\infty((0,\infty)\times\overline\Omega)$, and
		\begin{align}
			\label{AGG:exws:eq:weak_sol_d}
			\mu = F'(\varphi)-\Lap\varphi \quad\text{a.e.~in $Q$}. 
		\end{align}
		\end{subequations}
		Moreover, the boundary and initial conditions
		\begin{alignat*}{2}
			\partial_\n\varphi \vert_{\partial\Omega} &=0 &&\quad\text{a.e.~in $S$}, \\ 
			(\vel,\varphi) \vert_{t=0} &= (\vel_0,\varphi_0) &&\quad\text{a.e.~in $\Omega$}
		\end{alignat*}
		need to hold, as well as the energy inequality 
		\begin{align}
			\label{AGG:exws:eq:weak_sol_energy_ineq}
			&E_\tot\bigprt{\vel(t),\varphi(t)}
			+ \int_{Q_{(s,t)}} \Str(\varphi,\D\vel) : \D\vel + \mj(\varphi) \abs{\Grad\mu}^2 + \mr(\varphi) \bigprt{\mu+\alpha\lambda}^2 \dtaux \nonumber\\
			&\leq E_\tot\bigprt{\vel(s),\varphi(s)}
		\end{align}
		for almost all $s\in [0,\infty)$ including $s=0$, and for all $t\in[s,\infty)$. 
	\end{definition}

	\begin{remark}
		For stationary states $(\vel^*,\lambda^*,\mu^*,\varphi^*)$ of the system~\eqref{eq:AGG}, we deduce from the energy inequality \eqref{AGG:exws:eq:weak_sol_energy_ineq} that $\D\vel^*=0$, $\Grad\mu^*=0$ and $\mu^*+\alpha\lambda^*=0$. As a direct consequence, we obtain $\vel^*=0$ by Korn's inequality, as well as $\mu^*$ and $\lambda^*$ being constant. Therefore,~$(\mu^*,\varphi^*)$ is a stationary state of the Cahn--Hilliard equation, that is a solution of 
		\begin{equation*}
			\arraycolsep=2pt
			\begin{array}{rclll}
				F'(\varphi^*) - \Lap\varphi^* &=&\mu^* &\text{ in } \Omega, \\[0.5ex]
				\partial_\n \varphi^* \vert_{\partial\Omega} &=&0 &\text{ on } \partial\Omega.
			\end{array}
		\end{equation*}
	\end{remark}

	\subsection{Implicit Time Discretization} \label{AGG:exws:subsec:time_discr}
	
	In this section, we construct approximate weak solutions of system~\eqref{eq:AGG} using an implicit time discretization.
		
	\begin{definition}
		\label{AGG:exws:def:time_discr_system}
		With fixed $N\in\N$, we denote the time step size by $h=\frac{1}{N}$. For $k\in\N_0$, let~$\vel_k\in \rmL^2(\Omega)^d$ and $\varphi_k\in \rmH^1(\Omega)$ satisfying $F'(\varphi_k)\in \rmL^2(\Omega)$ be given, and set $\rho_k\coloneqq b_+ + b_- \varphi_k$. Then, we determine a quadruple $(\vel,\lambda,\mu,\varphi) \in \rmH_0^1(\Omega)^d \times \rmL^2(\Omega) \times \rmH^2_\Neum(\Omega) \times \mathcal D(\partial E)$ as a solution of the following time-discrete system at time step $k+1$:
		\begin{subequations}
		\begin{align}
			&\int_\Omega \frac{\rho\vel-\rho_k\vel_k}{h} \cdot \bpsi + \nabla\cdot(\rho_k\vel\otimes\vel) \cdot \bpsi + \nabla\cdot (\vel\otimes\tilde\J) \cdot \bpsi
			+ \Str(\varphi_k, \D\vel) : \Grad\bpsi - \lambda \nabla\cdot\bpsi \dx \nonumber\\
			&= -\int_\Omega \varphi_k\Grad\mu \cdot \bpsi \dx
			\label{AGG:exws:eq:time-discr_sol_a}
		\end{align}
		for all $\bpsi\in \rmC_0^\infty(\Omega)^d$, where $\tilde\J = -b_- \mj (\varphi_k) \Grad\mu$ and $\rho= b_+ + b_- \varphi$, as well as
		\begin{alignat}{2}
			\nabla\cdot\vel &= -\alpha \mr(\varphi_k) \bigprt{\mu+\alpha\lambda} &&\quad\text{a.e.~in $\Omega$},
			\label{AGG:exws:eq:time-discr_sol_b}\\
			\frac{\varphi-\varphi_k}{h} + \nabla\cdot(\varphi_k\vel) &= \nabla\cdot\bigprt{\mj(\varphi_k) \Grad\mu} - \mr(\varphi_k) \bigprt{\mu+\alpha\lambda} &&\quad\text{a.e.~in $\Omega$}, 
			\label{AGG:exws:eq:time-discr_sol_c}\\
			\mu + \kappa\frac{\varphi+\varphi_k}{2} &= F_0'(\varphi)-\Lap\varphi &&\quad\text{a.e.~in $\Omega$}.
			\label{AGG:exws:eq:time-discr_sol_d}
		\end{alignat}
		\end{subequations}
	\end{definition}

	\begin{remark}
		\label{AGG:exws:rmk:eqiv_time_discr_NS}
		Multiplying equation \eqref{AGG:exws:eq:time-discr_sol_c} by $b_-$, a direct computation including the definition of the quantities $\tilde\J$, $\alpha$ and $\rho_k$ or $\rho$, respectively, reveals, in combination with \eqref{AGG:exws:eq:time-discr_sol_b}, that $\nabla\cdot\tilde\J =-\frac{\rho-\rho_k}{h} - \nabla\cdot(\rho_k\vel)$.
		This, along with the identity $\nabla\cdot\prt{\vel\otimes\tilde\J} = \nabla\cdot\tilde\J\vel + \prt{\tilde\J\cdot\nabla}\vel$, allows for rewriting \eqref{AGG:exws:eq:time-discr_sol_a} in the equivalent form
		\begin{align}
			&\int_\Omega \frac{\rho\vel-\rho_k\vel_k}{h} \cdot\bpsi 
			+ \nabla\cdot(\rho_k\vel\otimes\vel) \cdot \bpsi
			+ \Bigprt{\nabla\cdot\tilde\J -\frac{\rho-\rho_k}{h} - \nabla\cdot(\rho_k\vel)} \frac{\vel}{2} \cdot \bpsi \dx \nonumber\\
			&\quad+ \int_\Omega \bigprt{\tilde\J\cdot\nabla}\vel \cdot \bpsi
			+ \Str(\varphi_k, \D\vel) : \Grad\bpsi - \lambda \nabla\cdot\bpsi \dx 
			= -\int_\Omega \varphi_k\Grad\mu \cdot \bpsi \dx
			\label{AGG:exws:eq:time-discr_sol_a_equiv}
		\end{align}
		for all $\bpsi\in \rmC_0^\infty(\Omega)^d$.
	\end{remark}

	\begin{remark}
		\label{AGG:exws:rmk:mean_phi_cons}
		For given $\varphi_k\in \rmH^1(\Omega)$, a solution $(\vel,\lambda,\mu,\varphi)$ to the time-discrete system~\eqref{AGG:exws:eq:time-discr_sol_a}--\eqref{AGG:exws:eq:time-discr_sol_d} fulfills 
		\begin{align*}
			\mean{\varphi} = \mean{\varphi_k}.
		\end{align*}
		This can be verified as follows by taking the spatial integral of \eqref{AGG:exws:eq:time-discr_sol_c} and using \eqref{AGG:exws:eq:time-discr_sol_b}, along with the divergence theorem and the boundary conditions:
		\begin{align*}
			\int_\Omega \frac{\varphi-\varphi_k}{h} \dx
			&= \int_\Omega - \nabla\cdot(\varphi_k\vel)
			+ \nabla\cdot\bigprt{\mj(\varphi_k) \Grad\mu} - \mr(\varphi_k) \bigprt{\mu+\alpha\lambda} \dx \\
			&= \int_\Omega - \nabla\cdot(\varphi_k\vel)
			+ \nabla\cdot\bigprt{\mj(\varphi_k) \Grad\mu} + \frac1\alpha \nabla\cdot\vel \dx
			= 0.
		\end{align*}
	\end{remark}

	\begin{lemma}
		\label{AGG:exws:lemma:estimates}
		Let $\varphi_k \in \rmH^2(\Omega)$ with $\abs{\varphi_k}\leq1$ a.e.~in $\Omega$ be given, and further assume a pair~$(\mu,\varphi) \in \rmH^1(\Omega) \times \mathcal D(\partial E)$ with $\mean{\varphi} = \mean{\varphi_k} \in (-1,1)$ to solve \eqref{AGG:exws:eq:time-discr_sol_d}. Then, there exists a constant $C=C(\mean{\varphi_k})>0$ such that
		\begin{align*}
			\norm{\varphi}_{\rmH^2(\Omega)} + \norm{F_0'(\varphi)}_{\rmL^2(\Omega)} + \Bigabs{\int_\Omega \mu \dx}
			&\leq C \bigprt{\norm{\Grad\mu}_{\rmL^2(\Omega)} + \norm{\Grad\varphi}_{\rmL^2(\Omega)}^2 + \norm{\Grad\varphi_k}_{\rmL^2(\Omega)}^2 + 1} , \\
			\norm{\partial E(\varphi)}_{\rmL^2(\Omega)}
			&\leq C\bigprt{\norm{\mu}_{\rmL^2(\Omega)} +1}.
		\end{align*}
	\end{lemma}

	\begin{proof}		
		Denoting $m\coloneqq \mean{\varphi} = \mean{\varphi_k}$, we first observe that $(\mu,\varphi)$ also solves the mean-free part of equality \eqref{AGG:exws:eq:time-discr_sol_d}, which is
		\begin{align*}
			P_0(\mu) + \kappa\frac{\varphi+\varphi_k}{2} - \kappa m = P_0(F_0'(\varphi)) - \Lap\varphi \quad\text{a.e.~in $\Omega$},
		\end{align*}
		where the right-hand side equals $\partial E_m(\varphi)$ due to identity~\eqref{eq:subgrad_with:subgradient}. Estimate~\eqref{eq:subgrad_with:estimate_H2} thus yields
		\begin{align*}
			\norm{\varphi}_{\rmH^2(\Omega)}^2 + \norm{F_0'(\varphi)}_{\rmL^2(\Omega)}^2 \dx
			&\leq C \bigprt{\norm{\partial E_m(\varphi)}_{\rmL^2(\Omega)}^2 + \norm{\varphi}_{\rmL^2(\Omega)}^2 +1 } \\
			&= C \biggprt{\Bignorm{P_0(\mu) + \kappa\frac{\varphi+\varphi_k}{2} - \kappa m}_{\rmL^2(\Omega)}^2 
				+ \norm{\varphi}_{\rmL^2(\Omega)}^2 +1 } \\
			&\leq C \bigprt{\|\Grad\mu\|_{\rmL^2(\Omega)}^2 + \|\Grad\varphi\|_{\rmL^2(\Omega)}^4 + \|\Grad\varphi_k\|_{\rmL^2(\Omega)}^4 + 1}.
		\end{align*}
		
		In order to show the claimed estimate for $\abs{\int_\Omega\mu\dx}$, we employ \eqref{AGG:exws:eq:time-discr_sol_d} along with the divergence theorem and the boundary condition for $\varphi\in \mathcal D(\partial E)$. Additionally using the above estimate for $\norm{F_0'(\varphi)}_{\rmL^2(\Omega)}$ and the bounds $\abs{\varphi_k}\leq1$, $\abs{\varphi}\leq1$ a.e.~in $\Omega$ entails
		\begin{align*}
			\Bigabs{\int_\Omega\mu\dx}
			&= \Bigabs{\int_\Omega F_0'(\varphi)-\Lap\varphi - \kappa\frac{\varphi+\varphi_k}{2} \dx} \\
			&= \Bigabs{\int_\Omega F_0'(\varphi) - \kappa\frac{\varphi+\varphi_k}{2} \dx + \int_{\partial\Omega} \partial_\n\varphi \vert_{\partial\Omega} \dsix }\\
			&\leq C \norm{F_0'(\varphi)}_{\rmL^2(\Omega)} + \frac\kappa2 \int_\Omega \abs{\varphi}+\abs{\varphi_k} \dx \\ 
			&\leq C \bigprt{ \norm{\Grad\mu}_{\rmL^2(\Omega)} + \norm{\Grad\varphi}_{\rmL^2(\Omega)}^2 + \norm{\Grad\varphi_k}_{\rmL^2(\Omega)}^2 +1 }.
		\end{align*}
		
		Eventually, from equation~\eqref{AGG:exws:eq:time-discr_sol_d} and identity~\eqref{eq:subgrad_without:subgradient}, we infer the estimate
		\begin{align*}
			\norm{\partial E(\varphi)}_{\rmL^2(\Omega)}
			= \norm{F_0'(\varphi)-\Lap\varphi}_{\rmL^2(\Omega)}
			= \Bignorm{\mu+\kappa\frac{\varphi+\varphi_k}{2}}_{\rmL^2(\Omega)}
			\leq C\bigprt{ \norm{\mu}_{\rmL^2(\Omega)} +1},
		\end{align*}
		which finishes the proof.
	\end{proof}

	\begin{lemma}
		\label{AGG:exws:lemma:discr_energy_est}
		Let $\vel_k\in \rmL^2(\Omega)^d$ and $\varphi_k\in \rmH^2(\Omega)$ be given, and set $\rho_k\coloneqq b_+ + b_- \varphi_k$. Then, a solution $(\vel,\lambda,\mu,\varphi) \in \rmH_0^1(\Omega)^d \times \rmL^2(\Omega)\times \rmH^2_\Neum(\Omega) \times  \mathcal D(\partial E) $ to the time-discrete system~\eqref{AGG:exws:eq:time-discr_sol_a}--\eqref{AGG:exws:eq:time-discr_sol_d} satisfies the energy inequality
		\begin{align}
			\label{AGG:exws:eq:discr_energy_ineq}
			&E_\tot(\vel,\varphi)
			+ \int_\Omega \rho_k \frac{\abs{\vel-\vel_k}^2}{2} + \frac{\abs{\Grad\varphi-\Grad\varphi_k}^2}{2} \dx \nonumber\\
			&\quad+ h\int_\Omega \Str(\varphi_k,\D\vel) : \D\vel + \mj(\varphi_k) \abs{\Grad\mu}^2 + \mr(\varphi_k) \bigprt{\mu+\alpha\lambda}^2 \dx \nonumber\\
			&\leq E_\tot(\vel_k,\varphi_k).
		\end{align}
	\end{lemma}

	\begin{proof}
		The strategy of this proof is to test the time-discrete equations suitably. Before testing \eqref{AGG:exws:eq:time-discr_sol_a_equiv} (which is equivalent to \eqref{AGG:exws:eq:time-discr_sol_a}) with $\vel$, we make some observations in order to simplify the resulting equation. First, by means of the divergence theorem, we infer
		\begin{align*}
			\int_\Omega \Bigprt{\nabla\cdot\tilde\J \frac\vel2 + \bigprt{\tilde\J\cdot\nabla}\vel} \cdot \vel \dx
			= \int_\Omega \nabla\cdot \biggprt{\tilde\J \frac{\abs{\vel}^2}{2}} \dx
			= 0
		\end{align*}
		as well as
		\begin{gather*}
			\int_\Omega \Bigprt{\nabla\cdot\bigprt{\rho_k\vel\otimes\vel} - \nabla\cdot(\rho_k\vel)\frac\vel2} \cdot \vel \dx \\
			= \int_\Omega \nabla\cdot\bigprt{\rho_k\vel} \abs{\vel}^2 + \rho_k\vel \cdot \nabla \biggprt{\frac{\abs{\vel}^2}{2}} - \nabla\cdot\bigprt{\rho_k\vel} \frac{\abs{\vel}^2}{2} \dx 
			= \int_\Omega \nabla\cdot \biggprt{\rho_k\vel \frac{\abs{\vel}^2}{2}} \dx
			= 0.
		\end{gather*}
		Moreover, the simple algebraic identity
		\begin{align*}
			\mathbf{a} \cdot (\mathbf{a}-\mathbf{b}) 
			= \frac{\abs{\mathbf{a}}^2}{2} - \frac{\abs{\mathbf{b}}^2}{2} + \frac{\abs{\mathbf{a}-\mathbf{b}}^2}{2}
			\quad\text{for } \mathbf{a},\mathbf{b}\in\R^d
		\end{align*}
		enables us to calculate
		\begin{align*}
			(\rho\vel-\rho_k\vel_k)\cdot\vel
			&= (\rho-\rho_k)\abs{\vel}^2 + \rho_k(\vel-\vel_k)\cdot\vel \\
			&= (\rho-\rho_k)\abs{\vel}^2 + \rho_k\biggprt{\frac{\abs{\vel}^2}{2} - \frac{\abs{\vel_k}^2}{2}} + \rho_k \frac{\abs{\vel-\vel_k}^2}{2} \\
			&= \biggprt{\rho\frac{\abs{\vel}^2}{2} - \rho_k\frac{\abs{\vel_k}^2}{2}} + (\rho-\rho_k)\frac{\abs{\vel}^2}{2}
			+ \rho_k \frac{\abs{\vel-\vel_k}^2}{2}.
		\end{align*}
		Now, using these equations to test \eqref{AGG:exws:eq:time-discr_sol_a} or \eqref{AGG:exws:eq:time-discr_sol_a_equiv}, respectively, with $\vel$  leads us to
		\begin{align*}
			0
			&= \int_\Omega \frac{\rho\abs{\vel}^2 - \rho_k\abs{\vel_k}^2}{2h}
			+ \rho_k \frac{\abs{\vel-\vel_k}^2}{2h}
			+ \Str(\varphi_k,\D\vel) : \D\vel 
			 - \lambda \nabla\cdot\vel + \varphi_k\Grad\mu\cdot\vel \dx.
		\end{align*}
		Testing \eqref{AGG:exws:eq:time-discr_sol_b} with $\lambda$ further gives
		\begin{align*}
			0
			= \int_\Omega \lambda\nabla\cdot\vel 
			+ \mr(\varphi_k) \bigprt{\mu\alpha\lambda+\alpha^2\lambda^2} \dx,
		\end{align*}
		while testing \eqref{AGG:exws:eq:time-discr_sol_c} with $\mu$ results in
		\begin{align*}
			0
			= \int_\Omega \frac{\varphi-\varphi_k}{h}\mu - \varphi_k\Grad\mu\cdot\vel
			+ \mj(\varphi_k) \abs{\Grad\mu}^2 + \mr(\varphi_k) \bigprt{\mu^2+\mu\alpha\lambda} \dx.
		\end{align*}
		Eventually, we choose $\frac{\varphi-\varphi_k}{h}$ as a test function in \eqref{AGG:exws:eq:time-discr_sol_d} and invoke the aforementioned algebraic identity to obtain
		\begin{align*}
			0
			&= \int_\Omega \Grad\varphi \cdot \frac{\Grad\varphi-\Grad\varphi_k}{h} + F_0'(\varphi)\frac{\varphi-\varphi_k}{h} - \mu\frac{\varphi-\varphi_k}{h} - \kappa\frac{\varphi+\varphi_k}{2} \frac{\varphi-\varphi_k}{h} \dx \\
			&= \int_\Omega \frac1h \biggprt{ \frac{\abs{\Grad\varphi}^2}{2} - \frac{\abs{\Grad\varphi_k}^2}{2} + \frac{\abs{\Grad\varphi-\Grad\varphi_k}^2}{2} } 
			+ F_0'(\varphi)\frac{\varphi-\varphi_k}{h} - \mu\frac{\varphi-\varphi_k}{h} - \kappa\frac{\varphi^2-\varphi_k^2}{2h} \dx.
		\end{align*}
		Summing these four identities and also taking into account the convexity of $F_0$ to estimate~$F_0'(\varphi)\prt{\varphi-\varphi_k} \geq F_0(\varphi)-F_0(\varphi_k)$ we conclude
		\begin{align*}
			0
			&\geq \int_\Omega \frac{\rho\abs{\vel}^2 - \rho_k\abs{\vel_k}^2}{2h}
			+ \rho_k \frac{\abs{\vel-\vel_k}^2}{2h} \dx \\
			&\quad + \int_\Omega \Str(\varphi_k,\D\vel) : \D\vel
			+ \mj(\varphi_k) \abs{\Grad\mu}^2 + \mr(\varphi_k) \bigprt{\mu+\alpha\lambda}^2 \dx \\ 
			&\quad + \int_\Omega \frac1h \bigprt{F_0(\varphi)-F_0(\varphi_k)} - \kappa\frac{\varphi^2-\varphi_k^2}{2h} 
			+ \frac1h \biggprt{ \frac{\abs{\Grad\varphi}^2}{2} - \frac{\abs{\Grad\varphi_k}^2}{2} + \frac{\abs{\Grad\varphi-\Grad\varphi_k}^2}{2} } \dx.
		\end{align*}
		Recalling the definition of $E_\tot(\vel,\varphi)$ we have thus established the time-discrete energy inequality~\eqref{AGG:exws:eq:discr_energy_ineq} as claimed.
	\end{proof}

	\begin{lemma}
		\label{AGG:exws:lemma:ex_time-discr}
		Let $\vel_k\in \rmL^2(\Omega)^d$ and $\varphi_k\in \rmH^2(\Omega)$ be given, and set $\rho_k\coloneqq b_+ + b_- \varphi_k$. Then, there exists a solution $(\vel,\lambda,\mu,\varphi) \in \rmH_0^1(\Omega)^d \times  \rmL^2(\Omega) \times \rmH^2_\Neum(\Omega) \times  \mathcal D(\partial E)$ to the time-discrete system~\eqref{AGG:exws:eq:time-discr_sol_a}--\eqref{AGG:exws:eq:time-discr_sol_d}.
	\end{lemma}

	\begin{proof}
		The central concept of this proof hinges on the application of the Leray--Schauder principle. For the preliminary setup, we work with two operators $\mathcal{L}_k,\mathcal{F}_k \colon X\to Y$ where
		\begin{align*}
			X &\coloneqq \rmH_0^1(\Omega)^d \times \rmL^2(\Omega) \times \rmH^2_\Neum(\Omega) \times \mathcal{D}(\partial E), \\
			Y &\coloneqq \rmH^{-1}(\Omega)^d \times \rmL^2(\Omega) \times \rmL^2(\Omega) \times \rmL^2(\Omega).
		\end{align*}
		These are constructed such that a quadruple $\z = (\vel,\lambda,\mu,\varphi) \in X$ solves system~\eqref{AGG:exws:eq:time-discr_sol_a}--\eqref{AGG:exws:eq:time-discr_sol_d} if and only if the identity
		\begin{align*}
			\mathcal{L}_k(\z) = \mathcal{F}_k(\z)
		\end{align*}
		is satisfied for suitable operators $\mathcal{L}_k$ and $\mathcal{F}_k$. After first introducing an operator
		\begin{align*}
			\langle L_k(\vel,\lambda,\mu), \bpsi \rangle 
			\coloneqq \int_\Omega \Str(\varphi_k,\D\vel) : \D\bpsi  - \lambda \nabla\cdot\bpsi + \varphi_k\Grad\mu\cdot\bpsi \dx
			\qquad \text{for all }\bpsi\in \rmH_0^1(\Omega)^d,
		\end{align*}
		we define, for $\z = (\vel,\lambda,\mu,\varphi) \in X$, the aforementioned operators $\mathcal{L}_k,\mathcal{F}_k \colon X\to Y$ by
		\begin{align*}
			\mathcal{L}_k(\z) \coloneqq
			\begin{pmatrix}
				L_k(\vel,\lambda,\mu) \\[1ex]
				\nabla\cdot\vel + \alpha \mr(\varphi_k) \bigprt{\mu+\alpha\lambda} \\[1ex]
				\nabla\cdot (\varphi_k\vel) -\nabla\cdot\bigprt{\mj(\varphi_k) \Grad\mu} + \mr(\varphi_k) \bigprt{\mu+\alpha\lambda} + \int_\Omega \mu \dy \\[1ex]
				\varphi + \partial E(\varphi)
			\end{pmatrix}
		\end{align*}
		and 
		\begin{align*}
			\mathcal{F}_k(\z) \coloneqq
			\begin{pmatrix}
				-\frac{\rho\vel-\rho_k\vel_k}{h} - \nabla\cdot(\rho_k\vel\otimes\vel) - \bigprt{\nabla\cdot\tilde\J -\frac{\rho-\rho_k}{h} - \nabla\cdot(\rho_k\vel)} \frac{\vel}{2} - \bigprt{\tilde\J\cdot\nabla}\vel \\[1ex]
				0 \\[1ex]
				-\frac{\varphi-\varphi_k}{h} + \int_\Omega \mu \dy \\[1ex]
				\varphi + \mu + \kappa\frac{\varphi+\varphi_k}{2}
			\end{pmatrix}.
		\end{align*}
		Our primary goal now is to show that the operator $\mathcal{L}_k\colon X\to Y$ is invertible. To this end, we initially focus on the first three lines of $\mathcal{L}_k$ to define the bilinear form $\mathcal{A}_k$ on the space~$\rmH_0^1(\Omega)^d \times \rmL^2(\Omega) \times \rmH^1(\Omega)$ by
		\begin{gather*}
			\mathcal{A}_k \bigprt{(\vel,\lambda,\mu), (\bpsi,\xi,\zeta)}
			\coloneqq \langle L_k(\vel,\lambda,\mu), \bpsi \rangle 
			+ \int_\Omega \xi \nabla\cdot\vel + \alpha \mr(\varphi_k) \bigprt{\mu+\alpha\lambda} \xi \dx \\
			\quad+ \int_\Omega \zeta \nabla\cdot (\varphi_k\vel) + \mj(\varphi_k) \Grad\mu \cdot\Grad\zeta + \mr(\varphi_k) \bigprt{\mu+\alpha\lambda} \zeta + \int_\Omega \mu \dy \, \zeta \dx.
		\end{gather*}
		We observe that $\mathcal{A}_k$ is both bounded and coercive as
		\begin{align*}
			&\mathcal{A}_k \bigprt{(\vel,\lambda,\mu), (\vel,\lambda,\mu)} 
			= \int_\Omega \Str(\varphi_k,\D\vel) : \D\vel \dx 
			+ \biggprt{\int_\Omega \mu \dx}^2 \\
			&\quad+ \int_\Omega \alpha \mr(\varphi_k) \bigprt{\mu+\alpha\lambda} \lambda \dx 
			+ \int_\Omega \mj(\varphi_k)\abs{\Grad\mu}^2 + \mr(\varphi_k) \bigprt{\mu+\alpha\lambda} \mu \dx \\
			& \geq C \biggprt{ \norm{\vel}_{\rmH^1(\Omega)}^2 + \biggprt{\int_\Omega \mu \dx}^2 + \norm{\Grad\mu}_{\rmL^2(\Omega)}^2 + \norm{\mu+\alpha\lambda}_{\rmL^2(\Omega)}^2} \\
			& \geq C \Bigprt{ \norm{\vel}_{\rmH^1(\Omega)}^2 + \norm{\mu}_{\rmH^1(\Omega)}^2 + \norm{\mu+\alpha\lambda}_{\rmL^2(\Omega)}^2} \\
			&\geq C \Bigprt{ \norm{\vel}_{\rmH^1(\Omega)}^2 + \norm{\lambda}_{\rmL^2(\Omega)}^2 + \norm{\mu}_{\rmH^1(\Omega)}^2}
		\end{align*}
		with a constant $C>0$ due to Korn's and Poincaré's inequalities and the positive lower bound on the coefficients $\nu$, $\eta$, $\mj$ and $\mr$. Thus, the Lax--Milgram theorem ensures that for any right-hand side $(\mathbf g_1,g_2,g_3)\in \rmH^{-1}(\Omega)^d \times \rmL^2(\Omega) \times \bigprt{\rmH^1(\Omega)}'$, there exists a unique triple~$(\vel,\lambda,\mu) \in \rmH_0^1(\Omega)^d \times \rmL^2(\Omega) \times \rmH^1(\Omega)$ fulfilling
		\begin{align*}
			\mathcal{A}_k \bigprt{(\vel,\lambda,\mu), (\bpsi,\xi,\zeta)} 
			= \langle \mathbf g_1, \bpsi \rangle_{\rmH_0^1(\Omega)} +  \langle g_2, \xi \rangle_{\rmL^2(\Omega)} +  \langle g_3, \zeta \rangle_{\rmH^1(\Omega)}
		\end{align*}
		for all $(\bpsi,\xi,\zeta)\in \rmH_0^1(\Omega)^d \times \rmL^2(\Omega) \times \rmH^1(\Omega)$. In particular, for given $g_3\in \rmL^2(\Omega)$, we see that~$\mu\in \rmH^1(\Omega)$ is a weak solution to the elliptic equation
		\begin{alignat*}{2}
			-\nabla\cdot\bigprt{\mj(\varphi_k) \Grad\mu} + \int_\Omega \mu \dy + \mr(\varphi_k) \mu
			&= g_3 - \nabla\cdot (\varphi_k\vel) - \mr(\varphi_k) \alpha\lambda
		\end{alignat*}
		with a homogeneous Neumann boundary condition. In order to prove that $\mu$ is even of regularity $\rmH^2_\Neum(\Omega)$, we note that $\mu$ additionally solves, in a weak sense, the equation
		\begin{alignat*}{2}
			\Lap\mu 
			= -\bigprt{\mj(\varphi_k)}^{\!-1} 
			\Bigprt{ \Grad \bigprt{\mj(\varphi_k)} \cdot \Grad\mu + \int_\Omega \mu \dy + \mr(\varphi_k) \mu
			 - g_3 - \nabla\cdot (\varphi_k\vel) - \mr(\varphi_k) \alpha\lambda }
		\end{alignat*}
		with the same boundary condition as above. Since it holds $\varphi_k\in \rmH^2(\Omega)$ with $\abs{\varphi_k}\leq1$ a.e.~in $\Omega$, we deduce by the chain rule that $\Grad \prt{\mj(\varphi_k)}\in \rmL^6(\Omega)$, and therefore, we conclude~$\Grad \prt{\mj(\varphi_k)} \cdot \Grad\mu \in \rmL^{\frac32}(\Omega)$. With the remaining terms on the right-hand side belonging to $\rmL^2(\Omega)$, it follows that for some function $g \in \rmL^{\frac32}(\Omega)$, the equation takes the form
		\begin{align*}
			\Lap\mu =  g \in \rmL^{\frac32}(\Omega).
		\end{align*}
		Hence, elliptic regularity theory yields $\mu\in \rmW^{2,\frac32}(\Omega)$, and in particular $\Grad\mu\in \rmW^{1,\frac32}(\Omega)$ which embeds into $\rmL^3(\Omega)$. This implies that the product $\Grad \prt{\mj(\varphi_k)} \cdot \Grad\mu$ is even in $\rmL^2(\Omega)$ and thus leads to
		\begin{align*}
			\Lap\mu =  g \in \rmL^2(\Omega).
		\end{align*}
		We consequently obtain the regularity $\mu\in \rmH^2(\Omega)$ along with the estimate
		\begin{align}
			\label{AGG:exws:eq:mu_H2}
			\norm{\mu}_{\rmH^2(\Omega)}
			\leq C\norm{g}_{\rmL^2(\Omega)}
			\leq C_k \bigprt{ \norm{\vel}_{\rmH^1(\Omega)} + \norm{\lambda}_{\rmL^2(\Omega)} + \norm{\mu}_{\rmH^1(\Omega)} + \norm{g_3}_{\rmL^2(\Omega)} }.
		\end{align}
		
		Examining the fourth line in $\mathcal{L}_k$ in a next step, $\partial E$ being a maximal monotone operator reveals that
		\begin{align*}
			I+\partial E \colon \mathcal{D}(\partial E) \to \rmL^2(\Omega)
		\end{align*}
		is invertible. The inverse operator, considered as a mapping 
		\begin{align*}
			\prt{I+\partial E}^{-1} \colon \rmL^2(\Omega)\to \rmH^{2-s}(\Omega), \qquad s\in\bigprt{0,\tfrac14},
		\end{align*}
		is continuous (and even compact) which can be verified using \eqref{eq:subgrad_without:subgradient} and \eqref{eq:subgrad_without:estimate_H2}. For details, we refer to the proof of Lemma~4.3 in \cite{AbelsDepnerGarcke2012}.
		
		Altogether, we have now shown that $\mathcal{L}_k\colon X\to Y$ is invertible with an inverse operator~$\mathcal{L}_k^{-1}\colon Y\to X$. Our next objective is to identify Banach spaces $\tilde X$, $\tilde Y$ such that we have a continuous embedding $X\hookrightarrow\tilde X$ and a compact embedding $\tilde Y \hookrightarrow Y$. This is ensured by introducing
		\begin{align*}
			\tilde X &\coloneqq \rmH_0^1(\Omega)^d \times  \rmL^2(\Omega) \times \rmH^2_\Neum(\Omega) \times \rmH^{2-s}(\Omega), \qquad s\in\bigprt{0,\tfrac14}, \\[1ex]
			\tilde Y &\coloneqq \rmL^{\frac32}(\Omega)^d \times \rmH^1(\Omega) \times \rmW^{1,\frac32}(\Omega) \times \rmH^1(\Omega).
		\end{align*}
		Eventually, $\mathcal{L}_k^{-1}\colon \tilde Y\to \tilde X$ is a compact operator.
		
		We proceed by showing that $\mathcal{F}_k \colon\tilde X\to\tilde Y$ is continuous and maps bounded sets into bounded sets. To this end, we analyze each individual term in $\mathcal{F}_k(\z)$ for $\z=(\vel,\lambda,\mu,\varphi) \in\tilde X$ similarly to \cite{AbelsDepnerGarcke2012} and achieve the following bounds:
		\begin{enumerate}[label=(\roman*)]
			\setlength\itemsep{0.8ex}
			\item $\norm{\rho\vel}_{\rmL^{\frac32}} \leq C\norm{\vel}_{\rmH^1} \bigprt{\norm{\varphi}_{\rmL^2} +1}$. \\
			We have the embedding $\vel\in \rmH^1\hookrightarrow \rmL^6$, and $\rho$ depends affine linearly on $\varphi\in \rmL^2$.
			\item $\norm{\nabla\cdot(\rho_k\vel\otimes\vel)}_{\rmL^{\frac32}} \leq C_k\norm{\vel}_{\rmH^1}^2$. \\
			Since $\rho_k\in \rmH^2$, the occurring terms of the form $\rho_k\partial_l\vel_i\vel_j$ are a product of functions in~$\rmL^\infty$, $\rmL^2$ and $\rmL^6$ whereas the products of the type $\partial_l\rho_k\vel_i\vel_j$ consist of three functions in $\rmL^6$.
			\item $\norm{\nabla\cdot\tilde\J \vel}_{\rmL^{\frac32}} \leq C\norm{\vel}_{\rmH^1}\norm{\mu}_{\rmH^2}$. \\
			In view of $\tilde\J=-\frac{\tilde\rho_+-\tilde\rho_-}{2} \mj(\varphi_k) \Grad\mu$, the resulting expressions include terms such as~$\mj'(\varphi_k) \partial_i\varphi_k \partial_j\mu \vel_l$ which involve multiplied functions in $\rmL^\infty$, $\rmL^6$, $\rmL^6$ and $\rmL^6$, as well as products $\mj(\varphi_k) \partial_i\partial_j\mu \vel_l$ of functions in $\rmL^\infty$, $\rmL^2$ and $\rmL^6$.
			\item $\norm{\nabla\cdot(\varphi_k\vel)\vel}_{\rmL^{\frac32}} \leq C_k\norm{\vel}_{\rmH^1}^2$. \\
			On the one hand, arising terms of the type $\partial_i\varphi_k \vel_j \vel_l$ are products of functions in~$\rmH^{1-s}\hookrightarrow \rmL^3$, $\rmL^6$ and $\rmL^6$. On the other hand, terms of the form $\varphi_k \partial_i\vel_j \vel_l$ contain a product of functions in $\rmL^\infty$, $\rmL^2$ and $\rmL^6$.
			\item $\norm{(\tilde\J\cdot\nabla)\vel}_{\rmL^{\frac32}} \leq C\norm{\vel}_{\rmH^1}\norm{\mu}_{\rmH^2}$. \\
			This expression consists of products $\mj(\varphi_k) \partial_i\mu \partial_j\vel_l$ with factors in $\rmL^\infty$, $\rmL^6$ and $\rmL^2$.
			\item The term $-\frac{\varphi-\varphi_k}{h} + \int_\Omega \mu \dy$ in the third line of $\mathcal{F}_k$ is obviously bounded in $\rmW^{1,\frac32}$.
			\item It is also clear that the fourth line term $\varphi + \mu + \kappa\frac{\varphi+\varphi_k}{2}$ is bounded in $\rmH^1$.
		\end{enumerate}
		
		Now, we aim to apply the Leray--Schauder principle (see Theorem 6.A in \cite{Zeidler1992}, for instance) on the Banach space $\tilde Y$. Finding a solution $\z=(\vel,\lambda,\mu,\varphi) \in X$ to the time-discrete system~\eqref{AGG:exws:eq:time-discr_sol_a}--\eqref{AGG:exws:eq:time-discr_sol_d}, i.e., a solution of $\mathcal{L}_k(\z) = \mathcal{F}_k(\z)$, is equivalent to solving the equation 
		\begin{align*}
			\f = \mathcal{K}_k(\f),
			\qquad\text{where } \mathcal{K}_k \coloneqq \mathcal{F}_k \circ \mathcal{L}_k^{-1} \text{ and } \f=  \mathcal{L}_k(\z). 
		\end{align*}
		Thus, our goal is to establish the existence of a fixed point of $\mathcal{K}_k$. Since $\mathcal{L}_k^{-1}\colon \tilde Y\to \tilde X$ is compact and $\mathcal{F}_k \colon \tilde X \to \tilde Y$ is continuous, $\mathcal{K}_k \colon \tilde Y \to \tilde Y$ is a compact operator. In order to obtain a fixed point of $\mathcal{K}_k$, the Leray--Schauder principle additionally requires to show a further condition: there exists a bound $R>0$ such that the following holds:
		\begin{align}
			\label{AGG:exws:eq:Leray-Schauder_cond}
			\text{If } \f\in \tilde Y \quad\text{and}\quad 0\leq\vartheta\leq1 \text{ fulfills } \f=\vartheta\mathcal{K}_k(\f), \quad\text{then } \norm{\f}_{\tilde Y} \leq R.
		\end{align}
		So let $\f\in\tilde Y$ and $0\leq\vartheta\leq1$ with $\f=\vartheta\mathcal{K}_k(\f)$ be given. Defining $\z=\mathcal{L}_k^{-1}(\f)$ this is equivalent to the equation
		\begin{align*}
			\mathcal{L}_k(\z) = \vartheta\mathcal{F}_k(\z),
		\end{align*}
		which corresponds to the following weak formulation:
		\begin{subequations}
		\begin{align}
			&\int_\Omega \vartheta \frac{\rho\vel-\rho_k\vel_k}{h} \cdot\bpsi 
			+ \vartheta \nabla\cdot(\rho_k\vel\otimes\vel) \cdot \bpsi
			+ \vartheta \Bigprt{\nabla\cdot\tilde\J -\frac{\rho-\rho_k}{h} - \nabla\cdot(\rho_k\vel)} \frac{\vel}{2} \cdot \bpsi \dx \nonumber\\
			&\quad+ \int_\Omega \vartheta \bigprt{\tilde\J\cdot\nabla}\vel \cdot \bpsi
			\dx + \Str(\varphi_k,\D\vel) : \Grad\bpsi - \lambda \nabla\cdot\bpsi \dx 
			= -\int_\Omega \varphi_k\Grad\mu \cdot \bpsi \dx
			\label{AGG:exws:eq:time-discr_sol_LS_a}
		\end{align}
		for all $\bpsi\in \rmH_0^1(\Omega)^d$, and
		\begin{alignat}{2}
			\nabla\cdot\vel 
			&= -\alpha \mr(\varphi_k) \bigprt{\mu+\alpha\lambda}, \label{AGG:exws:eq:time-discr_sol_LS_b}\\
			\vartheta \frac{\varphi-\varphi_k}{h} + \nabla\cdot(\varphi_k\vel) - \vartheta \int_\Omega \mu \dy 
			&= \nabla\cdot\bigprt{\mj(\varphi_k) \Grad\mu} - \mr(\varphi_k) \bigprt{\mu+\alpha\lambda} - \int_\Omega \mu \dy, \label{AGG:exws:eq:time-discr_sol_LS_c}\\
			\vartheta \varphi + \vartheta \mu + \vartheta \kappa\frac{\varphi+\varphi_k}{2} 
			&= \varphi + F_0'(\varphi) -\Lap\varphi, \label{AGG:exws:eq:time-discr_sol_LS_d}
		\end{alignat}
		\end{subequations}
		where the last three equations hold a.e.~in $\Omega$. Proceeding analogously to the proof of Lemma~\ref{AGG:exws:lemma:discr_energy_est}, we now test \eqref{AGG:exws:eq:time-discr_sol_LS_a} with $\vel$, \eqref{AGG:exws:eq:time-discr_sol_LS_b} with $\lambda$, \eqref{AGG:exws:eq:time-discr_sol_LS_c} with $\mu$ and \eqref{AGG:exws:eq:time-discr_sol_LS_d} with $\frac{\varphi-\varphi_k}{h}$.~As in the derivation there, summing the resulting equations and using the convexity of $F_0$~reveals
		\begin{align*}
			0
			&\geq \int_\Omega \vartheta \frac1h \biggprt{ \frac{\rho\abs{\vel}^2}{2} - \frac{\rho_k\abs{\vel_k}^2}{2}
				+ \frac{\rho_k\abs{\vel-\vel_k}^2}{2} } \dx \\
			&\quad + \int_\Omega \Str(\varphi_k,\D\vel) : \D\vel
			+ \mj(\varphi_k) \abs{\Grad\mu}^2 + \mr(\varphi_k) \bigprt{\mu+\alpha\lambda}^2 \dx \\
			&\quad + \int_\Omega \frac1h \biggprt{ \frac{\abs{\Grad\varphi}^2}{2} - \frac{\abs{\Grad\varphi_k}^2}{2} + \frac{\abs{\Grad\varphi-\Grad\varphi_k}^2}{2} }
			+ \frac1h \bigprt{ F_0(\varphi) - F_0(\varphi_k) } - \vartheta \kappa\frac{\varphi^2-\varphi_k^2}{2h} \dx \\
			&\quad + \int_\Omega (1-\vartheta) \frac1h \biggprt{ \frac{\varphi^2}{2} - \frac{\varphi_k^2}{2} + \frac{\prt{\varphi-\varphi_k}^2}{2} } \dx
			+ (1-\vartheta) \biggprt{ \int_\Omega \mu \dy}^2.
		\end{align*}
		Given that $\z=\mathcal{L}_k^{-1}(\f) \in X$, it follows that $\varphi$ belongs to the space $\mathcal{D}(\partial E)$, and in particular that $\abs{\varphi}\leq1$ a.e.~in $\Omega$, which implies $\rho\geq0$. Furthermore, after taking into account~the nonnegativity of the terms $ \rho\abs{\vel}^2$, $\rho_k \abs{\vel-\vel_k}^2$, $\varphi^2$, $\prt{\varphi-\varphi_k}^2$ and $\abs{\Grad\varphi-\Grad\varphi_k}^2$, we employ~$\vartheta,(1-\vartheta)\leq1$, the relation~$F_0(\varphi)-\kappa\frac{\varphi^2}{2} = F(\varphi)$ as well as the estimate $-\vartheta \kappa\varphi_k^2 \leq \abs{\kappa}\varphi_k^2 $. Since $\abs{\varphi}\leq1$ a.e.~in $\Omega$ ensures that the term $\int_\Omega F(\varphi) \dx$ is bounded, we may move it to the right-hand side. In summary, this leads to
		\begin{align*}
			&\int_\Omega h \Str(\varphi_k,\D\vel) : \D\vel
			+ h \mj(\varphi_k) \abs{\Grad\mu}^2 + h \mr(\varphi_k) \bigprt{\mu+\alpha\lambda}^2 \dx \\
			&\quad + \int_\Omega \frac{\abs{\Grad\varphi}^2}{2} \dx
			+ (1-\vartheta) h\biggprt{ \int_\Omega \mu \dy}^2 \\
			&\leq
			\int_\Omega \frac{\rho_k\abs{\vel_k}^2}{2} + \frac{\varphi_k^2}{2} + \frac{\abs{\Grad\varphi_k}^2}{2}  + F_0(\varphi_k) + \abs{\kappa}\frac{\varphi_k^2}{2} \dx
			+ \Bigabs{ \int_\Omega F(\varphi) \dx } 
			\leq C_k.
		\end{align*}
		In view of $\abs{\varphi}\leq1$ a.e.~in $\Omega$, Korn's inequality for $\vel\in \rmH^1_0(\Omega)^d$, and the positive lower bound on $\nu, \mj, \mr$, we arrive at
		\begin{align}
			\label{discr:4.21}
			\norm{\vel}_{\rmH^1(\Omega)} + \norm{\Grad\mu}_{\rmL^2(\Omega)} +  \norm{\mu+\alpha\lambda}_{\rmL^2(\Omega)}
			+ \norm{\varphi}_{\rmH^1(\Omega)} + \prt{1-\vartheta}^{\frac12} \Bigabs{ \int_\Omega \mu \dy}
			\leq C_k.
		\end{align}
		
		In a subsequent step, we control the $\rmL^2$-Norm of $\mu$ by considering two distinct cases. On the one hand, for $\vartheta\in [\frac12,1]$, we proceed as in the proof of Lemma \ref{AGG:exws:lemma:estimates}, with \eqref{AGG:exws:eq:time-discr_sol_d} replaced by \eqref{AGG:exws:eq:time-discr_sol_LS_d}, to get an estimate
		\begin{align*}
			\tfrac12 \Bigabs{\int_\Omega\mu\dx}
			\leq \vartheta \Bigabs{\int_\Omega\mu\dx}
			\leq C \bigprt{ \norm{\Grad\mu}_{\rmL^2(\Omega)} + \norm{\Grad\varphi}_{\rmL^2(\Omega)}^2 + \norm{\Grad\varphi_k}_{\rmL^2(\Omega)}^2 +1 }
			\leq C_k,
		\end{align*}
		where the right-hand side is bounded due to \eqref{discr:4.21}. On the other hand, if $\vartheta\in[0,\frac12)$, a direct use of \eqref{discr:4.21} implies $\abs{\int_\Omega\mu\dx} \leq C_k$.
		The control over both $\abs{\int_\Omega\mu\dx}$ and $\norm{\Grad\mu}_{\rmL^2(\Omega)}$ enables us to bound $\mu$ in $\rmH^1(\Omega)$ and therefore, given that $\norm{\mu+\alpha\lambda}_{\rmL^2(\Omega)}\leq C_k$, also $\lambda$ in~$\rmL^2(\Omega)$. Incorporating estimate \eqref{AGG:exws:eq:mu_H2}, where now $g_3=-\vartheta\frac{\varphi-\varphi_k}{h} + \vartheta\int_\Omega \mu \dy$, refines \eqref{discr:4.21} to
		\begin{align*}
			\norm{\vel}_{\rmH^1(\Omega)} + \norm{\lambda}_{\rmL^2(\Omega)} + \norm{\mu}_{\rmH^2(\Omega)} + \norm{\varphi}_{\rmH^1(\Omega)}
			\leq C_k.
		\end{align*}
	
		In combination with equation \eqref{AGG:exws:eq:time-discr_sol_LS_d}, this also entails that $\norm{\partial E(\varphi)}_{\rmL^2} \leq C_k$, from where estimate \eqref{eq:subgrad_without:estimate_H2} further yields the bound $\norm{\varphi}_{\rmH^2(\Omega)}\leq C_k$. Altogether, we have now shown that~$\z=(\vel,\lambda,\mu,\varphi) \in \tilde X = \rmH_0^1(\Omega)^d \times \rmL^2(\Omega) \times \rmH^2_\Neum(\Omega) \times \rmH^{2-s}(\Omega)$, where $s\in(0,\frac14)$ is arbitrary, satisfies the estimate
		\begin{align*}
			\norm{\z}_{\tilde X}
			\leq C_k.
		\end{align*}
		Eventually, since $\f\in\tilde Y$ fulfills $\f=\vartheta\mathcal{K}_k(\f) =  \vartheta\mathcal{F}_k \mathcal{L}_k^{-1} (\f) =  \vartheta\mathcal{F}_k (\z)$ and $\mathcal{F}_k \colon \tilde X\to\tilde Y$ maps bounded sets into bounded sets, we conclude
		\begin{align*}
			\norm{\f}_{\tilde Y}
			= \norm{\vartheta\mathcal{F}_k (\z)}_{\tilde Y}
			\leq C_k\bigprt{ \norm{\z}_{\tilde X} +1 }
			\leq C_k,
		\end{align*}
		which verifies the Leray--Schauder condition as required. 
		
		Hence, the operator $\mathcal{K}_k$ has a fixed point $\f\in\tilde Y \hookrightarrow Y$, and $\z=\mathcal{L}_k^{-1}(\f) \in X$ solves the equation $\mathcal{L}_k(\z) = \mathcal{F}_k(\z)$.
		For each $k\in\N$, we have thus established the existence of a solution~$\z=(\vel,\lambda,\mu,\varphi) = (\vel_{k+1},\lambda_{k+1},\mu_{k+1},\varphi_{k+1}) \in X$ to the time-discrete system~\eqref{AGG:exws:eq:time-discr_sol_a}--\eqref{AGG:exws:eq:time-discr_sol_d}.
		\end{proof}

	\subsection{Proof of Theorem \ref{thm:AGG}} \label{AGG:exws:subsec:proof_main}
	
	In this section, we present a proof of one of our main results, Theorem~\ref{thm:AGG}. Starting from the time-discrete solutions constructed in the previous subsection, we perform a limit passage in the corresponding time-continuous system based on compactness results.
	
	\begin{proof}
		For given $N\in\N$, let $(\vel_{k+1},\lambda_{k+1},\mu_{k+1},\varphi_{k+1})$, $k\in\N_0$, be iteratively chosen as a solution to the time-discrete system \eqref{AGG:exws:eq:time-discr_sol_a}--\eqref{AGG:exws:eq:time-discr_sol_d} with time step size $h=\frac1N$ and initial value~$(\vel_0,\varphi_0^N)$, where $\varphi_0^N$ is specified below. The existence of such solutions, which additionally satisfy the energy inequality~\eqref{AGG:exws:eq:discr_energy_ineq}, is ensured by Lemma~\ref{AGG:exws:lemma:discr_energy_est} and Lemma~\ref{AGG:exws:lemma:ex_time-discr}. To apply these lemmas to the first step, we approximate the original initial value~$\varphi_0$ by functions $\varphi_0^N \in \rmH^2(\Omega)$. For this construction, we use standard partial differential equation techniques, in particular the solution $u$ of the heat equation
		\begin{equation*}
			\arraycolsep=2pt
			\begin{array}{rclll}
			\delt u - \Lap u &=& 0 &\text{ in } Q, \\[0.25ex]
			\partial_\n u \vert_{\partial\Omega} &=&0 &\text{ on } S, \\[0.25ex]
			u \vert_{t=0} &=& \varphi_0 &\text{ in } \Omega.
			\end{array}
		\end{equation*}
		Setting $\varphi_0^N \coloneqq u\vert_{t=\frac1N}$ ensures $\varphi_0^N \in \rmH^2(\Omega)$ as well as $\abs{\varphi_0^N}\leq 1$ in $\Omega$ and $\varphi_0^N\to\varphi_0$ in $\rmH^1(\Omega)$ as desired.
		
		To start, we fix the following notation for $f\in\{\vel, \lambda, \mu,\varphi\}$. On $[-h,\infty)\times\Omega$, we define the piecewise constant interpolant of the time-discrete functions by 
		\begin{align*}
			f^N(t) \coloneqq f_k \qquad\text{for } t\in[(k-1)h,kh), \quad k\in \N_0,
		\end{align*}
		and additionally, we set $\rho^N \coloneqq \frac12(\tilde\rho_+ + \tilde\rho_-) + \frac12(\tilde\rho_+ - \tilde\rho_-)\varphi^N$. For all $k\in\N_0$, it consequently holds $f^N((k-1)h)=f_k$ and $f^N(kh)=f_{k+1}$ as well as $f^N(t)=f_{k+1}$ for $t\in[kh,(k+1)h)$. Moreover, we introduce notation for time differences, difference quotients and time shifts, namely
		\begin{gather*}
			(\Delta_h^+ f)(t) \coloneqq f(t+h)-f(t), \qquad (\Delta_h^- f)(t) \coloneqq f(t)-f(t-h), \\
			\partial_{t,h}^+ f(t) \coloneqq \frac1h (\Delta_h^+ f)(t), \qquad \partial_{t,h}^- f(t) \coloneqq \frac1h (\Delta_h^- f)(t), \\
			f_h(t) \coloneqq f(t-h).
		\end{gather*}
	
		From the time-discrete system \eqref{AGG:exws:eq:time-discr_sol_a}--\eqref{AGG:exws:eq:time-discr_sol_d}, we derive in the following the corresponding time-continuous equations, where we use the discrete integration by parts formula
		\begin{align}
			\label{AGG:exws:eq:discr_int_by_parts}
			\int_Q \partial_{t,h}^-\phi\xi \dtx = - \int_Q \phi \partial_{t,h}^+\xi \dtx,
		\end{align} 
		which holds for all $\phi\in \rmL^2(-h,\infty;\rmL^2(\Omega))$, $\xi\in \rmC_0^\infty((0,\infty);\rmL^2(\Omega))$, provided that $h>0$ is sufficiently small. 
		Therefore, for arbitrary $\bpsi\in \rmC_0^\infty(Q)^d$, testing \eqref{AGG:exws:eq:time-discr_sol_a} with $\tilde\bpsi \coloneqq \int_{kh}^{(k+1)h} \bpsi \dt$ and summing over $k\in \N_0$ yields
		\begin{subequations}
		\begin{align}
			&\int_Q - \rho^N\vel^N \cdot \partial_{t,h}^+ \bpsi - \bigprt{\rho_h^N\vel^N\otimes\vel^N} : \Grad\bpsi - \bigprt{\vel^N\otimes\tilde\J^N} : \Grad \bpsi \dtx \nonumber\\
			&\quad+ \int_Q \Str \bigprt{\varphi_h^N,\D\vel^N} : \D\bpsi  - \lambda^N \nabla\cdot\bpsi \dtx
			= -\int_Q \varphi_h^N\Grad\mu^N \cdot \bpsi \dtx
			\label{AGG:exws:eq:time-cont_sol_a}
		\end{align}
		for all $\bpsi\in \rmC_0^\infty(Q)^d$ if $h$ is small enough. We further obtain directly
		\begin{align}
			\nabla\cdot\vel^N = -\alpha \mr\bigprt{\varphi_h^N} \bigprt{\mu^N+\alpha\lambda^N} \quad\text{a.e.~in } Q, 
			\label{AGG:exws:eq:time-cont_sol_b}
		\end{align}
		whereas, for sufficiently small $h$, following the same testing procedure as in \eqref{AGG:exws:eq:time-cont_sol_a} leads to
		\begin{align}
			&\int_Q -\varphi^N \partial_{t,h}^+\zeta - \varphi_h^N\vel^N \cdot\Grad\zeta \dtx \nonumber\\
			&= \int_Q -\bigprt{\mj\bigprt{\varphi_h^N} \Grad\mu^N} \cdot \Grad\zeta - \mr\bigprt{\varphi_h^N} \bigprt{\mu^N+\alpha\lambda^N} \zeta \dtx
			\label{AGG:exws:eq:time-cont_sol_c}
		\end{align}
		for all $\zeta\in \rmC_0^\infty((0,\infty)\times\overline\Omega)$. Finally, it holds
		\begin{align}
			\mu^N + \kappa\frac{\varphi^N+\varphi_h^N}{2}= F_0'\bigprt{\varphi^N}-\Lap\varphi^N \quad\text{a.e.~in } Q.
			\label{AGG:exws:eq:time-cont_sol_d}
		\end{align}
		\end{subequations}
	
		Next, concerning the total energy of the system, let $E_N$ denote the piecewise linear interpolant of $E_\tot(\vel_k,\varphi_k)$ at time $t_k=kh$, defined as
		\begin{align}
			\label{AGG:exws:eq:E_N}
			E_N(t)
			\coloneqq \frac{(k+1)h-t}{h} E_\tot(\vel_k,\varphi_k) + \frac{t-kh}{h} E_\tot(\vel_{k+1},\varphi_{k+1})
		\end{align}
		for $t\in[t_k,t_{k+1})$ where $k\in\N_0$. For all $t\in(t_k,t_{k+1})$, $k\in\N_0$, we further define the dissipation
		\begin{align}
			\label{AGG:exws:eq:D_N}
			D_N(t) \!
			\coloneqq \!\! \int_\Omega \! \Str(\varphi_k,\D\vel_{k+1}) : \D\vel_{k+1} 
			+ \mj(\varphi_k) \abs{\Grad\mu_{k+1}}^2 + \mr(\varphi_k) \prt{\mu_{k+1}+\alpha\lambda_{k+1}}^2 \dx.
		\end{align}
		From the discrete energy inequality \eqref{AGG:exws:eq:discr_energy_ineq}, we obtain
		\begin{align*}
			-\ddt E_N(t) 
			= \frac{E_\tot(\vel_k,\varphi_k) - E_\tot(\vel_{k+1},\varphi_{k+1})}{h}
			\geq D_N(t)
		\end{align*}
		for all $t\in[t_k,t_{k+1})$ with $k\in\N_0$. On the one hand, we multiply this inequality by a function~$\sigma\in \rmW^{1,1}(0,\infty)$ with $\sigma\geq0$ and integrate by parts to get
		\begin{align}
			\label{AGG:exws:eq:energy_ineq_D_N}
			E_\tot\bigprt{\vel_0^N,\varphi_0}\sigma(0) + \int_0^\infty E_N(t)\sigma'(t) \dt
			\geq \int_0^\infty D_N(t)\sigma(t)\dt.
		\end{align}
		On the other hand, directly integrating the above inequality entails the following energy inequality for the time-continuous system, which reads
		\begin{align}
			\label{AGG:exws:eq:time-cont_energy}
			&E_\tot \bigprt{\vel^N(t),\varphi^N(t)} + \int_{Q_{(s,t)}} \Str\bigprt{\varphi_h^N,\D\vel^N} : \D\vel^N 
			+ \mj\bigprt{\varphi_h^N} \bigabs{\Grad\mu^N}^2 \dtaux \nonumber \\
			&\quad + \int_{Q_{(s,t)}} \mr\bigprt{\varphi_h^N} \bigprt{\mu^N+\alpha\lambda^N}^2 \dtaux 
			\leq E_\tot \bigprt{\vel^N(s),\varphi^N(s)}
		\end{align}
		for all $0\leq s\leq t<\infty$ with $s,t\in h\N_0$.
		
		From \eqref{AGG:exws:eq:time-cont_energy}, invoking Korn's inequality for $\vel^N$ and the fact that $\mean{\varphi^N}$ is constant, we infer that
		\begin{align*}
			\bigprt{\vel^N, \varphi^N} &\text{ is bounded in } \rmL^\infty \bigprt{ 0,\infty;\rmL^2(\Omega)^d \times \rmH^1(\Omega) }, \\
			\bigprt{\vel^N, \Grad\mu^N, \mu^N+\alpha\lambda^N} &\text{ is bounded in } \rmL^2 \bigprt{ 0,\infty;\rmH^1(\Omega)^d \times \rmL^2(\Omega)^d \times \rmL^2(\Omega) }.
		\end{align*}
		Furthermore, in combination with Lemma~\ref{AGG:exws:lemma:estimates}, we conclude
		\begin{align*}
			\int_0^T \Bigabs{\int_\Omega \mu^N \dx} \dt \leq M(T) \quad\text{for all }0<T<\infty \text{ with } M\colon\R_+\to\R_+ \text{ monotone}.
		\end{align*}
		Thus, we are in a position to pass to non-relabeled subsequences converging weakly (or weakly-*, respectively) as $N\to\infty$ as follows:
		\begin{alignat*}{2}
			\vel^N &\overset{(\ast)\;\,}{\rightharpoonup} \vel &&\quad\text{in } \rmL^2(0,\infty;\rmH^1(\Omega)^d) \cap \rmL^\infty(0,\infty;\rmL^2(\Omega)^d), \\
			\varphi^N &\overset{\ast}{\rightharpoonup} \varphi &&\quad\text{in } \rmL^\infty(0,\infty;\rmH^1(\Omega)), \\
			\Grad\mu^N &\rightharpoonup \Grad\mu &&\quad\text{in } \rmL^2(0,\infty;\rmL^2(\Omega)^d), \\
			\mu^N &\rightharpoonup \mu &&\quad\text{in } \rmL^2(0,T;\rmH^1(\Omega)) \text{ for all }0<T<\infty, \\
			\lambda^N &\rightharpoonup \lambda &&\quad\text{in } \rmL^2(0,T;\rmL^2(\Omega)) \text{ for all }0<T<\infty.
		\end{alignat*}
		
		In order to realize the limit passage in all non-linearities, we additionally show strong convergence of $\varphi^N$ and $\vel^N$ in the following. We start by defining $\tilde\varphi^N$ as the piecewise linear interpolant of $\varphi^N(t_k)$ where $t_k=kh$ and $k\in\N_0$. The time derivative of $\tilde\varphi^N$ is then given by a difference quotient, specifically
		\begin{align}
			\label{AGG:exws:eq:time_derivative_lin_interpolant}
			\delt\tilde\varphi^N 
			= \partial_{t,h}^-\varphi^N
			\quad\text{for almost all } t\in(0,\infty),
		\end{align}
		and the difference of the two interpolants may be estimated by
		\begin{align}
			\label{AGG:exws:eq:diff_of_interpolants}
			\bignorm{\tilde\varphi^N-\varphi^N}_{\rmH^{-1}(\Omega)}
			\leq h\bignorm{\delt\tilde\varphi^N}_{\rmH^{-1}(\Omega)}.
		\end{align}
		In view of the discrete integration by parts formula \eqref{AGG:exws:eq:discr_int_by_parts}, equation~\eqref{AGG:exws:eq:time-cont_sol_c} can be rewritten as
		\begin{align*}
			&\int_Q \partial_{t,h}^-\varphi^N\zeta - \varphi_h^N\vel^N \cdot\Grad\zeta \dtx \\
			&= \int_Q -\bigprt{\mj\bigprt{\varphi_h^N} \Grad\mu^N} \cdot \Grad\zeta - \mr\bigprt{\varphi_h^N} \bigprt{\mu^N+\alpha\lambda^N} \zeta \dtx
		\end{align*}
		for all $\zeta\in \rmC_0^\infty((0,\infty)\times\overline\Omega)$. Studying the latter, along with identity \eqref{AGG:exws:eq:time_derivative_lin_interpolant}, reveals that~$\delt\tilde\varphi^N$ is bounded in $\rmL^2(0,\infty;\rmH^{-1}(\Omega))$ since the terms $\varphi_h^N\vel^N$, $\Grad\mu^N$ and $\mu+\alpha\lambda^N$ are bounded in~$\rmL^2(0,\infty;\rmL^2(\Omega)^d)$ and $\rmL^2(0,\infty;\rmL^2(\Omega))$, respectively. Moreover, observing that the boundedness of $\varphi^N$ in~$\rmL^\infty(0,\infty;\rmH^1(\Omega))$ implies boundedness of $\tilde\varphi^N$ in the same space, the Aubin--Lions lemma yields the strong convergence
		\begin{align*}
			\tilde\varphi^N \to \tilde\varphi \quad\text{in } \rmL^2(0,T;\rmL^2(\Omega)) \text{ for all } 0<T<\infty
		\end{align*}
		for some function $\tilde\varphi\in \rmL^\infty(0,\infty;\rmL^2(\Omega))$. In particular it follows $\tilde\varphi^N \to \tilde\varphi$ a.e.~in $Q$ along a subsequence. Furthermore, estimate~\eqref{AGG:exws:eq:diff_of_interpolants} entails that
		\begin{align*}
			\tilde\varphi^N - \varphi^N \to 0 \quad\text{in } \rmL^2(0,\infty;\rmH^{-1}(\Omega))
		\end{align*}
		and thus $\tilde\varphi=\varphi$. Since $\tilde\varphi^N \to \varphi$ in $\rmL^2(0,T;\rmL^2(\Omega))$ for all $0<T<\infty$ and both $\tilde\varphi^N$ and $\varphi^N$ are bounded in $\rmL^2(0,T;\rmH^1(\Omega))$ we further conclude by \eqref{AGG:exws:eq:diff_of_interpolants} and an interpolation argument that
		\begin{align}
			\label{AGG:exws:eq:strong_conv_phi_L2}
			\varphi^N \to \varphi \quad\text{in } \rmL^2(0,T;\rmL^2(\Omega)) \text{ for all } 0<T<\infty.
		\end{align}
		Moreover, we employ the embedding
		\begin{align*}
			\rmH^1_\uloc([0,\infty);\rmH^{-1}(\Omega)) \cap \rmL^2_\uloc([0,\infty);\rmH^1(\Omega)) 
			\hookrightarrow \BC([0,\infty);\rmL^2(\Omega)),
		\end{align*}
		where this uniformly local variant can be derived immediately from Theorem~4.10.2, Chapter~III in \cite{Amann1995}. The boundedness of $\tilde\varphi^N$ in this space and in $\rmL^\infty(0,\infty;\rmH^1(\Omega))$ establishes~$\varphi\in \BCw([0,\infty);\rmH^1(\Omega))$, which can be verified by applying Lemma~4.1 in \cite{Abels2009g} to arbitrary intervals $[0,T]$ with $0<T<\infty$.
		
		Combining the fact that $\tilde\varphi^N$ is bounded in $\rmH^1 (0,\infty;\rmH^{-1}(\Omega))$ and in $\rmL^\infty(0,\infty;\rmH^1(\Omega))$ with the convergence $\tilde\varphi^N \to \varphi$ in $\rmL^2(0,T;\rmL^2(\Omega))$ for all $0<T<\infty$, we further deduce that~$\tilde\varphi^N\vert_{t=0}\to\varphi\vert_{t=0}$ in $\rmL^2(\Omega)$. In view of $\tilde\varphi^N\vert_{t=0} = \varphi_0^N\to\varphi_0$ in $\rmL^2(\Omega)$, we ultimately arrive at the identity $\varphi\vert_{t=0}=\varphi_0$.
		
		The affine linear dependence of $\rho^N$ on $\varphi^N$ enables us to derive analogous results for $\rho^N$.
		
		In order to pass to the limit in \eqref{AGG:exws:eq:time-cont_sol_d} now, we first point out that the right-hand side is given by $\partial E\bigprt{\varphi^N}$ which is bounded in $\rmL^2(0,T;\rmL^2(\Omega))$ for all $0<T<\infty$ according to Lemma~\ref{AGG:exws:lemma:estimates}. On the other hand, the left-hand side term converges weakly to $\mu+\kappa\varphi$ in~$\rmL^2(0,T;\rmL^2(\Omega))$ for all $0<T<\infty$. Consequently, it holds $\partial E(\varphi^N) \rightharpoonup \mu+\kappa\varphi$ in the same space, from where we conclude, along with \eqref{AGG:exws:eq:strong_conv_phi_L2}, that
		\begin{align*}
			\big\langle \partial E\bigprt{\varphi^N}, \varphi^N \big\rangle
			\to \big\langle \mu+\kappa\varphi,\varphi \big\rangle.
		\end{align*}
		With $\partial E$ being a maximal monotone operator, Proposition IV.1.6 in \cite{Showalter1997} hence establishes~${\partial E(\varphi) = \mu+\kappa\varphi}$, thereby eventually proving equation~\eqref{AGG:exws:eq:weak_sol_d}. Since this identity additionally involves $\partial E(\varphi)\in \rmL^2_\uloc([0,\infty);\rmH^1(\Omega))$, we are finally in a position to use estimate~\eqref{eq:subgrad_without:estimate_W2p} which, due to the continuous embedding $\rmH^1(\Omega) \hookrightarrow \rmL^p(\Omega)$ for $p=6$ if~$d=3$ and $p\in[2,\infty)$ if $d=2$, implies the claimed regularities $\varphi\in \rmL^2_\uloc([0,\infty);\rmW^{2,p}(\Omega))$ as well as~$F_0'(\varphi)\in \rmL^2_\uloc([0,\infty);\rmL^p(\Omega))$.
		
		Now, we aim to show an even better strong convergence for $\varphi^N$. To this end, we study Lemma~\ref{AGG:exws:lemma:estimates} to obtain boundedness of $\varphi^N$ in $\rmL^2(0,T;\rmH^2(\Omega))$ for all $0<T<\infty$. Together with \eqref{AGG:exws:eq:strong_conv_phi_L2}, an interpolation argument verifies
		\begin{align}
			\label{AGG:exws:eq:strong_conv_phi_H1}
			\varphi^N\to\varphi 
			\quad\text{in } \rmL^2(0,T;\rmH^1(\Omega)) \text{ for all } 0<T<\infty.
		\end{align}
		
		Our next goal is to prove strong convergence of $\vel^N$ in $\rmL^2(0,T;\rmL^2(\Omega)^d)$ for all $0<T<\infty$. With $\widetilde{\rho\vel}^N$ denoting the piecewise linear interpolant of $(\rho^N\vel^N)(t_k)$ where $t_k=kh$, $k\in\N_0$, it holds as previously that
		\begin{align*}
			\delt\bigprt{\widetilde{\rho\vel}^N} = \partial_{t,h}^-\bigprt{\rho^N\vel^N}.
		\end{align*}
		We rewrite \eqref{AGG:exws:eq:time-cont_sol_a} with the help of discrete integration by parts formula \eqref{AGG:exws:eq:discr_int_by_parts} into
		\begin{align*}
			&\int_Q \partial_{t,h}^-\bigprt{\rho^N\vel^N} \cdot \bpsi - \bigprt{\rho_h^N\vel^N\otimes\vel^N} : \Grad\bpsi - \bigprt{\vel^N\otimes\tilde\J^N} : \Grad \bpsi \dtx \nonumber\\
			&\quad+ \int_Q \Str \bigprt{\varphi_h^N,\D\vel^N} : \D\bpsi  - \lambda^N \nabla\cdot\bpsi \dtx
			= -\int_Q \varphi_h^N\Grad\mu^N \cdot \bpsi \dtx
		\end{align*}
		for all $\bpsi\in \rmC_0^\infty(Q)^d$. In the following, we study the single terms therein to show that~$\delt\bigprt{\widetilde{\rho\vel}^N}$ is bounded in $\rmL^{\frac87}(0,T;\rmW^{-1,\frac43}(\Omega)^d)$ for arbitrary $0<T<\infty$. 
		
		\begin{enumerate}[label=(\roman*)]
			\item $\rho_h^N\vel^N\otimes\vel^N$ is bounded in $\rmL^2(\rmL^3)$.\\
			Since $\vel^N$ is bounded in $\rmL^\infty(\rmL^2)$ and $\rmL^2(\rmH^1)\hookrightarrow \rmL^2(\rmL^6)$ and $\rho_h^N$ in $\rmL^\infty(Q)$ we obtain boundedness of the occurring products $\rho_h^N\vel_i^N\vel_j^N$ in the desired space $\rmL^2(\rmL^3)$.
			\item $\vel^N\otimes\tilde\J^N$ is bounded in $\rmL^{\frac87}(\rmL^{\frac43})$.\\
			Given that $\tilde\J^N=-b_- \mj(\varphi_h^N)\Grad\mu^N$, it suffices to consider $\vel^N\otimes\Grad\mu^N$. For the~resulting terms of the form $\vel_i^N\partial_j\mu^N$, combining $\vel^N$ being bounded in~$\rmL^2(\rmL^6) \cap \rmL^\infty(\rmL^2)$ and $\Grad\mu^N$ in $\rmL^2(\rmL^2)$, implies boundedness in the spaces $\rmL^2(\rmL^1)$ and $\rmL^1(\rmL^\frac32)$. An interpolation argument then leads to boundedness of these products in $\rmL^{\frac87}(\rmL^{\frac43})$.
			\item $\Str(\varphi_h^N,\D\vel^N)$, $\lambda^N$ and $\varphi_h^N\Grad\mu^N$ are obviously bounded in $\rmL^2(\rmL^2)$.
		\end{enumerate}
		Thus, in particular, all the terms above are bounded in $\rmL^{\frac87}(0,T;\rmL^{\frac43}(\Omega))$ in their respective dimension. As a consequence, it is also possible to test \eqref{AGG:exws:eq:time-cont_sol_a} with functions satisfying both~$\bpsi\in \bigprt{\rmL^{\frac87}(0,T;\rmL^{\frac43}(\Omega)^d)}' = \rmL^8(0,T;\rmL^4(\Omega)^d)$ and $\Grad\bpsi\in \rmL^8(0,T;\rmL^4(\Omega)^{d\times d})$ as well as $\bpsi\vert_{\partial\Omega}=0$, i.e.~with $\bpsi \in \rmL^8(0,T;\rmW_0^{1,4}(\Omega)^d)$. This eventually ensures that $\delt\bigprt{\widetilde{\rho\vel}^N}$ is bounded in the desired space $\bigprt{\rmL^8(0,T;\rmW_0^{1,4}(\Omega)^d)}' = \rmL^{\frac87}(0,T;\rmW^{-1,\frac43}(\Omega)^d)$. 
		
		From this fact, along with the boundedness of $\widetilde{\rho\vel}^N$ in $\rmL^2(0,T;\rmH^1(\Omega)^d)$ for all $0<T<\infty$, which follows from the boundedness of $\rho^N$ in $\rmL^\infty(Q)$ and of $\vel^N$ in $\rmL^2(0,\infty;\rmH^1(\Omega)^d)$, and the boundedness of both $\Grad\rho^N$ and $\vel^N$ in $\rmL^2(0,T;\rmH^1(\Omega)^d) \cap \rmL^\infty(0,\infty;\rmL^2(\Omega)^d)$ for all $0<T<\infty$ which embeds into $\rmL^4(0,T;\rmL^4(\Omega)^d)$,
		we deduce in view of the Aubin--Lions lemma the strong convergence
		\begin{align*}
			\widetilde{\rho\vel}^N \to \tilde\vel 
			\quad\text{in } \rmL^2(0,T;\rmL^2(\Omega)^d) \text{ for all } 0<T<\infty
		\end{align*}
		for some $\tilde\vel\in \rmL^\infty(0,\infty;\rmL^2(\Omega)^d)$. Since $\widetilde{\rho\vel}^N \rightharpoonup \rho\vel$ in $\rmL^2(0,T;\rmL^2(\Omega)^d)$ it follows $\tilde\vel=\rho\vel$. For every $0<T<\infty$, combining the strong convergence $\widetilde{\rho\vel}^N \to \rho\vel$ in $\rmL^2(0,T;\rmL^2(\Omega)^d)$ and the weak convergence $\vel^N\rightharpoonup\vel$ in the same space, we obtain
		\begin{align*}
			\int_{Q_{(0,T)}} \rho^N \bigabs{\vel^N}^2 \dtx
			\to \int_{Q_{(0,T)}} \rho \abs{\vel}^2 \dtx,
		\end{align*}
		and as a consequence $\prt{\rho^N}^\frac12\vel^N \to \rho^\frac12\vel$ in $\rmL^2(0,T;\rmL^2(\Omega)^d)$ for all $0<T<\infty$. Recalling that $\rho^N\to\rho$ a.e.~in $Q$ due to \eqref{AGG:exws:eq:strong_conv_phi_L2} and that $\rho^N$ is bounded away from zero we conclude
		\begin{align}
			\label{AGG:exws:eq:strong_conv_vel}
			\vel^N 
			= \bigprt{\rho^N}^{-\frac12} \Bigprt{\bigprt{\rho^N}^\frac12\vel^N}
			\to \vel
			\quad\text{in } \rmL^2(0,T;\rmL^2(\Omega)^d) \text{ for all } 0<T<\infty
		\end{align}
		as desired. In particular, we establish the convergence $\vel^N\to\vel$ a.e.~in $Q$ along a subsequence.
		
		Eventually, by means of the proven convergence results, passing to the limit to verify the equations~\eqref{AGG:exws:eq:weak_sol_a}--\eqref{AGG:exws:eq:weak_sol_c} becomes possible.
		
		Moreover, following a similar argumentation as for proving $\varphi\in \BCw([0,\infty);\rmH^1(\Omega))$, we obtain that the boundedness of $\widetilde{\rho\vel}^N$ in 
		\begin{align*}
			\rmW^{1,\frac87}_\uloc([0,\infty);\rmW^{-1,4}(\Omega)^d) \hookrightarrow \BC([0,\infty);\rmW^{-1,4}(\Omega)^d)
		\end{align*}
		and in $\rmL^\infty(0,\infty;\rmL^2(\Omega)^d)$, along with the strong convergence to $\rho\vel$ in $\rmL^2(0,T;\rmL^2(\Omega)^d)$ for all~$0<T<\infty$, implies $\rho\vel\in \BCw([0,\infty);\rmL^2(\Omega)^d)$. Given that $\rho$ belongs to the space~$\BC([0,\infty);\rmL^2(\Omega))$ and is bounded from below by a positive constant, we eventually infer $\vel=\frac{1}{\rho} \rho\vel \in \BCw([0,\infty);\rmL^2(\Omega)^d)$.
		
		In order to show that $\vel$ attains the initial value $\vel_0$, we observe that due to the above embedding, it holds
		\begin{align*}
			\int_\Omega \widetilde{\rho\vel}^N \vert_{t=0} \cdot \bpsi \dx 
			\to \int_\Omega \rho_0\vel \vert_{t=0} \cdot \bpsi \dx 
		\end{align*}
		for all $\bpsi\in \rmC_0^\infty(\Omega)^d$. Since $\widetilde{\rho\vel}^N \vert_{t=0} = \rho_0\vel_0$ by definition, this implies
		\begin{align*}
			\int_\Omega \rho_0\vel_0 \cdot \bpsi \dx 
			= \int_\Omega \rho_0\vel \vert_{t=0} \cdot \bpsi \dx 
		\end{align*}
		for all $\bpsi\in \rmC_0^\infty(\Omega)^d$ and by approximation for all $\bpsi\in \rmH_0^1(\Omega)^d$ as well. Testing with the difference $\vel_0-\vel\vert_{t=0}$ entails $\int_\Omega \rho_0\bigabs{ \vel_0 - \vel \vert_{t=0} }^2 \dx = 0$ and therefore $\vel_0 = \vel\vert_{t=0}$ a.e.~in $\Omega$ as desired.
		
		For the remaining proof of the energy inequality \eqref{AGG:exws:eq:weak_sol_energy_ineq}, we refer to Subsection~\ref{AGG:exws:subsec:energy_ineq} which follows directly.
	\end{proof}

	\subsection{Energy Inequality} \label{AGG:exws:subsec:energy_ineq}
	
	Eventually, we are in a position to complete the proof of Theorem~\ref{thm:AGG} by verifying the energy inequality \eqref{AGG:exws:eq:weak_sol_energy_ineq}. Due to the strong convergences \eqref{AGG:exws:eq:strong_conv_phi_H1} and~\eqref{AGG:exws:eq:strong_conv_vel} we can extract a subsequence satisfying both $\vel^N(t)\to\vel(t)$ in $\rmL^2(\Omega)^d$ as well as~$\varphi^N(t)\to\varphi(t)$ in $\rmH^1(\Omega)$ for almost every $t\in(0,\infty)$. For $E_N$ defined in \eqref{AGG:exws:eq:E_N}, it hence follows
	\begin{align*}
		E_N(t) \to E_\tot\bigprt{\vel(t),\varphi(t)} \quad\text{for almost all } t\in(0,\infty).
	\end{align*}
	Furthermore, employing lower semicontinuity of norms together with $\varphi^N\to\varphi$ a.e.~in $Q$, we obtain for $D_N$ from \eqref{AGG:exws:eq:D_N} that
	\begin{align*}
		\liminf_{N\to\infty} \int_0^\infty D_N(t)\sigma(t)\dt
		\geq \int_0^\infty D(t)\sigma(t)\dt
	\end{align*}
	for all $\sigma\in \rmW^{1,1}(0,\infty)$ with $\sigma\geq0$, where the dissipation $D$ is accordingly given by
	\begin{align*}
		D(t)
		\coloneqq \int_\Omega \Str(\varphi,\D\vel) : \D\vel + \mj(\varphi) \abs{\Grad\mu}^2 + \mr(\varphi) \bigprt{\mu+\alpha\lambda}^2 \dx.
	\end{align*}
	Passing to the limit in the above inequality~\eqref{AGG:exws:eq:energy_ineq_D_N} thus yields
	\begin{align*}
		E_\textnormal{tot}\bigprt{\vel_0,\varphi_0}\sigma(0) + \int_0^\infty E_\tot(t)\sigma'(t) \dt
		\geq \int_0^\infty D(t)\sigma(t)\dt
	\end{align*}
	for all $\sigma\in \rmW^{1,1}(0,\infty)$ with $\sigma\geq0$. According to Lemma~4.3 in \cite{Abels2009g} this is equivalent to the desired energy inequality \eqref{AGG:exws:eq:weak_sol_energy_ineq}.

	\section{Reformulation of the Model with Mass Averaged Velocity} \label{sec:LT_reformulation}
	
	This section is dedicated to a slight reformulation of the thermodynamically consistent diffuse interface model by Aki et al.~\cite{Aki2014} for binary flows with phase transitions. In contrast to the model derived in Section~\ref{sec:AGG_modeling}, the present quasi-incompressible system incorporates a mass averaged velocity to describe the motion of the mixture constituents with unmatched densities. The modifications performed in the following lay the groundwork for the analytical studies conducted in Section~\ref{sec:LT_exws}.
	
	Our starting point is the following variant of the Navier--Stokes/Cahn--Hilliard system proposed in \cite{Aki2014}:
	\begin{equation*}
		\label{LT:reform:original_model}
		\arraycolsep=2pt
		\begin{array}{rclll}
			\rho \bigprt{\delt\uel + (\uel\cdot\nabla)\uel} - \nabla\cdot \Str(\varphi, \D\uel) + \Grad\lambda &=& -\Grad p(\varphi) + \varphi\Grad\Lap\varphi
			&\text{ in $Q$}, \\[1ex]
			\nabla\cdot\uel &=& c_-c_+ \bigprt{\nabla\cdot \bigprt{\mj(\varphi)\Grad} - \mr(\varphi)} \bigprt{\mu+\frac{c_-}{c_+}\lambda}
			&\text{ in $Q$}, \\[1ex]
			\partial_t \varphi + \nabla\cdot (\varphi\uel) &=& c_+^2 \bigprt{\nabla\cdot \bigprt{\mj(\varphi)\Grad} - \mr(\varphi)} \bigprt{\mu+\frac{c_-}{c_+}\lambda}
			&\text{ in $Q$}, \\[1ex]
			\mu &=& F'(\varphi) - \Lap\varphi 
			&\text{ in $Q$}, \\[1ex]
			p(\varphi) &=& \varphi F'(\varphi) - F(\varphi) 
			&\text{ in $Q$},
		\end{array}
	\end{equation*}
	where the relations
	\begin{equation*}
		\arraycolsep=2pt
		\begin{array}{rcl}
			\Str(\varphi, \D\uel) &=& 2\nu(\varphi) \D\uel + \eta(\varphi) \nabla\cdot\uel \mathbf I,\\[1ex]
			\rho(\varphi) &=& b_+ + b_- \varphi
		\end{array}
	\end{equation*}
	are satisfied, and constants related to the specific mass densities are abbreviated by
	\begin{equation*}
		\arraycolsep=2pt
		\begin{array}{rcl}
			b_\pm = \frac{\tilde\rho_+ \pm \tilde\rho_-}{2}, \qquad
			c_\pm = \frac{1}{\tilde\rho_+} \pm \frac{1}{\tilde\rho_-}, \qquad
			\text{where } \tilde\rho_\pm>0, \quad \tilde\rho_+ \neq \tilde\rho_-.
		\end{array}
	\end{equation*}
	We point out that $\uel$ describes the \textit{mass averaged} velocity of the mixture. In this context,~$\mu$~denotes the chemical potential, $p$ represents the non-monotone pressure, and $\lambda$~serves as a Lagrange multiplier ensuring the incompressibility of the constituents. For the remaining quantities, we refer to the introduction.
	
	In order to reformulate the momentum equation, we begin by computing
	\begin{align*}
		-\Grad p(\varphi) + \varphi\Grad\Lap\varphi
		= - \varphi F''(\varphi)\Grad\varphi + \varphi\Grad\Lap\varphi
		= -\varphi \Grad\bigprt{F'(\varphi) - \Lap\varphi}
		= -\varphi\Grad\mu.
	\end{align*}
	From this point onward, we consider $\lambda$ as the pressure and work with a modified chemical potential $\omega \coloneqq \mu+\frac{c_-}{c_+}\lambda$. A simple computation reveals the identity
	\begin{align*}
		\prt{1 - \alpha\varphi} \Grad\lambda
		= \beta\rho \Grad\lambda,
		\qquad\text{where } \alpha = \tfrac{c_-}{c_+}, \quad \beta = \tfrac{2}{\tilde\rho_+ + \tilde\rho_-}.
	\end{align*}
	Since the divergence equation together with the phase field equation are equivalent to the mass balance and the phase field equation, we choose the latter combination and finally end up with the model
	\begin{equation}
		\label{LT:reform:final_model}
		\arraycolsep=2pt
		\begin{array}{rclll}
			\rho \prt{\delt\uel + (\uel\cdot\nabla)\uel} - \nabla\cdot \Str(\varphi, \D\uel) + \beta\rho\Grad\lambda
			&=& -\varphi\Grad \omega
			&\text{ in $Q$}, \\[1ex]
			\partial_t \rho + \nabla\cdot(\rho\uel) 
			&=& 0
			&\text{ in $Q$}, \\[1ex]
			\partial_t \varphi + \nabla\cdot (\varphi\uel)
			&=& c_+^2 \bigprt{\nabla\cdot \bigprt{\mj(\varphi)\Grad\omega} - \mr(\varphi) \omega}
			&\text{ in $Q$}, \\[1ex]
			\omega
			&=& F'(\varphi) - \Lap\varphi + \alpha\lambda
			&\text{ in $Q$}.
		\end{array}
	\end{equation}
	To this system, we add the boundary conditions
	\begin{equation}
		\label{LT:reform:bdry_cond}
		\arraycolsep=2pt
		\begin{array}{rclll}
			\n\cdot\uel \vert_{\partial\Omega}&=&0 &\text{ on $S$}, \\[1ex]
			(\Str\n)_{\btau} &=&-\gamma\uel_{\btau} \vert_{\partial\Omega} &\text{ on $S$}, \\[1ex]
			\partial_\n\varphi \vert_{\partial\Omega} = \partial_\n \omega \vert_{\partial\Omega} &=&0 &\text{ on $S$}.
		\end{array}
	\end{equation}
	The Navier-slip boundary condition for $\uel$ describes that only the normal part of the velocity vanishes on the boundary, whereas the tangential part of $\uel$ is related to the tangential part of the normal stresses via some friction parameter $\gamma>0$. For $\varphi$ and $\omega$, we work with a~$\pi/2$-contact-angle condition and a no-flux boundary condition, respectively.
	
	Furthermore, we note that the total energy of the system is given by
	\begin{align*}
		E_\tot\prt{\uel,\varphi}
		= E_\kin(\uel) + E_\free(\varphi) 
		= \int_\Omega \rho \frac{\abs \uel^2}{2} \dx + \int_\Omega F(\varphi) + \frac{\abs{\Grad\varphi}^2}{2} \dx.
	\end{align*}
	Multiplying the first equation of \eqref{LT:reform:final_model} by $\uel$, the second by $\beta\lambda$, the third by $\omega$, and the fourth by $-\delt\varphi$, then integrating over $\Omega$ and adding, reveals in view of the boundary conditions~\eqref{LT:reform:bdry_cond} that every sufficiently smooth solution to the model satisfies
	\begin{align*}
		\ddt E_{\tot}\bigprt{\uel(t),\varphi(t)} 
		&= - \int_\Omega \Str(\varphi, \D\uel) : \D\uel + \mj(\varphi) c_+^2 \abs{\Grad \omega}^2  
		+ \mr(\varphi) c_+^2 \omega^2 \dx \\
		&\quad- \int_{\partial\Omega} \gamma \abs{\uel_{\btau}}^2 \dsix.
	\end{align*}
	Section \ref{sec:LT_exws} is devoted to proving existence of weak solutions to the quasi-stationary variant of this system.

	\section{Existence of Weak Solutions to the Quasi-Stationary Model with Mass Averaged Velocity} \label{sec:LT_exws}
	
	In this section, we prove existence of weak solutions to the system \eqref{eq:LT}, which is the quasi-stationary version of the reformulated model from Section \ref{sec:LT_reformulation}. To this end, we impose the assumptions \eqref{ass:domain}--\eqref{ass:friction} for the entire section. Since the kinetic energy is neglected, the total energy of the system only consists of the free energy
	\begin{align*}
		E_\free(\varphi) 
		= \int_\Omega F(\varphi) + \frac{\abs{\Grad\varphi}^2}{2} \dx.
	\end{align*}
	Furthermore, with $F_0(s)\coloneqq F(s)+\frac\kappa2s^2$ being the convex part of $F$, we denote the convex part of the free energy by
	\begin{align*}
		E_m(\varphi) =
		\begin{cases}
			\int_\Omega F_0(\varphi) + \frac12 \abs{\Grad \varphi}^2 \dx &\text{for } \varphi\in\dom E_m, \\
			\infty &\text{else},
		\end{cases}
	\end{align*}
	where $\dom E_m=\{\varphi\in \rmH^1_{(m)}(\Omega) : \abs{\varphi}\leq1 \text{ a.e.~in }\Omega\}$. Since for solutions, the spatial mean of $\varphi$ is conserved, cf.~Remark~\ref{LT:exws:rmk:mean-free_eq_for_chem_pot} below, we can prescribe the mean $\mean{\varphi}=m\in(-1,1)$ here without loss of generality. This enables us to use results related to the subgradient of $E_m$ for functions with prescribed mean value, see Proposition~\ref{prelim:prop:subgrad_mean}.

	\subsection{Definition of Weak Solutions} \label{LT:exws:subsec:def_ws}
	
	In the following, we give our precise definition of weak solutions to system~\eqref{eq:LT}.
	
	\begin{definition}
		\label{LT:exws:def:weak_sol}
		Let $\varphi_0\in \rmH^1(\Omega)$ with $\abs{\varphi_0}\leq1$ a.e.~in $\Omega$. Under the assumptions~\eqref{ass:domain}--\eqref{ass:friction} a quadruple $(\uel,\lambda,\omega,\varphi)$ with the properties
		\begin{alignat*}{2}
			\uel &\in \rmL^2(0,\infty;\rmH^1_\n(\Omega)),\\
			\lambda &\in \rmL^2_\uloc([0,\infty);\rmL^r(\Omega)),\\
			\omega &\in \rmL^2(0,\infty;\rmH^1(\Omega)),\\
			\varphi &\in \BCw([0,\infty);\rmH^1(\Omega)) \cap \rmL^2_\uloc([0,\infty);\rmW^{2,r}(\Omega)),\\
			F'(\varphi) &\in \rmL^2_\uloc([0,\infty);\rmL^r(\Omega)),
		\end{alignat*}
		where $r\in(\frac{2d}{d+2},\frac{d}{d-1})$ is arbitrary, is called a \emph{weak solution of \eqref{eq:LT}} if the following equations are satisfied:
		\begin{subequations}
		\begin{align}
			&\int_{Q} \rho^{-1} \Str(\varphi, \D\uel) : \Grad\bpsi + \Str \bigprt{\varphi,\D\uel} \Grad \rho^{-1} \cdot \bpsi - \beta \lambda \nabla\cdot\bpsi \dtx \nonumber\\
			&\quad+ \int_{S} \gamma \rho^{-1} \uel_{\btau} \cdot \bpsi_{\btau} \dtsix 
			= \int_{Q} - \rho^{-1} \varphi\Grad \omega \cdot \bpsi \dtx
			\label{LT:exws:eq:weak_sol_a}
		\end{align}
		for all $\bpsi\in \rmC_0^\infty((0,\infty);\rmW^{1,r'}_\n(\Omega))$, 
		\begin{align}
			\int_{Q} -\rho \delt\xi + \nabla\cdot(\rho\uel) \xi \dtx = 0 \label{LT:exws:eq:weak_sol_b}
		\end{align}
		for all $\xi\in \rmC_0^\infty((0,\infty);\rmH^1(\Omega))$
		as well as
		\begin{align}
			\int_{Q} - \varphi\delt\zeta - \varphi\uel\cdot\Grad\zeta \dtx 
			= \int_{Q} -\mj(\varphi) c_+^2 \Grad \omega \cdot\Grad\zeta - \mr(\varphi) c_+^2 \omega \zeta \dtx
			\label{LT:exws:eq:weak_sol_c}
		\end{align}
		for all $\zeta\in \rmC_0^\infty((0,\infty);\rmH^1(\Omega))$ and
		\begin{align}
			\omega = F'(\varphi)-\Lap\varphi + \alpha\lambda \quad\text{a.e.~in $Q$}  
			\label{LT:exws:eq:weak_sol_d}.
		\end{align}
		\end{subequations}
		Additionally, we require the boundary and initial conditions
		\begin{alignat*}{2}
			\partial_\n\varphi \vert_{\partial\Omega} &=0 &&\quad\text{a.e.~in $S$}, \\ 
			\varphi \vert_{t=0} &= \varphi_0 &&\quad\text{a.e.~in $\Omega$}
		\end{alignat*}
		to hold, as well as the energy inequality
		\begin{align}
			&E_\free\bigprt{\varphi(t)} 
			+ \int_{Q_{(s,t)}} \Str(\varphi, \D\uel) : \D\uel + \mj(\varphi) c_+^2 \abs{\Grad \omega}^2  
			+ \mr(\varphi) c_+^2 \omega^2 \dtaux \nonumber\\
			&\quad + \int_{S_{(s,t)}} \gamma \abs{\uel_{\btau}}^2 \dtausix
			\leq E_\free\bigprt{\varphi(s)}
			\label{LT:exws:eq:weak_sol_energy_ineq}
		\end{align}
		for almost all $s\in [0,\infty)$ including $s=0$, and for all $t\in[s,\infty)$. 
	\end{definition}

	\begin{remark}
		Assume that the friction coefficient $\gamma$ is positive, and that both the viscosity coefficient $\nu$ and the mobility coefficients $\mj$, $\mr$ are bounded from below by some positive constant. For stationary states $(\uel^*,\lambda^*,\omega^*,\varphi^*)$ of system~\eqref{eq:LT}, the energy inequality \eqref{LT:exws:eq:weak_sol_energy_ineq} reveals that $\D\uel^*=0$, $\omega^*=0$ and $\uel^*_{\btau}=0$. Therefore, we derive $\uel^*=0$ by Korn's inequality, cf.~(3.5.11) in \cite{Abels2007}. We conclude that $(\lambda^*,\varphi^*)$ is a stationary state of the Cahn--Hilliard equation, i.e., a solution of 
		\begin{equation*}
			\arraycolsep=2pt
			\begin{array}{rclll}
				F'(\varphi^*) - \Lap\varphi^* &=&-  \alpha\lambda^*&\text{ in } \Omega, \\[0.5ex]
				\partial_\n \varphi^* \vert_{\partial\Omega} &=&0 &\text{ on } \partial\Omega.
			\end{array}
		\end{equation*}
		In contrast, if $\gamma=0$ and the domain $\Omega$ has a symmetry axis, Korn's inequality is not applicable, and consequently, stationary states can also take the form of rigid motions with respect to the symmetry axis.
	\end{remark}

	\subsection{Implicit Time Discretization} \label{LT:exws:subsec:time_discr}
	
	This section consists of the construction of approximate weak solutions to system~\eqref{eq:LT} with the help of an implicit time discretization.
	
	\begin{definition}
		With $N\in\N$ being fixed arbitrarily, let $h=\frac{1}{N}$ denote the time step size. For $k\in\N_0$, consider some given function $\varphi_k\in \rmH^1_{(m)}(\Omega)$, where $m\in(-1,1)$, which satisfies~$F'(\varphi_k)\in \rmL^2(\Omega)$, and define $\rho_k\coloneqq b_+ + b_- \varphi_k$. Then, we determine a quadruple~$(\uel,\lambda_0,\omega,\varphi) \in \rmH^1_\n(\Omega) \times \rmL^2_{(0)}(\Omega) \times \rmH^2_\Neum(\Omega) \times \mathcal D(\partial E_m)$ as a solution of the following time-discrete system at time step $k+1$:
		\begin{subequations}
		\begin{align}
			\int_\Omega \Str(\varphi_k, \D\uel) : \Grad\bpsi - \beta  \lambda_0 \nabla\cdot(\rho_k\bpsi) \dx + \int_{\partial\Omega} \gamma \uel_{\btau} \cdot \bpsi_{\btau} \dsix 
			= \int_\Omega -\varphi_k\Grad \omega \cdot \bpsi \dx
			\label{LT:exws:eq:time-discr_sol_a}
		\end{align}
		for all $\bpsi\in \rmH^1_\n(\Omega)$, as well as
		\begin{align}
			\frac{\rho-\rho_k}{h} + \nabla\cdot(\rho_k\uel) + h\lambda_0 = 0 \quad\text{a.e.~in $\Omega$}, \label{LT:exws:eq:time-discr_sol_b}
		\end{align}
		where $\rho= b_+ + b_- \varphi$, 
		\begin{align}
			\int_\Omega \frac{\varphi-\varphi_k}{h}\zeta - \varphi_k\uel\cdot\Grad\zeta \dx 
			= \int_\Omega -\mj(\varphi_k) c_+^2 \Grad \omega \cdot\Grad\zeta - \mr(\varphi_k) c_+^2 \omega \zeta \dx
			\label{LT:exws:eq:time-discr_sol_c}
		\end{align}
		for all $\zeta\in \rmH^1(\Omega)$, and
		\begin{align}
			P_0(\omega) + \kappa\frac{\varphi+\varphi_k}{2} - \kappa m = P_0 \bigprt{F_0'(\varphi)} - \Lap\varphi + \alpha\lambda_0 \quad\text{a.e.~in $\Omega$}.  \label{LT:exws:eq:time-discr_sol_d1} \owntag{1}
		\end{align}
		\end{subequations}
	\end{definition}

	\begin{remark}
		\label{LT:exws:rmk:mean-free_eq_for_chem_pot}
		Concerning the time-discrete system above, we point out the following.
		\begin{enumerate}[label=(\roman*)]
			\item In \eqref{LT:exws:eq:time-discr_sol_b}, a damping term $h\lambda_0$ is added in order to obtain an \textit{a priori} estimate for the pressure $\lambda_0$ in $\rmL^2(\Omega)$ on the time-discrete level, cf. the energy inequality~\eqref{LT:exws:eq:discr_energy_est} below. Although the term degenerates in the limit passage $h\to0$ we are able to keep $\lambda_0$ under control in a suitable space.
			\item For given $\varphi_k\in \rmH^1_{(m)}(\Omega)$, let $(\uel,\lambda_0,\omega,\varphi)$ be a solution to the time-discrete system~\eqref{LT:exws:eq:time-discr_sol_a}--\eqref{LT:exws:eq:time-discr_sol_d1}.
			Here, requiring $\varphi_k$ and $\varphi$ to have equal mean values is not a restriction for the following reason. Integrating \eqref{LT:exws:eq:time-discr_sol_b} over $\Omega$ shows, in light of the divergence theorem and the boundary condition for $\uel$, that $\mean{\rho}=\mean{\rho_k}$. Since~$\varphi$ and~$\varphi_k$ have the same affine linear dependence on $\rho$ and $\rho_k$, respectively, we conclude~$\mean{\varphi}=\mean{\varphi_k}$.
			\item Testing \eqref{LT:exws:eq:time-discr_sol_c} with $\zeta=1$, the identity $\mean{\varphi}=\mean{\varphi_k}$ further reveals that the spatial mean of $\omega$ is given by
			\begin{align*}
				\mean{\omega} = \biggprt{\int_\Omega \mr(\varphi_k) \dx}^{-1} \int_\Omega \mr(\varphi_k)P_0(\omega) \dx.
			\end{align*}
			\item Provided that \eqref{LT:exws:eq:time-discr_sol_c} is satisfied, the time-discrete equation for the chemical potential~\eqref{LT:exws:eq:time-discr_sol_d2} below and its spatial mean-free part~\eqref{LT:exws:eq:time-discr_sol_d1} above are equivalent in the following sense. 
			For $k\in\N_0$, let $\varphi_k\in \rmH^1_{(m)}(\Omega)$ be given. 
			If a quadruple~$(\uel,\lambda_0,\omega,\varphi) \in \rmH^1_\n(\Omega) \times \rmL^2_{(0)}(\Omega)\times \rmH^2_\Neum(\Omega) \times \mathcal D(\partial E_m)$ with $\mean{\varphi}=\mean{\varphi_k}$ is a solution to \eqref{LT:exws:eq:time-discr_sol_c}--\eqref{LT:exws:eq:time-discr_sol_d1}, setting~$\lambda\coloneqq\lambda_0 + \bar{\lambda}$, where $\bar{\lambda}$ is defined by the identity 
			\begin{align*}
				-\alpha\bar\lambda = \mean{F_0'(\varphi)} - \bar{\omega}- \kappa m, 
				\qquad \bar\omega\coloneqq \biggprt{\int_\Omega \mr(\varphi_k) \dx}^{-1} \int_\Omega \mr(\varphi_k)P_0(\omega) \dx,
			\end{align*}
			implies that $(\uel,\lambda,\omega,\varphi)$ solves \eqref{LT:exws:eq:time-discr_sol_c}  and
			\addtocounter{equation}{-1}
			\begin{subequations}
			\addtocounter{equation}{3}
			\begin{align}
				\omega + \kappa\frac{\varphi+\varphi_k}{2} = F_0'(\varphi)-\Lap\varphi + \alpha\lambda \quad\text{a.e.~in $\Omega$}.  
				\label{LT:exws:eq:time-discr_sol_d2} \owntag{2}
			\end{align}
			\end{subequations}
			On the other hand, if $(\uel,\lambda,\omega,\varphi) \in \rmH^1_\n(\Omega) \times  \rmL^2(\Omega)\times \rmH^2_\Neum(\Omega) \times \mathcal D(\partial E_m)$ solves \eqref{LT:exws:eq:time-discr_sol_c}--\eqref{LT:exws:eq:time-discr_sol_d2}, then $(\uel,\lambda_0,\omega,\varphi)$ with $\lambda_0\coloneqq \lambda-\mean{\lambda}$ is a solution to \eqref{LT:exws:eq:time-discr_sol_c}--\eqref{LT:exws:eq:time-discr_sol_d1}.
		\end{enumerate}
	\end{remark}

	\begin{lemma}
		\label{LT:exws:lemma:estimates}
		Let $\varphi_k \in \rmH^2_{(m)}(\Omega)$, where $m\in(-1,1)$, be given with $\abs{\varphi_k}\leq1$ a.e.~in $\Omega$. We assume a quadruple $(\uel,\lambda_0,\omega,\varphi) \in \rmH^1_\n(\Omega) \times \rmL^r_{(0)}(\Omega) \times \rmH^2_\Neum(\Omega) \times \mathcal D(\partial E_m)$ to solve \eqref{LT:exws:eq:time-discr_sol_c}--\eqref{LT:exws:eq:time-discr_sol_d1}. Then there exists a constant $C=C(m,r)>0$ such that
		\begin{align*}
			&\norm{\varphi}_{\rmW^{2,r}(\Omega)} + \norm{F_0'(\varphi)}_{\rmL^r(\Omega)} \\
			&\leq C \bigprt{ \norm{\lambda_0}_{\rmL^r(\Omega)} + \norm{\Grad \omega}_{\rmL^2(\Omega)} + \norm{\varphi}_{\rmH^1(\Omega)} + \norm{\varphi_k}_{\rmH^1(\Omega)} + 1 } \bigprt{\norm{\varphi}_{\rmH^1(\Omega)} + 1}.
		\end{align*}
	\end{lemma}

	\begin{proof}
		We start the proof by noting that, since $\mean\varphi \in (-1,1)$ and $\lim_{s\to\pm1} F_0'(s) = \pm\infty$, there exists some $\varepsilon>0$ such that the following conditions are satisfied:
		\begin{gather*}
			\mean\varphi \in (-1+\varepsilon,1-\varepsilon),\\
			F_0'(\varphi) \geq 0 \text{ for } \varphi\in \big[1-\tfrac\varepsilon2, 1 \big] 
			\quad\text{and}\quad
			F_0'(\varphi) \leq 0 \text{ for } \varphi\in \big[-1,-1+\tfrac\varepsilon2\big].
		\end{gather*}
		For this particular choice of $\varepsilon$, it is possible to show that the inequality
		\begin{align}
			\label{AGG:exws:eq:est_F0'}
			F_0'(\varphi) \prt{\varphi-\mean\varphi}
			\geq C_1 \abs{F_0'(\varphi)} - C_2
		\end{align}
		is valid for all $\varphi\in[-1,1]$, with constants $C_{1/2}=C_{1/2}(\varepsilon)>0$. 
		
		Recalling that~\eqref{LT:exws:eq:time-discr_sol_d1} and \eqref{LT:exws:eq:time-discr_sol_d2} are equivalent, see Remark~\ref{LT:exws:rmk:mean-free_eq_for_chem_pot}, we test the latter equation with $\varphi-m$. In view of estimate \eqref{AGG:exws:eq:est_F0'}, Poincaré's inequality, and the bound~$\abs{\varphi}\leq1$ a.e.~in $\Omega$ this leads to
		\begin{align*}
			&\Bigabs{ \int_\Omega F_0'(\varphi) \dx }
			\leq \int_\Omega \abs{F_0'(\varphi)} \dx
			\leq C \int_\Omega F_0'(\varphi) \prt{\varphi-m} \dx + C \\
			&= C \int_\Omega \Lap\varphi \prt{\varphi-m} + \omega \prt{\varphi-m} + \kappa\frac{\varphi+\varphi_k}{2} \prt{\varphi-m} - \alpha \lambda (\varphi-m) \dx + C \\
			&= C \int_\Omega -\abs{\Grad\varphi}^2 + \omega \prt{\varphi-m} 
			+  \frac\kappa2\prt{\varphi-m}^2 
			+  \frac\kappa2\prt{\varphi_k-m}\prt{\varphi-m}
			- \alpha \lambda_0 (\varphi-m) \dx + C \\
			&\leq C \bigprt{ \norm{\Grad \omega}_{\rmL^2(\Omega)} \norm{\Grad\varphi}_{\rmL^2(\Omega)} + \norm{\Grad\varphi}_{\rmL^2(\Omega)}^2 
				+ \norm{\Grad\varphi_k}_{\rmL^2(\Omega)}\norm{\Grad\varphi}_{\rmL^2(\Omega)} \\
			&\quad + \norm{\lambda_0}_{\rmL^r(\Omega)} \norm{\varphi}_{\rmL^{r'}(\Omega)} +1 }\\
			&\leq C \bigprt{ \norm{\lambda_0}_{\rmL^r(\Omega)} + \norm{\Grad \omega}_{\rmL^2(\Omega)} + \norm{\varphi}_{\rmH^1(\Omega)} + \norm{\varphi_k}_{\rmH^1(\Omega)} + 1 } \bigprt{\norm{\varphi}_{\rmH^1(\Omega)} + 1}
		\end{align*}
		due to the embedding $\rmH^1(\Omega) \hookrightarrow \rmL^{r'}(\Omega)$ since $r>\frac{2d}{d+2}$.
		Moreover, we combine \eqref{eq:subgrad_with:estimate_W2r} with the above estimate which, invoking \eqref{LT:exws:eq:time-discr_sol_d1}, entails
		\begin{align*}
			&\norm{\varphi}_{\rmW^{2,r}(\Omega)} + \norm{F_0'(\varphi)}_{\rmL^r(\Omega)} \\
			&\leq C \bigprt{\norm{\partial E_m(\varphi)}_{\rmL^r(\Omega)} 
				+ \bigabs{\mean{F_0'(\varphi)}} 
				+ \norm{\varphi}_{\rmL^2(\Omega)} +1 } \\
			&= C \biggprt{\Bignorm{P_0(\omega) + \kappa\frac{\varphi+\varphi_k}{2} - \kappa m - \alpha\lambda_0 }_{\rmL^r(\Omega)} 
				+ \bigabs{\mean{F_0'(\varphi)}} 
				+ \norm{\varphi}_{\rmL^2(\Omega)} +1 } \\
			&\leq C \bigprt{ \norm{\lambda_0}_{\rmL^r(\Omega)} + \norm{\Grad \omega}_{\rmL^2(\Omega)} + \norm{\varphi}_{\rmH^1(\Omega)} + \norm{\varphi_k}_{\rmH^1(\Omega)} + 1 } \bigprt{\norm{\varphi}_{\rmH^1(\Omega)} + 1}
		\end{align*}
		since $r< \frac{d}{d-1}$, and thereby completes the proof.
	\end{proof}

	\begin{lemma}
		\label{LT:exws:lemma:discr_energy_est}
		Let $\varphi_k\in \rmH^2_{(m)}\Omega)$ with  $m\in(-1,1)$ be given, and set $\rho_k\coloneqq b_+ + b_- \varphi_k$. Then, a solution $(\uel,\lambda_0,\omega,\varphi) \in \rmH^1_\n(\Omega) \times \rmL^2_{(0)}(\Omega)\times \rmH^2_\Neum(\Omega) \times  \mathcal D(\partial E_m) $ to the time-discrete system~\eqref{LT:exws:eq:time-discr_sol_a}--\eqref{LT:exws:eq:time-discr_sol_d1} satisfies the energy inequality
		\begin{align}
			\label{LT:exws:eq:discr_energy_est}
			&E_\free(\varphi)
			+ \int_\Omega \frac{\abs{\Grad\varphi-\Grad\varphi_k}^2}{2} \dx \nonumber\\
			&\quad+ h\int_\Omega \Str(\varphi_k,\D\uel) : \D\uel + \mj(\varphi_k) c_+^2 \abs{\Grad \omega}^2 + \mr(\varphi_k) c_+^2 \omega^2 + h\beta\lambda_0^2 \dx \nonumber\\
			&\quad+ h \int_{\partial\Omega} \gamma \abs{\uel_{\btau}}^2 \dsix 
			\leq E_\free(\varphi_k).
		\end{align}
	\end{lemma}

	\begin{proof}
		We perform the same testing procedure on the time-discrete system as in the proof of Lemma~\ref{AGG:exws:lemma:discr_energy_est}. First, we choose $\uel$ as a test function in \eqref{LT:exws:eq:time-discr_sol_a} to obtain
		\begin{align*}
			0
			= \int_\Omega \Str(\varphi_k,\D\uel) : \D\uel - \beta\lambda_0 \nabla\cdot \prt{\rho_k\uel} + \varphi_k\Grad \omega\cdot\uel \dx 
			+ \int_{\partial\Omega} \gamma \abs{\uel_{\btau}}^2 \dsix.
		\end{align*}
		Testing \eqref{LT:exws:eq:time-discr_sol_b} with $\beta \lambda_0$ leads to
		\begin{align*}
			0
			= \int_\Omega -\alpha \lambda_0 \frac{\varphi-\varphi_k}{h} + \beta \lambda_0 \nabla\cdot(\rho_k\uel) + h \beta \lambda_0^2 \dx
		\end{align*}
		and \eqref{LT:exws:eq:time-discr_sol_c} with $\omega$ gives
		\begin{align*}
			0
			= \int_\Omega \frac{\varphi-\varphi_k}{h}\omega - \varphi_k\uel\cdot\Grad \omega + \mj(\varphi_k) c_+^2 \abs{\Grad \omega}^2 + \mr(\varphi_k) c_+^2 \omega^2  \dx.
		\end{align*}
		Finally, we test \eqref{LT:exws:eq:time-discr_sol_d1} with $\frac{\varphi-\varphi_k}{h}$ which results in
		\begin{align*}
			0
			&= \int_\Omega \frac1h \biggprt{ \frac{\abs{\Grad\varphi}^2}{2} - \frac{\abs{\Grad\varphi_k}^2}{2} + \frac{\abs{\Grad\varphi-\Grad\varphi_k}^2}{2} } + \bigprt{ P_0(F_0'(\varphi)) + \mean{F_0'(\varphi)} }\frac{\varphi-\varphi_k}{h} \dx \\
			&\quad + \int_\Omega - \omega\frac{\varphi-\varphi_k}{h} - \kappa \frac{\varphi^2-\varphi_k^2}{2h} + \alpha \lambda_0 \frac{\varphi-\varphi_k}{h} \dx,
		\end{align*}
		where we note that the term $\int_\Omega - \mean{F_0'(\varphi)}  \frac{\varphi-\varphi_k}{h} - \mean{\omega} \frac{\varphi-\varphi_k}{h} + \kappa m \frac{\varphi-\varphi_k}{h} \dx$ vanishes due to $\varphi$ and $\varphi_k$ having equal mean values. Summing these identities proves the time-discrete energy inequality~\eqref{LT:exws:eq:discr_energy_est}, cf.~the proof of Lemma~\ref{AGG:exws:lemma:discr_energy_est} for details.
	\end{proof}

	\begin{lemma}
		\label{LT:exws:lemma:ex_time-discr}
		Let $\varphi_k\in \rmH^2_{(m)}(\Omega)$ be given with $m\in(-1,1)$, and set $\rho_k\coloneqq b_+ + b_- \varphi_k$. Then, there exists a solution $(\uel,\lambda_0, \omega,\varphi) \in \rmH^1_\n(\Omega) \times  \rmL^2_{(0)}(\Omega) \times \rmH^2_\Neum(\Omega) \times  \mathcal D(\partial E_m)$ to the time-discrete system \eqref{LT:exws:eq:time-discr_sol_a}--\eqref{LT:exws:eq:time-discr_sol_d1}.
	\end{lemma}

	\begin{proof}
		As in the proof of Lemma~\ref{AGG:exws:lemma:ex_time-discr}, we aim at applying the Leray--Schauder principle. To this end, setting
		\begin{align*}
			X &\coloneqq \rmH^1_\n(\Omega) \times \rmL^2_{(0)}(\Omega) \times \rmH^2_\Neum(\Omega) \times \mathcal{D}(\partial E_m), \\[1ex]
			Y &\coloneqq \rmH^{-1}_\n(\Omega) \times \rmL^2_{(0)}(\Omega) \times \rmL^2(\Omega) \times \rmL^2(\Omega),
		\end{align*}
		we seek operators $\mathcal{L}_k,\mathcal{F}_k \colon X\to Y$ satisfying the equivalence that $\z = (\uel,\lambda_0, \omega,\varphi) \in X$ solves the system \eqref{LT:exws:eq:time-discr_sol_a}--\eqref{LT:exws:eq:time-discr_sol_d1} if and only if $\mathcal{L}_k(\z) - \mathcal{F}_k(\z) = 0$. Therefore, we introduce
		\begin{align*}
			\langle L_k(\uel,\lambda_0, \omega), \bpsi \rangle 
			&\coloneqq \int_\Omega \Str(\varphi_k, \D\uel) : \Grad\bpsi - \beta \lambda_0 \nabla\cdot(\rho_k\bpsi) + \varphi_k\Grad \omega \cdot \bpsi \dx 
			+ \int_{\partial\Omega} \gamma \uel_{\btau} \cdot \bpsi_{\btau} \dsix
		\end{align*}
		for all $\bpsi\in \rmH^1_\n(\Omega)$, and subsequently define, for $\z = (\uel,\lambda_0, \omega,\varphi) \in X$, the operators
		\begin{align*}
			\mathcal{L}_k(\z)=
			\begin{pmatrix}
				L_k(\uel,\lambda_0, \omega) \\[1ex]
				\beta \nabla\cdot(\rho_k\uel) + \beta h\lambda_0 \\[1ex]
				\nabla\cdot (\varphi_k\uel) - c_+^2 \bigprt{ \nabla\cdot \prt{\mj(\varphi_k) \Grad\omega} - \mr(\varphi_k)\omega } \\[1ex]
				\varphi + \partial E_m(\varphi)
			\end{pmatrix}
		\end{align*}
		and 
		\begin{align*}
			\mathcal{F}_k(\z)=
			\begin{pmatrix}
				0 \\[1ex]
				- \beta \frac{\rho-\rho_k}{h} \\[1ex]
				-\frac{\varphi-\varphi_k}{h} \\[1ex]
				\varphi + P_0(\omega) + \kappa\frac{\varphi+\varphi_k}{2} - \kappa m - \alpha\lambda_0 
			\end{pmatrix}.
		\end{align*}
		In order to prove that the operator $\mathcal{L}_k\colon X\to Y$ is invertible we assemble the first three lines of $\mathcal{L}_k$ into a bilinear form $\mathcal{A}_k$ on $\rmH^1_\n(\Omega)^d \times \rmL^2_{(0)}(\Omega) \times \rmH^1(\Omega)$ given by
		\begin{align*}
			\mathcal{A}_k \bigprt{(\uel,\lambda_0, \omega), (\bpsi,\xi,\zeta)} 
			&\coloneqq \langle L_k(\uel,\lambda_0, \omega), \bpsi \rangle 
			+ \int_\Omega \xi \beta \nabla\cdot(\rho_k\uel) + \xi \beta h\lambda_0 \dx \\
			&\quad+ \int_\Omega \zeta \nabla\cdot (\varphi_k\uel) + c_+^2\mj(\varphi_k) \Grad \omega \cdot \Grad\zeta + c_+^2\mr(\varphi_k) \omega\zeta \dx,
		\end{align*}
		which is bounded and coercive as
		\begin{align*}
			\mathcal{A}_k \bigprt{(\uel,\lambda_0, \omega), (\uel,\lambda_0, \omega)} 
			&\geq C \Bigprt{ \norm{\D\uel}_{\rmL^2(\Omega)}^2 + h\norm{\lambda_0}_{\rmL^2(\Omega)}^2 + \norm{ \omega}_{\rmH^1(\Omega)}^2 + \norm{\uel_{\btau}}_{\rmL^2(\partial\Omega)}^2 } \\
			&\geq C \Bigprt{ \norm{\uel}_{\rmH^1(\Omega)}^2 + h\norm{\lambda_0}_{\rmL^2(\Omega)}^2 + \norm{ \omega}_{\rmH^1(\Omega)}^2}
		\end{align*}
		with a constant $C>0$ in view of Korn's inequality (cf.~(3.5.11) in \cite{Abels2007}) and the positive lower bound on the coefficients $\nu$, $\eta$, $\mj$, $\mr$ and $\gamma$. From the Lax--Milgram theorem, we therefore derive that for any right-hand side $(\mathbf g_1,g_2,g_3)\in \rmH^{-1}_\n(\Omega) \times \rmL^2_{(0)}(\Omega) \times \bigprt{\rmH^1(\Omega)}'$, there exists a unique triple $(\uel,\lambda_0, \omega) \in \rmH^1_\n(\Omega) \times \rmL^2_{(0)}(\Omega) \times \rmH^1(\Omega)$ satisfying
		\begin{align*}
			\mathcal{A}_k \bigprt{(\uel,\lambda_0, \omega), (\bpsi,\xi,\zeta)} 
			= \langle \mathbf g_1, \bpsi \rangle_{\rmH^1_\n(\Omega)} +  \langle g_2, \xi \rangle_{\rmL^2_{(0)}(\Omega)} +  \langle g_3, \zeta \rangle_{\rmH^1(\Omega)}
		\end{align*}
		for all $(\bpsi,\xi,\zeta)\in \rmH^1_\n(\Omega) \times \rmL^2_{(0)}(\Omega) \times \rmH^1(\Omega)$.
		In particular, for given $g_3\in \rmL^2(\Omega)$, $\omega$ is a weak solution to the elliptic equation
		\begin{align*}
			-c_+^2\bigprt{ \nabla\cdot\prt{\mj(\varphi_k)\Grad\omega} - \mr(\varphi_k) \omega }
			= g_3- \nabla\cdot(\varphi_k\uel)
		\end{align*}
		with homogeneous Neumann boundary conditions whose right-hand side is in $\rmL^2(\Omega)$. In light of standard elliptic regularity theory, we conclude that $\omega$ belongs to $\rmH^2_\Neum(\Omega)$ with
		\begin{align}
			\label{LT:exws:eq:omega_H2}
			\norm{\omega}_{\rmH^2(\Omega)}
			\leq C_k \bigprt{ \norm{g_3}_{\rmL^2(\Omega)} + \norm{\nabla\cdot(\varphi_k\uel)}_{\rmL^2(\Omega)}}.
		\end{align}
		Analogously to the proof of Lemma~\ref{AGG:exws:lemma:ex_time-discr}, concerning the fourth line of $\mathcal{L}_k$, we establish that
		\begin{align*}
			\prt{I+\partial E_m}^{-1} \colon \rmL^2(\Omega)\to \rmH^{2-s}(\Omega), \qquad s\in\bigprt{0,\tfrac14},
		\end{align*}
		is a compact operator, where we employ the results~\eqref{eq:subgrad_with:subgradient} and \eqref{eq:subgrad_with:estimate_H2} for the subgradient $\partial E_m$.
		
		Altogether, we have verified the existence of an inverse $\mathcal{L}_k^{-1}\colon Y\to X$. Next, in order to introduce Banach spaces $\tilde X$, $\tilde Y$ such that we have continuous embeddings $X\hookrightarrow\tilde X$, $\tilde Y\hookrightarrow Y$ and such that the operator $\mathcal{L}_k^{-1}\colon \tilde Y\to \tilde X$ is compact, we define
		\begin{align*}
			\tilde X &\coloneqq \rmH^1_\n(\Omega) \times  \rmL^2_{(0)}(\Omega) \times \rmH^2_\Neum(\Omega) \times \rmH^{2-s}(\Omega), \qquad s\in\bigprt{0,\tfrac14}, \\[1ex]
			\tilde Y &\coloneqq \rmH^1_\n(\Omega) \times \rmH^1_{(0)}(\Omega) \times \rmH^1(\Omega) \times \rmL^2(\Omega).
		\end{align*}
		These definitions ensure the desired continuity of the embeddings $X\hookrightarrow\tilde X$, $\tilde Y\hookrightarrow Y$. Moreover, we note that $\rmH^1_\n(\Omega) \times \rmH^1_{(0)}(\Omega) \times \rmH^1(\Omega)$, comprising the first three factors of the product space $\tilde Y$, embeds compactly into the Banach space $\rmH^{-1}_\n(\Omega) \times \rmL^2_{(0)}(\Omega) \times \rmL^2(\Omega)$ consisting of the corresponding first three components of $Y$. Recalling now the compactness of~$\prt{I+\partial E_m}^{-1} \colon \rmL^2(\Omega)\to \rmH^{2-s}(\Omega)$, $s\in(0,\tfrac14)$ the composed operator $\mathcal{L}_k^{-1}\colon \tilde Y\to \tilde X$ is compact.
		
		Furthermore, studying the operator $\mathcal{F}_k$ with the spaces introduced above, it immediately becomes evident that $\mathcal{F}_k \colon\tilde X\to\tilde Y$ is continuous and maps bounded sets into bounded sets.
		
		Finally, to obtain a fixed point of the compact operator $\mathcal{K}_k \coloneqq \mathcal{F}_k \circ \mathcal{L}_k^{-1} \colon \tilde Y\to\tilde Y$ by the Leray--Schauder principle, we verify condition \eqref{AGG:exws:eq:Leray-Schauder_cond} following a similar procedure as in the proof of Lemma~\ref{AGG:exws:lemma:ex_time-discr}. To this end, let $\f\in\tilde Y$ and $0\leq\vartheta\leq1$ satisfying $\f=\vartheta\mathcal{K}_k(\f)$ be given. With $\z=\mathcal{L}_k^{-1}(\f)$, this equation can be rewritten as $\mathcal{L}_k(\z) - \vartheta\mathcal{F}_k(\z) = 0$, which is equivalent to the following weak formulation:
		\begin{subequations}
		\begin{align}
			\int_\Omega \Str(\varphi_k, \D\uel) : \Grad\bpsi - \beta  \lambda_0 \nabla\cdot(\rho_k\bpsi) \dx + \int_{\partial\Omega} \gamma \uel_{\btau} \cdot \bpsi_{\btau} \dsix 
			= \int_\Omega -\varphi_k\Grad \omega \cdot \bpsi \dx
			\label{LT:exws:eq:time-discr_sol_LS_a}
		\end{align}
		for all $\bpsi\in \rmH^1_\n(\Omega)$, as well as
		\begin{align}
			\vartheta\beta\frac{\rho-\rho_k}{h} + \beta\nabla\cdot(\rho_k\uel) + \beta h\lambda_0 = 0 \quad\text{a.e.~in $\Omega$}, \label{LT:exws:eq:time-discr_sol_LS_b}
		\end{align}
		where $\rho= b_+ + b_- \varphi$, 
		\begin{align}
			\int_\Omega \vartheta\frac{\varphi-\varphi_k}{h}\zeta - \varphi_k\uel\cdot\Grad\zeta \dx 
			= \int_\Omega -\mj(\varphi_k) c_+^2 \Grad \omega \cdot\Grad\zeta - \mr(\varphi_k) c_+^2 \omega \zeta \dx
			\label{LT:exws:eq:time-discr_sol_LS_c}
		\end{align}
		for all $\zeta\in \rmH^1(\Omega)$, and
			\begin{align}
				\vartheta\varphi +  \vartheta P_0(\omega) + \vartheta\kappa\frac{\varphi+\varphi_k}{2} - \vartheta\kappa m - \vartheta\alpha\lambda_0 
				= \varphi + P_0 \bigprt{F_0'(\varphi)} - \Lap\varphi
				\quad\text{a.e.~in $\Omega$}.  
				\label{LT:exws:eq:time-discr_sol_LS_d1} \owntag{1}
			\end{align}
		\end{subequations}
		Now, as for proving the time-discrete energy inequality in Lemma~\ref{LT:exws:lemma:discr_energy_est}, testing \eqref{LT:exws:eq:time-discr_sol_LS_a} with~$\uel$, \eqref{LT:exws:eq:time-discr_sol_LS_b} with $\lambda_0$, \eqref{LT:exws:eq:time-discr_sol_LS_c} with $\omega$ and \eqref{LT:exws:eq:time-discr_sol_LS_d1} with $\frac{\varphi-\varphi_k}{h}$, followed by a summation of the resulting equations, gives
		\begin{align*}
			0
			&\geq \int_\Omega \Str(\varphi_k,\D\uel) : \D\uel +  \mj(\varphi_k) c_+^2 \abs{\Grad \omega}^2 + \mr(\varphi_k) c_+^2 \omega^2 + h\beta\lambda_0^2 \dx
			+ \int_{\partial\Omega} \gamma \abs{\uel_{\btau}}^2 \dsix \\
			&\quad + \int_\Omega
			\frac1h \biggprt{\frac{\abs{\Grad\varphi}^2}{2} - \frac{\abs{\Grad\varphi_k}^2}{2} +  \frac{\abs{\Grad\varphi-\Grad\varphi_k}^2}{2}} 
			+ \frac1h \bigprt{F_0(\varphi)-F_0(\varphi_k)} - \vartheta \kappa \frac{\varphi^2-\varphi_k^2}{2h} \dx \\
			&\quad+ \int_\Omega (1-\vartheta) \frac1h\biggprt{ \frac{\varphi^2}{2} - \frac{\varphi_k^2}{2} + \frac{\prt{\varphi-\varphi_k}^2}{2} } \dx.
		\end{align*}
		The same argumentation as in the proof of Lemma~\ref{AGG:exws:lemma:ex_time-discr} leads to
		\begin{align*}
			& \int_\Omega h\Str(\varphi_k,\D\uel) : \D\uel +  h\mj(\varphi_k) c_+^2 \abs{\Grad \omega}^2 + h\mr(\varphi_k) c_+^2 \omega^2 + h^2\beta\lambda_0^2 \dx
			+ \int_{\partial\Omega} h\gamma \abs{\uel_{\btau}}^2 \dsix \\
			&\quad + \int_\Omega \frac{\abs{\Grad\varphi}^2}{2} \dx 
			\leq \int_\Omega \frac{\varphi_k^2}{2} + \frac{\abs{\Grad\varphi_k}^2}{2}  + F_0(\varphi_k) + \abs{\kappa}\frac{\varphi_k^2}{2} \dx
			+ \biggabs{ \int_\Omega F(\varphi) \dx }
			\leq C_k,
		\end{align*}
		from where, incorporating $\abs{\varphi}\leq1$ a.e.~in $\Omega$, Korn's inequality for $\uel\in \rmH^1_\n(\Omega)$ (in the version~(3.5.11) in \cite{Abels2007}) and the positive lower bound on $\nu,\mj,\mr$, we derive
		\begin{align*}
			\norm{\uel}_{\rmH^1(\Omega)} + \norm{\lambda_0}_{\rmL^2(\Omega)} + \norm{\omega}_{\rmH^1(\Omega)} + \norm{\varphi}_{\rmH^1(\Omega)}
			\leq C_k.
		\end{align*}
		Using \eqref{LT:exws:eq:time-discr_sol_LS_d1}, this estimate further reveals $\norm{\partial E_m(\varphi)}_{\rmL^2(\Omega)} \leq C_k$. Consequently, we are in a position to control the $\rmH^2(\Omega)$-norms of both $\omega$ in light of \eqref{LT:exws:eq:omega_H2}, where now $g_3=-\vartheta\frac{\varphi-\varphi_k}{h}$, and of $\varphi$ due to \eqref{eq:subgrad_with:estimate_H2}. In summary, $\z=(\uel,\lambda_0, \omega,\varphi) \in \tilde X = \rmH^1_\n(\Omega) \times  \rmL^2_{(0)}(\Omega) \times \rmH^2_\Neum(\Omega) \times \rmH^{2-s}(\Omega)$, where $s\in(0,\frac14)$ is arbitrary, satisfies
		\begin{align*}
			\norm{\z}_{\tilde X}
			\leq C_k.
		\end{align*}
		Concluding $\norm{\f}_{\tilde Y}\leq C_k$ from the fact that $\mathcal{F}_k \colon \tilde X\to\tilde Y$ maps bounded sets into bounded sets finishes the proof of the Leray--Schauder condition. Hence, we have established that for each $k\in\N$, there exists a solution $\z=(\uel,\lambda_0, \omega,\varphi) \in X$ to the time-discrete system~\eqref{LT:exws:eq:time-discr_sol_a}--\eqref{LT:exws:eq:time-discr_sol_d1}.
	\end{proof}

	\subsection{Proof of Theorem \ref{thm:LT}} \label{LT:exws:subsec:proof_main}
	
	In the following, we prove our second main result, Theorem~\ref{thm:LT}, by passing to the limit in the time-discrete solutions from the previous subsection.
	
	\begin{proof}
		For given $N\in\N$, let $(\uel_{k+1},\lambda_{0,k+1}, \omega_{k+1},\varphi_{k+1})$, $k\in\N_0$, be successively chosen as a solution to the time-discrete system~\eqref{LT:exws:eq:time-discr_sol_a}--\eqref{LT:exws:eq:time-discr_sol_d1} with time step size $h=\frac1N$ and with initial value~$\varphi_0^N$, where the original initial datum $\varphi_0$ is approximated by functions $\varphi_0^N \in \rmH^2(\Omega)$ as demonstrated in the proof of Theorem~\ref{thm:AGG}. We also adopt the notation used therein. The existence of such solutions, which additionally satisfy the energy inequality~\eqref{LT:exws:eq:discr_energy_est}, follows from Lemma~\ref{LT:exws:lemma:discr_energy_est} and Lemma~\ref{LT:exws:lemma:ex_time-discr}.
		
		From \eqref{LT:exws:eq:time-discr_sol_a}--\eqref{LT:exws:eq:time-discr_sol_d1}, we deduce the following corresponding time-continuous system, where the equations~\eqref{LT:exws:eq:time-cont_sol_b} and \eqref{LT:exws:eq:time-cont_sol_c} below are obtained using the discrete integration by parts formula \eqref{AGG:exws:eq:discr_int_by_parts}. The system reads
		\begin{subequations}
		\begin{align}
			\label{LT:exws:eq:time-cont_sol_a}
			&\int_Q \bigprt{\rho_h^N}^{-1} \Str \bigprt{\varphi_h^N,\D\uel^N} : \Grad\bpsi + \Str \bigprt{\varphi_h^N,D\uel^N} \Grad \bigprt{\rho_h^N}^{-1} \cdot \bpsi
			- \beta\lambda_0^N \nabla\cdot \bpsi \dtx \nonumber\\
			&\quad+ \int_S \gamma \bigprt{\rho_h^N}^{-1} \uel_{\btau}^N \cdot \bpsi_{\btau} \dtsix
			= - \int_Q \bigprt{\rho_h^N}^{-1} \varphi_h^N\Grad \omega^N \cdot \bpsi \dtx
		\end{align}
		for all $\bpsi\in \rmC_0^\infty((0,\infty); \rmW^{1,r'}_\n(\Omega))$,
		\begin{align}
			\label{LT:exws:eq:time-cont_sol_b}
			\int_Q - \rho^N \partial_{t,h}^+ \xi + \rho_h^N\uel^N \cdot\Grad\xi + h\lambda_0^N \xi \dtx= 0
		\end{align}
		for all $\xi\in \rmC_0^\infty((0,\infty);\rmH^1(\Omega))$, provided that $h$ is sufficiently small, as well as
		\begin{align}
			\label{LT:exws:eq:time-cont_sol_c}
			&\int_Q - \varphi^N \partial_{t,h}^+\zeta - \varphi_h^N\uel^N \cdot\Grad\zeta \dtx \nonumber\\
			&= \int_Q -\mj\bigprt{\varphi_h^N} c_+^2 \Grad \omega^N \cdot\Grad\zeta - \mr\bigprt{\varphi_h^N} c_+^2 \omega^N \zeta  \dtx
		\end{align}
		for all $\zeta\in \rmC_0^\infty((0,\infty); \rmH^1(\Omega))$ if $h$ is small enough, and
		\begin{align}
			\label{LT:exws:eq:time-cont_sol_d1}
			P_0\bigprt{\omega^N} + \kappa\frac{\varphi^N+\varphi_h^N}{2} - \kappa m = P_0\bigprt{F_0'\bigprt{\varphi^N}} - \Lap\varphi^N + \alpha\lambda_0^N
			\quad\text{a.e.~in $(0,\infty)\times\Omega$}. \owntag{1}
		\end{align}
		\end{subequations}
		
		Moreover, analogously to the proof of Theorem~\ref{thm:AGG}, the time-discrete energy inequality~\eqref{LT:exws:eq:discr_energy_est} transforms into
		\begin{align}
			\label{LT:exws:eq:time-cont_E+}
			&E_\free \bigprt{\varphi^N(t)} \nonumber\\
			&\quad+ \int_{Q_{(s,t)}} \Str \bigprt{\varphi_h^N,\D\uel^N} : \D\uel^N + \mj\bigprt{\varphi_h^N} c_+^2 \bigabs{\Grad \omega^N}^2 + \mr\bigprt{\varphi_h^N} c_+^2 \bigprt{\omega^N}^2 
			 + h\beta \bigprt{\lambda_0^N}^2 \dtaux \nonumber\\
			&\quad+ \int_{S_{(s,t)}} \gamma \bigabs{\bigprt{\uel^N}_{\btau}}^2 \dtausix
			\leq E_\free \bigprt{\varphi^N(s)}
		\end{align}
		for all $0\leq s\leq t<\infty$ with $s,t\in h\N_0$. The bounds provided by this energy estimate, together with Korn's inequality for $\uel^N$ (version~(3.5.11) in \cite{Abels2007}) and the conservation property~$\mean{\varphi^N} \equiv m$, imply the following  weak (or weak-*, respectively) convergences as~$N\to\infty$ along non-relabeled subsequences:
		\begin{alignat*}{2}
			\uel^N &\rightharpoonup \uel &&\quad\text{in } \rmL^2(0,\infty;\rmH^1(\Omega)^d), \\
			\varphi^N &\overset{\ast}{\rightharpoonup} \varphi &&\quad\text{in } \rmL^\infty(0,\infty;\rmH^1(\Omega)), \\
			\omega^N &\rightharpoonup  \omega &&\quad\text{in } \rmL^2(0,\infty;\rmH^1(\Omega)).
		\end{alignat*}
	
		However, since the term containing $\lambda_0^N$ in \eqref{LT:exws:eq:time-cont_E+} degenerates as $N\to\infty$ due to its pre-factor $h=\frac1N$ a different approach in order to establish a bound on $\lambda_0^N$ is needed. Therefore, following the strategy from \cite{Abels2009g}, we select $\bpsi(t)\coloneqq \sigma(t)\Grad\phi$ as a test function in \eqref{LT:exws:eq:time-cont_sol_a} where~$\sigma\in \rmC_0^\infty(0,\infty)$ and $\phi\in \rmW^{2,r'}_\Neum(\Omega)$ with $r'$ being the dual exponent of $r\in (\frac{2d}{d+2},\frac{d}{d-1})$. We point out that this is an admissible test function since $\n\cdot\Grad\phi \vert_{\partial\Omega} = \partial_\n\phi \vert_{\partial\Omega} = 0$. This leads to the equation
		\begin{align*}
			&\int_0^\infty \int_\Omega \beta\lambda_0^N \Lap\phi \dx \;\sigma \dt 
			= \int_0^\infty \int_\Omega \bigprt{\rho_h^N}^{-1} \Str \bigprt{\varphi_h^N,\D\uel^N} : \Grad^2\phi \dx \; \sigma \dt \\
			&\quad + \int_0^\infty \int_\Omega \Bigprt{ \Str \bigprt{\varphi_h^N,\D\uel^N} \Grad \bigprt{\rho_h^N}^{-1} 
				+ \bigprt{\rho_h^N}^{-1} \varphi_h^N\Grad \omega^N }
			\cdot \Grad\phi \dx \; \sigma \dt \\
			&\quad + \int_0^\infty \int_{\partial\Omega} \gamma \bigprt{\rho_h^N}^{-1} \uel_{\btau}^N \cdot (\Grad\phi)_{\btau} \dsix \; \sigma \dt
		\end{align*}
		for all $\sigma\in \rmC_0^\infty(0,\infty)$ and $\phi\in \rmW^{2,r'}_\Neum(\Omega)$. Thus, $\beta\lambda_0^N$ is a very weak solution of the Neumann--Laplace equation, see Proposition~\ref{prelim:prop:very_weak_sol}, whose right-hand side $f^N(t)$ may be estimated by
		\begin{align*}
			&\norm{f^N(t)}_{\prt{\rmW^{2,r'}_\Neum(\Omega)}'}
			\leq C \Bigprt{ \bignorm{\bigprt{\rho_h^N}^{-1} \Str \bigprt{\varphi_h^N,\D\uel^N}}_{\rmL^2(\Omega)}  \\
			&\quad+ \bignorm{\Str \bigprt{\varphi_h^N,\D\uel^N} \Grad \bigprt{\rho_h^N}^{-1} 
			+ \bigprt{\rho_h^N}^{-1} \varphi_h^N\Grad \omega^N}_{\rmL^1(\Omega)} 
			+ \bignorm{\gamma \bigprt{\rho_h^N}^{-1} \uel_{\btau}^N}_{\rmL^2(\partial\Omega)} }
		\end{align*}
		for almost all $t\in (0,\infty)$ in view of $\Grad\phi\in \rmW^{1,r'}(\Omega) \hookrightarrow \rmC^0(\overline{\Omega})$. As a consequence, according to estimate~\eqref{eq:Neum-Lap:estimate}, there exists a constant~$C_r>0$ such that
		\begin{align*}
			\norm{\lambda_0^N}_{\rmL^2(0,\infty; \rmL^r(\Omega))}
			\leq C_r \norm{f^N}_{\rmL^2 \bigprt{0,\infty;\prt{\rmW^{2,r'}_\Neum(\Omega)}'}}
			\leq C
		\end{align*}
		due to the uniform bounds derived above, from where we eventually conclude that
		\begin{align*}
			\lambda_0^N \rightharpoonup \lambda_0 \quad\text{in } \rmL^2(0,\infty;\rmL^r(\Omega))
		\end{align*}
		along a subsequence.
		
		Next, in order to pass to the limit in the non-linearities, we need to show strong convergence of $\varphi^N$. By the same reasoning as in the proof of Theorem~\ref{thm:AGG}, we first prove 
		\begin{align}
			\label{LT:exws:eq:strong_conv_phi_L2}
			\varphi^N \to \varphi \quad\text{in } \rmL^2(0,\infty;\rmL^2(\Omega)) \text{ for all } 0<T<\infty
		\end{align}
		as well as $\varphi\in \BCw([0,\infty);\rmH^1(\Omega))$ and $\varphi\vert_{t=0}=\varphi_0$. Furthermore, we study the estimate from Lemma~\ref{LT:exws:lemma:estimates} for the piecewise constant functions to observe
		\begin{align*}
			&\bignorm{\varphi^N}_{\rmL^2(0,T;\rmW^{2,r}(\Omega))}
			+ \bignorm{F_0'\bigprt{\varphi^N}}_{\rmL^2(0,T;\rmL^r(\Omega))} \\
			&\leq C \Bigprt{\bignorm{\lambda_0^N}_{\rmL^2(0,T;\rmL^r(\Omega))} + \bignorm{\Grad \omega^N}_{\rmL^2(0,T;\rmL^2(\Omega))}
			+ \bignorm{\varphi^N}_{\rmL^2(0,T;\rmH^1(\Omega))} + \bignorm{\varphi_h^N}_{\rmL^2(0,T;\rmH^1(\Omega))} + 1 } \\
			&\quad\Bigprt{\bignorm{\varphi^N}_{\rmL^\infty(0,T;\rmH^1(\Omega))} + 1} 
		\end{align*}
		for all $0<T<\infty$. Since the right-hand side is bounded uniformly in $N$, we obtain for a subsequence that $\varphi^N\rightharpoonup\varphi$ in $\rmL^2(0,T;\rmW^{2,r}(\Omega))$ as well as $F_0'(\varphi^N) \rightharpoonup F^*$ in~$\rmL^2(0,T;\rmL^r(\Omega))$ for all $0<T<\infty$ with some limit function $F^* \in \rmL^2_\uloc([0,\infty);\rmL^r(\Omega))$. Finally, in view of the compact embedding~$\rmW^{2,r}(\Omega) \hookrightarrow \rmH^1(\Omega)$, we even establish 
		\begin{align*}
			\varphi^N\to\varphi \quad\text{in } \rmL^2(0,T;\rmH^1(\Omega)) \text{ for all } 0<T<\infty.
		\end{align*}
		By means of the proven convergence results, we are now in a position to pass to the limit in \eqref{LT:exws:eq:time-cont_sol_a}--\eqref{LT:exws:eq:time-cont_sol_c} to establish \eqref{LT:exws:eq:weak_sol_a}--\eqref{LT:exws:eq:weak_sol_c}.
	
		Continuing with the limit passage in equation~\eqref{LT:exws:eq:time-cont_sol_d1}, we first note that $-\Lap\varphi^N \rightharpoonup -\Lap\varphi$ in~$\rmL^2(0,T;\rmL^r(\Omega))$ for all $0<T<\infty$ since $-\Lap\colon \rmW^{2,r}(\Omega)\to \rmL^r(\Omega)$ is a linear bounded operator. Already knowing that all other terms in \eqref{LT:exws:eq:time-cont_sol_d1} converge weakly in $\rmL^2(0,T;\rmL^r(\Omega))$ for all $0<T<\infty$ as well it only remains to verify $F^*=F'(\varphi)$ a.e.~in $(0,\infty)\times\Omega$ in order to pass to the correct limit in $P_0(F_0'(\varphi^N))$. To prove this, we follow an approach in \cite{AbelsBosiaGrasselli2014}. To this end, let $0<T<\infty$ and $l\in\N$ be fixed, and let $0<t<T$ be arbitrary. From the strong convergence \eqref{LT:exws:eq:strong_conv_phi_L2}, it follows $\varphi^N(t) \to \varphi(t)$ a.e.~in $\Omega$ along a subsequence. Additionally, Egorov's theorem implies the existence of a set $\Omega_l\subset\Omega$ with $\abs{\Omega_l} \geq \abs{\Omega}-\frac1{2l}$ on which~$\varphi^N(t)\to\varphi(t)$ uniformly. Defining
		\begin{align*}
			M_{\delta,N} \coloneqq \big\{ x\in\Omega : \bigabs{\varphi^N(t,x)} > 1-\delta \big\}
		\end{align*}
		we see that $\abs{M_{\delta,N}}$ is decreasing in $\delta>0$ for all $N\in\N$.  Moreover, it holds
		\begin{align*}
			c_\delta \coloneqq \inf_{\abs{r}\geq1-\delta} \abs{F'(r)} \to \infty \quad\text{as } \delta\to0
		\end{align*}
		due to the assumption~\eqref{ass:pot}. This convergence and the uniform (with respect to $N$) bound on $F'(\varphi^N(t))$ in $\rmL^2(\Omega)$ entail, in light of Chebyshev's inequality, that
		\begin{align*}
			\abs{M_{\delta,N}} \leq \frac{1}{c_\delta^2} \int_\Omega \bigabs{F'\bigprt{\varphi^N(t)}}^2 \dx 
			\to0 \quad\text{as } \delta\to0 \quad\text{uniformly in }N\in\N.
		\end{align*}
		Hence, there exists $\delta=\delta(l)$ independent of $N$ such that it holds
		\begin{align*}
			\Bigabs{ \big\{ x\in\Omega : \bigabs{\varphi^N(t,x)} > 1-\delta \big\} } \leq \frac{1}{2l} \quad\text{for all } N\in\N.
		\end{align*}
		Moreover, the uniform convergence above yields some $\bar{N}\in\N$ such that $\abs{\varphi^N(t)-\varphi(t)} < \frac{\delta}{4}$~for all~$N\geq\bar{N}$ on $\Omega_l$. With
		\begin{align*}
			\Omega_l' \coloneqq \Omega_l \cap \big\{ x\in\Omega : \bigabs{\varphi^{\bar{N}}(t,x)} \leq 1-\delta \big\}
		\end{align*}
		we derive that $\abs{\Omega_l'} \geq \abs{\Omega} - \frac1l$ as well as $\abs{\varphi^N(t,x)} < 1-\frac{\delta}{2}$ for all $N\geq\bar{N}$ and for all~$x\in\Omega_l'$. The continuity of $F'$ on $(-1,1)$ shows that $F'(\varphi^N(t))\to F'(\varphi(t))$ uniformly on $\Omega_l'$, which results in the convergence $F'(\varphi^N)\to F'(\varphi)$ a.e.~in $(0,\infty)\times\Omega$ since $l$ and $T$ were arbitrary. Eventually, due to the uniqueness of both weak and pointwise limits, we conclude that $F^*$ and $F'(\varphi)$ coincide a.e.~in $(0,\infty)\times\Omega$.
		
		Altogether, it is now possible to perform the limit passage in equation \eqref{LT:exws:eq:time-cont_sol_d1}. Thus, following the same procedure as in Remark~\ref{LT:exws:rmk:mean-free_eq_for_chem_pot} on a time-continuous level, in particular setting~$-\alpha\bar\lambda = \mean{F_0'(\varphi)} - \bar{\omega} - \kappa m$ a.e.~in $(0,\infty)$, where $\bar\omega\coloneqq \prt{\int_\Omega \mr(\varphi) \dx}^{-1} \int_\Omega \mr(\varphi)P_0(\omega) \dx$, we establish the validity of equation~\eqref{LT:exws:eq:weak_sol_d}. 
		
		The remaining proof of the energy inequality \eqref{LT:exws:eq:weak_sol_energy_ineq} follows the same steps as in Subsection~\ref{AGG:exws:subsec:energy_ineq}.
	\end{proof}

	\noindent\textbf{Acknowledgment.}
	The third author was funded by the Deutsche Forschungsgemeinschaft (DFG, German Research Foundation) Research Training Group~2339 ``Interfaces, Complex Structures, and Singular Limits in Continuum Mechanics -- Analysis and Numerics''. The support is gratefully acknowledged.

	\bibliographystyle{abbrv}
	\bibliography{Paper_AGG_LT.bib}

\begin{thebibliography}{10}

\bibitem{Abels2007}
H.~Abels.
\newblock Diffuse interface models for two-phase flows of viscous,
  incompressible fluids.
\newblock Habilitation Thesis, 2007.

\bibitem{Abels2009g}
H.~Abels.
\newblock Existence of weak solutions for a diffuse interface model for
  viscous, incompressible fluids with general densities.
\newblock {\em Communications in Mathematical Physics}, 289(1):45--73, 2009.

\bibitem{Abels2009m}
H.~Abels.
\newblock On a diffuse interface model for two-phase flows of viscous,
  incompressible fluids with matched densities.
\newblock {\em Archive for Rational Mechanics and Analysis}, 194(2):463--506,
  2009.

\bibitem{Abels2012}
H.~Abels.
\newblock Strong well-posedness of a diffuse interface model for a viscous,
  quasi-incompressible two-phase flow.
\newblock {\em SIAM Journal on Mathematical Analysis}, 44(1):316--340, 2012.

\bibitem{AbelsBosiaGrasselli2014}
H.~Abels, S.~Bosia, and M.~Grasselli.
\newblock Cahn–{H}illiard equation with nonlocal singular free energies.
\newblock {\em Annali di Matematica Pura ed Applicata}, 2014.

\bibitem{AbelsDepnerGarcke2012}
H.~Abels, D.~Depner, and H.~Garcke.
\newblock Existence of weak solutions for a diffuse interface model for
  two-phase flows of incompressible fluids with different densities.
\newblock {\em Journal of Mathematical Fluid Mechanics}, 15(3):453--480, 2012.

\bibitem{AbelsDepnerGarcke2013}
H.~Abels, D.~Depner, and H.~Garcke.
\newblock On an incompressible {N}avier–{S}tokes/{C}ahn–{H}illiard system
  with degenerate mobility.
\newblock {\em Annales de l’Institut Henri Poincaré C, Analyse non
  linéaire}, 30(6):1175--1190, 2013.

\bibitem{AbelsGarckeGiorgini2023}
H.~Abels, H.~Garcke, and A.~Giorgini.
\newblock Global regularity and asymptotic stabilization for the incompressible
  {N}avier–{S}tokes-{C}ahn–{H}illiard model with unmatched densities.
\newblock {\em Mathematische Annalen}, 389(2):1267--1321, 2023.

\bibitem{AbelsGarckeGruen2012}
H.~Abels, H.~Garcke, and G.~Gr\"{u}n.
\newblock Thermodynamically consistent, frame indifferent diffuse interface
  models for incompressible two-phase flows with different densities.
\newblock {\em Mathematical Models and Methods in Applied Sciences}, 2012.

\bibitem{AbelsWeber2020}
H.~Abels and J.~Weber.
\newblock Local well-posedness of a quasi-incompressible two-phase flow.
\newblock {\em Journal of Evolution Equations}, 21(3):3477--3502, 2020.

\bibitem{AbelsWilke2007}
H.~Abels and M.~Wilke.
\newblock Convergence to equilibrium for the {C}ahn–{H}illiard equation with
  a logarithmic free energy.
\newblock {\em Nonlinear Analysis: Theory, Methods \& Applications},
  67(11):3176--3193, 2007.

\bibitem{Aki2014}
G.~L. Aki, W.~Dreyer, J.~Giesselmann, and C.~Kraus.
\newblock A quasi-incompressible diffuse interface model with phase transition.
\newblock {\em Mathematical Models and Methods in Applied Sciences},
  24(05):827--861, 2014.

\bibitem{Amann1995}
H.~Amann.
\newblock {\em Linear and Quasilinear Parabolic Problems}, volume~1.
\newblock Birkhäuser Basel, 1995.

\bibitem{Blesgen1999}
T.~Blesgen.
\newblock A generalization of the {N}avier-{S}tokes equations to two-phase
  flows.
\newblock {\em Journal of Physics {D}: Applied Physics}, 32(10):1119--1123,
  1999.

\bibitem{Boyer1999}
F.~Boyer.
\newblock Mathematical study of multi‐phase flow under shear through order
  parameter formulation.
\newblock {\em Asymptotic Analysis}, 20:175--212, 1999.

\bibitem{Boyer2002}
F.~Boyer.
\newblock A theoretical and numerical model for the study of incompressible
  mixture flows.
\newblock {\em Computers \& Fluids}, 31(1):41--68, 2002.

\bibitem{Ding2007}
H.~Ding, P.~D.~M. Spelt, and C.~Shu.
\newblock Diffuse interface model for incompressible two-phase flows with large
  density ratios.
\newblock {\em Journal of Computational Physics}, 226(2):2078--2095, 2007.

\bibitem{GalGrasselli2010}
C.~G. Gal and M.~Grasselli.
\newblock Longtime behavior for a model of homogeneous incompressible two-phase
  flows.
\newblock {\em Discrete \& Continuous Dynamical Systems - A}, 28(1):1--39,
  2010.

\bibitem{GalGrasselliWu2019}
C.~G. Gal, M.~Grasselli, and H.~Wu.
\newblock Global weak solutions to a diffuse interface model for incompressible
  two-phase flows with moving contact lines and different densities.
\newblock {\em Archive for Rational Mechanics and Analysis}, 234(1):1--56,
  2019.

\bibitem{GalLvWu2024}
C.~G. Gal, {Maoyin Lv}, and H.~Wu.
\newblock On a thermodynamically consistent diffuse interface model for
  two-phase incompressible flows with non-matched densities: Dynamics of moving
  contact lines, surface diffusion, and mass transfer.
\newblock 2024.

\bibitem{Giorgini2021}
A.~Giorgini.
\newblock Well-posedness of the two-dimensional {A}bels–{G}arcke–{G}r\"{u}n
  model for two-phase flows with unmatched densities.
\newblock {\em Calculus of Variations and Partial Differential Equations},
  60(3), 2021.

\bibitem{Giorgini2022}
A.~Giorgini.
\newblock Existence and stability of strong solutions to the
  {A}bels–{G}arcke–{G}r\"{u}n model in three dimensions.
\newblock {\em Interfaces and Free Boundaries, Mathematical Analysis,
  Computation and Applications}, 24(4):565--608, 2022.

\bibitem{GiorginiGrasselliWu2022}
A.~Giorgini, M.~Grasselli, and H.~Wu.
\newblock On the mass-conserving {A}llen--{C}ahn approximation for
  incompressible binary fluids.
\newblock {\em Journal of Functional Analysis}, 283(9):109631, 2022.

\bibitem{GiorginiKnopf2023}
A.~Giorgini and P.~Knopf.
\newblock Two-phase flows with bulk–surface interaction: Thermodynamically
  consistent {N}avier–{S}tokes–{C}ahn–{H}illiard models with dynamic
  boundary conditions.
\newblock {\em Journal of Mathematical Fluid Mechanics}, 25(3), 2023.

\bibitem{GiorginiMiranvilleTemam2019}
A.~Giorgini, A.~Miranville, and R.~Temam.
\newblock Uniqueness and regularity for the
  {N}avier--{S}tokes-{C}ahn--{H}illiard system.
\newblock {\em SIAM Journal on Mathematical Analysis}, 51(3):2535--2574, 2019.

\bibitem{Gurtin1996}
M.~E. Gurtin, D.~Polignone, and J.~Viñals.
\newblock Two-phase binary fluids and immiscible fluids described by an order
  parameter.
\newblock {\em Mathematical Models and Methods in Applied Sciences},
  06(06):815--831, 1996.

\bibitem{Heida2011}
M.~Heida, J.~Málek, and K.~R. Rajagopal.
\newblock On the development and generalizations of {C}ahn–{H}illiard
  equations within a thermodynamic framework.
\newblock {\em Zeitschrift für angewandte Mathematik und Physik},
  63(1):145--169, 2011.

\bibitem{Hohenberg1977}
P.~C. Hohenberg and B.~I. Halperin.
\newblock Theory of dynamic critical phenomena.
\newblock {\em Reviews of Modern Physics}, 49(3):435--479, 1977.

\bibitem{JiangLiLiu2017}
J.~Jiang, Y.~Li, and C.~Liu.
\newblock Two-phase incompressible flows with variable density: An energetic
  variational approach.
\newblock {\em Discrete \& Continuous Dynamical Systems - A}, 37(6):3243--3284,
  2017.

\bibitem{Kotschote2012}
M.~Kotschote.
\newblock Strong solutions of the {N}avier–{S}tokes equations for a
  compressible fluid of {A}llen–{C}ahn type.
\newblock {\em Archive for Rational Mechanics and Analysis}, 206(2):489--514,
  2012.

\bibitem{LowengrubTruskinovsky1998}
J.~Lowengrub and L.~Truskinovsky.
\newblock Quasi–incompressible {C}ahn–{H}illiard fluids and topological
  transitions.
\newblock {\em Proceedings of the Royal Society of London. Series A:
  Mathematical, Physical and Engineering Sciences}, 454(1978):2617--2654, 1998.

\bibitem{ShokrpourRoudbari2018}
M.~Shokrpour~Roudbari, G.~Şimşek, E.~H. van Brummelen, and K.~G. van~der Zee.
\newblock Diffuse-interface two-phase flow models with different densities: A
  new quasi-incompressible form and a linear energy-stable method.
\newblock {\em Mathematical Models and Methods in Applied Sciences},
  28(04):733--770, 2018.

\bibitem{Showalter1997}
R.~E. Showalter.
\newblock {\em Monotone Operators in Banach Space and Nonlinear Partial
  Differential Equations}.
\newblock American Mathematical Society, 1997.

\bibitem{tenEikelder2023}
M.~F.~P. ten Eikelder, K.~G. van~der Zee, I.~Akkerman, and D.~Schillinger.
\newblock A unified framework for {N}avier–{S}tokes {C}ahn–{H}illiard
  models with non-matching densities.
\newblock {\em Mathematical Models and Methods in Applied Sciences},
  33(01):175--221, 2023.

\bibitem{tenEikelder2024}
M.~F.~P. ten Eikelder, K.~G. van~der Zee, and D.~Schillinger.
\newblock Thermodynamically consistent diffuse-interface mixture models of
  incompressible multicomponent fluids.
\newblock {\em Journal of Fluid Mechanics}, 990, 2024.

\bibitem{Weber2016}
J.~T. Weber.
\newblock Analysis of diffuse interface models for two-phase flows with and
  without surfactants.
\newblock PhD Thesis, 2016.

\bibitem{Zeidler1992}
E.~Zeidler.
\newblock {\em Nonlinear Functional Analysis and its Applications I}.
\newblock Springer, 1992.

\end{thebibliography}
	
\end{document}